\documentclass{amsart}
\usepackage{verbatim}
\usepackage[textsize=scriptsize,disable]{todonotes}
\usepackage{etoolbox}
\usepackage{longtable}
\newtoggle{draft}
\newtoggle{final}
\usepackage{multicol}

\usepackage{tikz}
\usetikzlibrary{matrix,arrows}
\usepackage{tikz-cd}
\usepackage{appendix}

\usepackage{etoolbox}
\iftoggle{draft} {
\usepackage[margin=1.5in]{geometry}
}{ 
\usepackage[margin=1in]{geometry}
\setlength{\marginparwidth}{0.75in}
}
\usepackage{amsmath}
\usepackage{amssymb}
\usepackage{amsthm}
\usepackage{amscd}
\usepackage{xcolor}
\usepackage{enumerate}
\usepackage[pdfusetitle,unicode,hidelinks,pagebackref=true]{hyperref}
\usepackage{bbm}
\usepackage{etoolbox}

\usepackage[utf8]{inputenc}
\usepackage[T1]{fontenc}

\newcommand{\tomemail}{\href{mailto:tom.bachmann@zoho.com}{tom.bachmann@zoho.com}}

\newtheorem{proposition}{Proposition}
\newtheorem{thmdefn}[proposition]{Theorem-Definition}
\newtheorem{corollary}[proposition]{Corollary}
\newtheorem{lemma}[proposition]{Lemma}
\newtheorem{theorem}[proposition]{Theorem}

\newtheorem*{conjecture*}{Conjecture}
\newtheorem*{theorem*}{Theorem}
\newtheorem*{corollary*}{Corollary}
\newtheorem*{proposition*}{Proposition}
\newtheorem*{lemma*}{Lemma}
\theoremstyle{definition}
\newtheorem{definition}[proposition]{Definition}

\newtheorem*{definition*}{Definition}
\newtheorem*{construction*}{Construction}
\newtheorem{remark}[proposition]{Remark}
\newtheorem*{remark*}{Remark}
\newtheorem{question}[proposition]{Question}
\newtheorem{example}[proposition]{Example}
\newtheorem*{example*}{Example}

\newcommand{\id}{\operatorname{id}}
\newcommand{\Z}{\mathbb{Z}}
\def\C{\mathbb C}
\newcommand{\N}{\mathbb{N}}

\newcommand{\F}{\mathbb{F}}

\let\scr=\mathcal
\let\bb=\mathbb
\newcommand{\Gm}{{\mathbb{G}_m}}
\newcommand{\Gmp}[1]{{\mathbb{G}_m^{\wedge #1}}}
\def\A{\bb A}

\def\R{\bb R}
\newcommand{\1}{\mathbbm{1}}

\newcommand{\SH}{\mathcal{SH}}

\DeclareMathOperator*{\colim}{colim}

\let\lim=\relax
\DeclareMathOperator*{\lim}{lim}
\def\Map{\mathrm{Map}}
\def\map{\mathrm{map}}
\def\CAlg{\mathrm{CAlg}}
\def\Alg{\mathrm{Alg}}

\def\PSh{\mathcal{P}}
\def\Shv{\mathcal{S}\mathrm{hv}}
\def\Ind{\mathrm{Ind}}
\def\Pro{\mathrm{Pro}}
\def\Span{\mathrm{Span}}

\def\Spc{\mathcal{S}\mathrm{pc}{}}
\def\Fin{\cat F\mathrm{in}}
\def\Fun{\mathrm{Fun}}

\newcommand{\wequi}{\simeq}
\newcommand{\Mod}{\text{-}\mathcal{M}\mathrm{od}}
\def\adj{\leftrightarrows}

\DeclareRobustCommand{\ul}{\underline}
\newcommand{\heart}{{c\heartsuit}}

\newcommand{\Hom}{\operatorname{Hom}}

\def\op{\mathrm{op}}

\let\cat=\mathrm
\def\Sm{{\cat{S}\mathrm{m}}}

\newcommand{\et}{{\acute{e}t}}

\def\ph{\mathord-}

\definecolor{violet}{rgb}{0.56, 0.0, 1.0}

\definecolor{lanse}{rgb}{0.02, 0.0, 1.0}

\definecolor{fense}{rgb}{0.9, 0.3, 0.3}

\definecolor{chartreuse}{rgb}{0.5, 1.0, 0.0}

\numberwithin{proposition}{section}
\numberwithin{equation}{subsection}

\newcommand{\MGL}{\mathrm{MGL}}
\newcommand{\BPGL}{\mathrm{BPGL}}
\newcommand{\MU}{\mathrm{MU}}
\newcommand{\BP}{\mathrm{BP}}
\newcommand{\Th}{\mathrm{Th}}
\newcommand{\pure}{\mathrm{pure}}
\newcommand{\tate}{\mathrm{tate}}
\newcommand{\rig}{\mathrm{lisse}}
\newcommand{\Ext}{\mathrm{Ext}}
\newcommand{\Tot}{\mathrm{Tot}}
\newcommand{\comp}{\wedge}
\newcommand{\PM}{\mathrm{PM}}
\newcommand{\cell}{\mathrm{cell}}
\newcommand{\AT}{\mathrm{wcell}}
\newcommand{\CoMod}{\text{-}\mathcal{C}\mathrm{o}\mathcal{M}\mathrm{od}}
\def\imap{\ul{\mathrm{map}}}
\newcommand{\triplearrows}{\begin{smallmatrix} \to \\ \to \\ \to \end{smallmatrix} }
\newcommand{\CB}{\mathrm{CB}^\bullet}
\def\Mack{\cat M\mathrm{ack}}

\newcommand{\Stable}{\mathcal{H}\mathrm{ov}}
\newcommand{\Thick}{\mathrm{Thick}}
\newcommand{\Ab}{\mathrm{Ab}}
\newcommand{\Set}{\mathrm{Set}}

\iftoggle{final} {
\newcommand{\NB}[1]{}
}{ 
\newcommand{\NB}[1]{\todo[color=gray!40]{#1}}
}
\newcommand{\tom}[1]{\todo[color=green!40]{#1}}

\newcommand{\wcell}{\mathrm{wcell}}

\title{The Chow $t$-structure on the $\infty$-category of motivic spectra}

\author{Tom Bachmann}
\address{Mathematisches Institut, LMU Munich, Munich, Germany}
\email{\tomemail}

\author{Hana Jia Kong}
\address{School of Mathematics, Institute for Advanced Study, Princeton, NJ 08540, USA}
\email{hanajk@ias.edu}

\author{Guozhen Wang}
\address{Shanghai Center for Mathematical Sciences, Fudan University, Shanghai, China, 200433}
\email{wangguozhen@fudan.edu.cn}

\author{Zhouli Xu}
\address{Department of Mathematics, UC San Diego, La Jolla, CA 92093, USA}
\email{xuzhouli@ucsd.edu}

\begin{document}

\maketitle

\begin{abstract}
We define the Chow $t$-structure on the $\infty$-category of motivic spectra $\SH(k)$ over an arbitrary base field $k$. We identify the heart of this $t$-structure $\SH(k)^{c\heartsuit}$ when the exponential characteristic of $k$ is inverted. Restricting to the cellular subcategory, we identify the Chow heart $\SH(k)^{\textup{cell}, c\heartsuit}$ as the category of even graded $\MU_{2*}\MU$-comodules. Furthermore, we show that the $\infty$-category of modules over the Chow truncated sphere spectrum $\1_{c=0}$ is algebraic.

Our results generalize the ones in Gheorghe--Wang--Xu \cite{GWX} in three aspects: To integral results; To all base fields other than just $\C$; To the entire $\infty$-category of motivic spectra $\SH(k)$, rather than a subcategory containing only certain cellular objects.

We also discuss a strategy for computing motivic stable homotopy groups of ($p$-completed) spheres over an arbitrary base field $k$ using the Postnikov--Whitehead tower associated to the Chow $t$-structure and the motivic Adams spectral sequences over $k$.
\end{abstract}

\tableofcontents

\section{Introduction}

\subsection{Overview}
Motivic stable homotopy theory, introduced by Voevodsky and Morel, is
a subject that has successfully applied abstract homotopy theory to solve problems in number theory and algebraic geometry.

Recent work by Gheorghe--Wang--Xu \cite{GWX} has reversed this information flow---using motivic stable homotopy theory, the computation of classical stable homotopy groups of spheres, a fundamental problem in topology, has been extended into a much larger range \cite{IWX, IWX2}.  

Gheorghe--Wang--Xu \cite{GWX} defined and studied a subcategory of ($p$-completed) cellular objects of the $\infty$-category of motivic spectra over the complex numbers $\SH(\C)$, and identify its heart as the abelian category of $p$-complete, even graded comodules over the Hopf algebroid $\MU_{2*}\MU$, which is equivalent to the abelian category quasi-coherent sheaves on the moduli stack of formal groups over $\mathbb{Z}_p$-algebras, and is central to chromatic homotopy theory.

A natural question to ask is whether we can generalize this work
\begin{itemize}
\item to the entire motivic stable homotopy category, rather than just a subcategory consisting of certain cellular objects over the $p$-completed sphere spectrum,
\item to an integral rather than $p$-complete result,
\item and to other base fields than just $\C$.
\end{itemize}

In this paper, we achieve all three goals by defining a $t$-structure, which we call the \emph{Chow $t$-structure}, on the motivic stable homotopy category $\SH(k)$ over any base field $k$.
Its non-negative part $\SH(k)_{c\geq 0}$ is generated under colimits and extensions by Thom spectra $\Th(\xi)$ associated to $K$-theory points $\xi\in K(X)$ (equivalently formal differences of vector bundles $\xi = [V_1] - [V_2]$) on smooth and proper schemes $X$ (see Definition \ref{def:chow-t}).
We denote the truncation functors by $E \mapsto E_{c=0}, E_{c \ge 0}$, and so on.

One of our key theorems is that the motivic bigraded homotopy groups of a homotopy object with respect to the Chow $t$-structure can be expressed by $\Ext$ groups over $\MU_{2*}\MU$ of its $\MGL$ homology (Theorem~\ref{thm:intro-main}).
Using the theory of comonadic descent, we give a concrete algebraic description of the Chow heart $\SH(k)^\heart$ (Theorem \ref{thm:intro-heart}). We show that the $\infty$-category of modules over $\1_{c=0}$ (where $\1$ denotes the motivic sphere spectrum) is algebraic (Theorem \ref{intro:thm_algebraicity_of_1_module}).
Restricting to the subcategory of cellular modules over $\1_{c=0}$, we get an equivalence with Hovey's stable category $\Stable(\MU_{2*}\MU)$ (Theorem \ref{thm:intro-cell-heart}).

This equivalence between stable $\infty$-categories results in an isomorphism between the motivic Adams spectral sequence and the algebraic Novikov spectral sequence, i.e., the Adams spectral sequence in the stable category $\Stable(\MU_{2*}\MU).$
Since the latter spectral sequence is purely algebraic, its computation is much more accessible and can be done by computer programs.
Therefore, this isomorphism allows us to fully compute the motivic Adams spectral sequence for all spectra in $\SH(k)^\heart$ within any reasonable range.
We can then adapt the same methodology as in \cite{GWX, IWX, IWX2}: using naturality of the motivic Adams spectral sequences, we can obtain new information on differentials and extensions from spectra in the Chow heart $\SH(k)^\heart$ to the sphere spectrum, and therefore compute motivic stable stems over any base field $k$ in a reasonable range.

\subsection{Main result}
From now on, we denote by $e$ the exponential characteristic of the field $k$.
We prove the following key theorem about the Chow $t$-structure:
\begin{theorem}[see Theorem \ref{thm:main-computation}] \label{thm:intro-main}
Let $E \in \SH(k)[1/e]$. 
Then there is a canonical isomorphism \[ \pi_{2w-s, w} E_{c=i} \cong \Ext^{s,2w}_{\MU_{2*}\MU}(\MU_{2*},\MGL_{2*+i,*}E). \] 
\end{theorem}
Here, on the right hand side, $s$ is the homological degree, and $2w$ is the internal degree.

Over $k = \C$ for $E$ the sphere spectrum $\1$, Theorem~\ref{thm:intro-main} tells us 
\[ \pi_{2w-s, w} \1_{c=0} \cong \Ext^{s,2w}_{\MU_{2*}\MU}(\MU_{2*},\MU_{2*}) \]
as an integral statement. This is in contrast to Isaksen's result \cite[Proposition 6.2.5]{Isa}
\[ \pi_{2w-s, w} \1_p^\comp/\tau \cong \Ext^{s,2w}_{\BP_{2*}\BP}(\BP_{2*},\BP_{2*}) \]
at each prime $p$, where $\tau: \Sigma^{0,-1} \1_p^\comp \to \1_p^\comp$ is the endomorphism of \cite{gheorghe2017motivic, Isa}.

We have the following easy corollary of Theorem~\ref{thm:intro-main}, reproving \cite[Theorem 3.13]{gheorghe2017motivic}.

\begin{corollary} \label{cor:intro-C}
Over $\C$, both $\1_{c=0}$ and $\1_p^\comp/\tau$ are $\scr E_\infty$-rings.
\end{corollary}


\begin{remark}
Another easy corollary of Theorem~\ref{thm:intro-main} is that smashing with $\MGL$ detects Chow-$\infty$-connectivity (Corollary \ref{cor:detect-infinity-connective}). 
\end{remark}

\subsection{Reconstruction theorems}

It is shown in \cite[Corollary~1.2]{GWX} that the category $\1_p^\comp/\tau\Mod$ of cellular modules over $\1_p^\comp/\tau$ is purely algebraic\footnote{It is known that over the complex numbers, the $p$-completion and the $\textup{H}\Z/p$-nilpotent completion for the sphere are equivalent.}:
 under the $t$-structure defined in \cite{GWX}, the heart  
is equivalent to $\MU_{2*}\MU_p^\comp\CoMod$;
 and the entire module category can be identified with Hovey's \emph{stable category of comodules} $\Stable(\MU_{2*}\MU_p^\comp)$ (see \cite{Hovey} for a definition of this category). Other versions of this result over $\C$ can be found in \cite{Krause, pstrkagowski2018synthetic, mmf}.
 We upgrade this to an integral result over arbitrary fields, using $\1_{c=0}$ as a replacement for $\1_p^\comp/\tau$ (which makes sense by the discussion after Theorem~\ref{thm:intro-main}).

\subsubsection{Description of the Chow heart} The category $\SH(k)^\heart[1/e]$ can be described as a category of presheaves over the category of pure $\MGL$-motives, together with some extra data.

\begin{definition}
	The category of pure $\MGL$-motives, denoted by $\PM_\MGL(k)$, is the smallest idempotent complete additive subcategory of $\MGL_{c=0}\Mod$ containing the object $X\{i\}:=(\Sigma^{2i,i}X_+\wedge \MGL)_{c=0}$ for   each $i\in \Z$ and smooth proper variety $X$.
	\end{definition}
	
	When it is clear from the context, we abuse notation and denote $X\{0\}$ by $X$.
	
It turns out that $\PM_\MGL(k)$ is an additive ordinary $1$-category, and 
the mapping set $[X, Y\{*\}]_{\PM_\MGL(k)}$ is equivalent to $ \MGL^{2*,*}(Y\times X)$. This definition of pure $\MGL$-motives coincides with the definition in \cite{nenashev2006oriented} by taking the idempotent completion of the category of $\MGL$-correspondences.


The category $\PM_\MGL(k)$ is naturally enriched in $\MU_{2*}\MU$-comodules. 
There is thus a classical notion of enriched presheaves on $\PM_\MGL(k)$ (see e.g. \cite[\S3.5]{riehl2014categorical}).
\begin{theorem}[see \S\ref{subsec:chow-heart} and Remark \ref{rem:concise_description}]\label{thm:intro-heart}
The functor sending $F\in \SH(k)^\heart[1/e]$ to the presheaf on $\PM_\MGL(k)$ given by $F_*(X) = [\Sigma^{2*,*} X_+, F\wedge \MGL]$ 
induces an equivalence of categories between $\SH(k)^\heart[1/e]$ and the category of enriched presheaves on $\PM_\MGL(k)$ (with values in $\MU_{2*}\MU$-comodules).
\end{theorem}

\begin{remark}
Explicitly, the structure of enriched presheaves on $\PM_\MGL(k)$ requires the following compatibility: for every graded $\MGL$-correspondence $\alpha: X \to Y$, the following diagram commutes
\begin{equation*}
\label{intro_diag_compatibility}
\begin{CD}
F(Y)_* @>{\alpha^*}>> F(X)_* \\
@V{\Delta_{F,Y}}VV  @V{\Delta_{F,X}}VV \\
\MU_{2*}\MU \otimes_{\MU_{2*}} F(Y)_* @>{\Delta(\alpha)^*}>> \MU_{2*}\MU \otimes_{\MU_{2*}} F(X)_*.
\end{CD}
\end{equation*}
Here $\Delta(\alpha)$ is the effect of comultiplication. 
We give the full explicit description in \S\ref{subsec:chow-heart}.

\end{remark}

\begin{remark}
P. Sechin is studying a category of ``Landweber-equivariant Grothendieck motives'' over a field $k$ \cite{sechin-landweber-equiv}.
One may show that this category embeds fully faithfully into $\SH(k)^\heart$.
\end{remark}

\subsubsection{Algebraicity of $\1_{c=0}\Mod[1/e]$}
By \cite[Lemma 29]{bachmann-tambara}, $\1_{c=0}\Mod[1/e]^\heart\simeq \SH(k)[1/e]^\heart$. 
The Chow heart $\SH(k)[1/e]^\heart$ is an abelian category. 
In general, for a compactly generated abelian category $A$ viewed as the heart of its derived category, we consider the category $\Stable(A):=\Ind(\Thick(A^{\omega}))$ where
$A^{\omega}$ denotes the subcategory of the compact objects (see \S\ref{subsub:compact-gen}). 
We prove the module category $\1_{c=0}\Mod[1/e]$ is algebraic.

\begin{theorem}[see Proposition \ref{prop:comonadic-reconstruction}]
\label{intro:thm_algebraicity_of_1_module}
There is a symmetric monoidal equivalence
 $$\1_{c=0}\Mod[1/e]\simeq \Stable(\1_{c=0}\Mod[1/e]^\heart).$$
\end{theorem}


\subsubsection{Results in the cellular case}

In the cellular case, the results take the following simpler form.

\begin{theorem}[see Corollary \ref{cor:wcellular_heart_ordinary_cellular}]\hfill
\label{thm:intro-cell-heart}
\begin{enumerate}
\item There is an equivalence $$\SH(k)[1/e]^{\cell, \heart} \wequi \MU_{2*}\MU\CoMod[1/e].$$
\item The cellular subcategory is equivalent to Hovey’s stable category of comodules 
$$\1_{c=0}\Mod^{\cell}[1/e] \wequi \Stable(\MU_{2*}\MU)[1/e].$$ 
\end{enumerate}
\end{theorem}

\begin{remark}
It comes as a surprise that Theorem \ref{thm:intro-cell-heart} does not depend on the base field $k$.
In fact, let $l/k$ be a field extension.
The base change functor $\SH(k)^\cell \to \SH(l)^\cell$ is right-$t$-exact (for the Chow $t$-structures) and induces \[ \SH(k)[1/e]^{\cell, \heart} \wequi \MU_{2*}\MU\CoMod[1/e] \wequi \SH(l)[1/e]^{\cell, \heart}. \]
This does \emph{not} imply that base change is conservative on cellular spectra which are Chow bounded below, since the Chow $t$-structure is not left complete (e.g. all $\eta$-periodic spectra are Chow $\infty$-connective, by Proposition \ref{prop:eta-periodic}).
\end{remark}

In more general cases involving field extensions, we have the following results.

\begin{definition}[see Definition \ref{def:wcellular}]\hfill
\label{def:intro-wcellular}

	Let $W$ be a set of smooth proper schemes over $k$ that contains $Spec (k)$ and is closed under finite products. Define the 
	\emph{$W$-cellular category}, denoted by $\SH(k)^\wcell$ to be the subcategory of $\SH(k)$ generated under taking colimits and desuspensions by objects of the form $\Th(\xi)$ for $\xi\in K(X)$ and $X\in W$.
\end{definition}

We also define the Chow $t$-structure on the $W$-cellular category.

\begin{theorem}[see Corollary \ref{cor:wcellular_heart_absolute} and \ref{cor:wcellular_heart_fin_Galois_ext}]\hfill
\label{cor:intro_wcellular_ext}

Let $G$ and $W$ be as in one of the following situations.
\begin{enumerate}
	\item Let $l/k$ be a finite Galois extension with Galois group $G$. Let $W$ be $\{Spec (l') ~~\vert ~~ l'$ is a subextension of $l/k \}$. 
	\item Let $W$ to be $\{Spec (l)\vert  ~~ l/k \textit{ is a finite separable extension} \}$ and let $G = Gal(k)$ be the absolute Galois group.
\end{enumerate}
We have 
\begin{align*} 
\SH(k)[1/e]^{\AT,\heart} & \wequi \ul{\MU_{2*}\MU}\CoMod[1/e], \\
\1_{c=0}\Mod[1/e]^\AT & \wequi \Stable(\ul{\MU_{2*}\MU})[1/e], \\ 
\MGL_{c=0}\Mod[1/e]^{\AT, \heart} & \wequi \ul{\MU_{2*}}\Mod[1/e]. 
\end{align*}
Here $\ul{\MU_{2*}\MU}$ and $\ul{\MU_{2*}}$ denote the corresponding constant Mackey functors, and $\CoMod$, $\Stable$ and $\Mod$ are performed relative to the category of $G$-Mackey functors.

\end{theorem}

\begin{remark}
	In \cite{burklund2020galois}, Burklund, Hahn and Senger prove a similar result in the special case when the field extension is $\C/\R$ with the Galois group $C_2$.
\end{remark}



\subsection{Proof strategy of Theorem \ref{thm:intro-heart}}

We obtain the description of the Chow heart using the theory of comonadic descent; to show it applies, we use results that follow from Theorem \ref{thm:intro-main}. In the discussion below we omit $[1/e]$ from the notation.
More precisely, for a motivic commutative ring spectrum $A\in \CAlg(\SH(k))$, there is an induced Chow $t$-structure on the category $A\Mod$ (just let the non-negative part be generated by free $A$-modules of the form $A \wedge \Th(\xi)$, for $X$ smooth proper and $\xi \in K(X)$). 
We focus on the induced $t$-structures on $\1_{c=0}\Mod$ and also $\MGL_{c=0}\Mod$.
The free-forgetful adjunction $$\1_{c=0}\Mod\leftrightarrows \MGL_{c=0}\Mod$$ defines a comonad $C$ over $\MGL_{c=0}\Mod$.
We show that when restricted to subcategories of suitably bounded objects, the adjunction is comonadic.
So $\1_{c=0}\Mod^\heart$ is equivalent to the category of $C^\heart$-comodules in $\MGL_{c=0}\Mod^\heart$, where $C^\heart$ denotes the restriction of $C$.

Using spectral Morita theory, we can identify the category $\MGL_{c=0}\Mod$ as a presheaf category.
Under this equivalence, we give an explicit description of the comonad $C^\heart$.
In the cellular case this turns out to be exactly the comonad describing $\MU_{2*}\MU$-comodules over $\MU_{2*}$-modules.
This description together with the equivalences $\SH(k)^\heart\simeq \1_{c=0}\Mod^\heart\simeq C^\heart\CoMod$ results in Theorem \ref{thm:intro-heart}. 

\begin{remark}[work of Bondarko]\label{rmk:workofBondarko}
The Chow $t$-structure on $\MGL_{c=0}\Mod$ is adjacent to a weight structure in the sense of Bondarko \NB{I'm really not sure which of Bondarko's many papers to cite...} \cite{bondarko2007weight}.
Bondarko's work then yields an explicit description of $\MGL_{c=0}\Mod^\heart \wequi \MGL\Mod^\heart$ (see \cite[Theorem 4.4.2(4)]{bondarko2007weight}) and a similar description of all of $\MGL_{c=0}\Mod$ is easily obtained.
This proof is equivalent to what we summarized above as ``spectral Morita theory''.
\end{remark}

\begin{remark}
The proofs show that the scalar extension functor $\MGL\Mod \to \MGL_{c=0}\Mod$ identifies with Bondarko's \emph{weight complex functor}.
\end{remark}

\subsection{Proof strategy for Theorem \ref{thm:intro-main}} 
Before explaining some further applications of Theorem \ref{thm:intro-main}, let us sketch its proof.
Our approach uses another $t$-structure on $\SH(k)$, the homotopy $t$-structure.
Its non-negative part $\SH(k)_{\geq 0}$ is generated under colimits and extensions by $\{\Sigma^\infty_+X\wedge \Gmp{n} \mid n\in \Z, X \in \Sm_k\}$ (see \S\ref{subsec:colimit-diagram}). 
For all $d\geq 0$, the intersections $$I^d:=\Sigma^d\SH(k)_{\geq 0}\cap \SH(k)_{c\geq 0}$$ define a sequence of further $t$-structures.
We write $\tau^d_{= 0}$ for the $0$-th truncation functor with respect to $I^d.$ 
It turns out that these $t$-structures form a direct system with the Chow $t$-structure as the colimit (Lemma \ref{lemm:chow-approx}). 

One key property of these $t$-structures is the following vanishing result.
\begin{proposition}[see Proposition \ref{lemm:key}] \label{lemma:intro-key}
Let $E\in \SH(k)[1/e]$.
\begin{enumerate}
\item $\pi_{{*,*}} \tau^d_{= 0} E$ is concentrated in Chow degrees $\le 0$, and
\item $\MGL_{{2*,*}} \tau^d_{= 0} E$ equals $ \MGL_{{2*,*}} E$ and vanishes for other bidegrees.
\end{enumerate}
\end{proposition}

As a consequence of Proposition \ref{lemma:intro-key}, the Adams--Novikov spectral sequence for $\tau_{=0}^d E$ collapses and converges to $\pi_{*,*}\tau_{=0}^d E$. 
In other words, we have canonical isomorphisms $$\pi_{2w-s, w}(\tau^d_{=0}E)\cong \Ext^{s,2w}_{\MU_*\MU}(\MU_*,\MGL_{2*,*}E)$$
Since we can approximate $E_{c=0}$ as a colimit by $\tau^d_{=0}E$, the same result holds for the Chow $t$-structure.
This is exactly the statement of Theorem \ref{thm:intro-main}.


\subsection{Towards Computing Motivic Stable Homotopy Groups of Spheres} 

We now turn to computational applications of Theorems \ref{thm:intro-main} and \ref{thm:intro-cell-heart}. 

Computations of motivic stable homotopy groups of spheres is currently a very popular subject. Most results are in the following two categories, using mainly two different computational tools. 

The first type of results gives a description for certain bidegrees of motivic stable stems over an arbitrary base field $k$, in terms of number theoretical invariants of the base field $k$. Two major results in this direction are 
\begin{itemize}

\item Morel's work \cite{Morel0line} on the $0$-line, i.e., $\displaystyle \bigoplus_s \pi_{s,s}\1$, in terms of Milnor K-theory,

\item R{\"o}ndigs--Spitzweck--{\O}stv{\ae}r's more recent work \cite{rondigs2019first} on the $1$-line, i.e., $\displaystyle \bigoplus_s \pi_{s+1,s}\1$, in terms of hermitian and Milnor $K$-groups of the base field $k$ whose characteristic is not $2$.

\end{itemize}
These computations use the slice spectral sequence and the results are for the integral sphere $\1$.

The second type of results determine the motivic stable stems over a specific field $k$ in a reasonable but much larger range of bidegrees. The computational tool is the motivic Adams spectral sequence. These results provide $p$-primary information one prime at a time, where $p$ is different from the exponential characteristic of the base field $k$. Most notably, 
\begin{itemize}

\item over $\C$, Isaksen--Wang--Xu \cite{IWX} have computed $\displaystyle \bigoplus_w \pi_{s, w} \1^\comp_2$ for $s \leq 90$,

\item over $\mathbb{R}$, Belmont--Isaksen \cite{BelIsa} have computed $\displaystyle \bigoplus_w \pi_{s, w} \1^\comp_2$ for $s-w \leq 11$,

\item over finite fields, Wilson--\O stv\ae r \cite{WilsonOst} have computed $\pi_{s,0} \1^\comp_2$ for $s \leq 18$.
\end{itemize}

Our goal is to pursue the second direction using the motivic Adams spectral sequence, and to use the Chow $t$-structure to do computations over an arbitrary field $k$, so the $p$-primary computations can be achieved in a large and reasonable range.

The computation of the $E_2$-page of the motivic Adams spectral sequence depends on understanding the motivic homology of a point and the motivic Steenrod algebra action on it. The mod $p$ homology of a point is computed in terms of Milnor $K$-theory by work of Voevodsky \cite{voevodsky2003motivic, voevodsky2011motivic} when the base field contains a primitive $p$-th root of unity, and has characteristic coprime to $p$. The motivic Steenrod algebra is computed by work of Voevodsky \cite{voevodsky2003reduced} and Hoyois--Kelly--\O stv\ae r \cite{Hoyois2017}. 
With the knowledge of the structure of the motivic Steenrod algebra, the motivic Adams $E_2$-page can be fully determined by a purely homological algebra computation. 

The next step, which is also the hardest part, is to determine the motivic Adams differentials. According to the Mahowald Uncertainty Principle \cite{Goerss, Xutalk}, any Adams type spectral sequence that converges to the stable homotopy groups of spheres has infinitely many nonzero differentials after the $E_2$-page, and any method that is used to compute these nonzero differentials leaves infinitely many unsolved by that method. The spirit of the Mahowald Uncertainty Principle is that we need to combine all known methods together to push forward these computations.

Nevertheless, classically and over the base field $\C$, Gheorghe--Wang--Xu \cite{GWX} has developed a new method to compute the motivic Adams differentials, which is so far the most efficient one. The most crucial part of this method is to obtain the following isomorphism of two spectral sequences, using an equivalence of two stable $\infty$-categories.

\begin{theorem}[Theorem~1.3 of \cite{GWX}] \label{intro GWX iso of ss}
For each prime p, there is an isomorphism of spectral sequences, between the $\textup{H}\Z/p$-based motivic Adams spectral sequence for $\1^\comp_p/\tau$, and the algebraic Novikov spectral sequence for $\BP_{2*}$.  
\end{theorem}
Since the algebraic Novikov spectral sequence is a purely algebraic spectral sequence, and can be computed in a large range by an automated computer program \cite{IWX, IWX2}, Theorem~\ref{intro GWX iso of ss} allows us to obtain nonzero Adams differentials for free.


Using Theorems~\ref{intro:thm_algebraicity_of_1_module} and \ref{thm:intro-cell-heart}, we generalize Theorem~\ref{intro GWX iso of ss} to the following Theorem~\ref{thm:intro-iso-of-ss} that works for any spectrum $F$ in $\SH(k)^\heart$ and over an arbitrary base field $k$. This allows us to obtain nonzero Adams differentials for various spectra for free in a similar fashion to the case over $\C$ but over an arbitrary base field $k$.


In the discussion below, we again implicitly invert the exponential characteristic $e$ of the base field throughout to ease notation.

Let $\textup{H}\Z/p$ denote the motivic Eilenberg-MacLane spectrum. Under the equivalence in Theorem~\ref{thm:intro-cell-heart}
$$\1_{c=0}\Mod^{\cell}\simeq \Stable(\MU_{2*}\MU),$$
the spectrum $\textup{H}\Z/p \wedge \1_{c=0}$ corresponds to the $\MU_{2*}\MU$-comodule $H = \MU_{2*}\MU/(p,a_1,a_2,\dots)$. The equivalence in Theorem \ref{thm:intro-heart} gives an equivalence between the $\textup{H}\Z/p$-based motivic Adams spectral sequences, and the $H$-based Adams type spectral sequences in the category of stable $\MU_{2*}\MU$-comodules.

\begin{theorem}[see Theorem \ref{thm: iso of ss_cellular} and \ref{thm: iso of ss}]\label{thm:intro-iso-of-ss}
Let $F \in \SH(k)^\heart$. 
Let $M = \MGL_{2*,*}F$ be the associated $\MU_{2*}\MU$-comodule.
Then the trigraded motivic Adams spectral sequence for $F$ based on $\textup{H}\Z/p$ is isomorphic (with all higher and multiplicative structure) to the trigraded 
algebraic Novikov spectral sequence based on $H$ for $M$.
\end{theorem}
See Definition \ref{def: iso of ss} for the algebraic Novikov spectral sequence for an object $M$ in $\Stable(\MU_{2*}\MU)$ based on a commutative monoid $H$.


Now we discuss how to compute the motivic homotopy groups of the $p$-completion of a general motivic spectrum $X$ with the property that $X \simeq X_{c \geq 0}$ using Theorem~\ref{thm:intro-iso-of-ss}. Note that the motivic sphere spectrum is such an example. Consider its Postnikov--Whitehead tower with respect to the Chow $t$-structure.

\begin{center}
\begin{tikzcd}
& \arrow[d] &\\
& {X}_{c \geq 2} \arrow[d] \arrow[r] & {X}_{c=2}   \\
& {X}_{c \geq 1} \arrow[d] \arrow[r] & {X}_{c=1}  \\
X \arrow[r, equal] & X_{c \geq 0} \arrow[r] & {X}_{c=0}
\end{tikzcd}
\end{center}
For every term $X_{c=n}$ in the Postnikov--Whitehead tower, consider the motivic Adams spectral sequences based on $\textup{H}\Z/p$ and the induced maps among them. This gives us the following tower of motivic Adams spectral sequences.
\begin{center}
\begin{tikzcd}
& \arrow[d] &\\
& \mathbf{motASS}({X}_{c \geq 2}) \arrow[d] \arrow[r] & \mathbf{motASS}({X}_{c=2}) \arrow[r, equal] & \mathbf{algNSS}(\MGL_{{2*+2, *}}X)   \\
& \mathbf{motASS}({X}_{c \geq 1}) \arrow[d] \arrow[r] & \mathbf{motASS}({X}_{c=1}) \arrow[r, equal] & \mathbf{algNSS}(\MGL_{{2*+1,*}}X)  \\
\mathbf{motASS}({X}) \arrow[r, equal] & \mathbf{motASS}({X_{c \geq 0}}) \arrow[r] & \mathbf{motASS}({X}_{c=0}) \arrow[r, equal] & \mathbf{algNSS}(\MGL_{{2*, *}}X)
\end{tikzcd}
\end{center}
For every comodule $\MGL_{{2*+n,*}}X$, its algebraic Novikov spectral sequence can be computed in a large range by a computer program. This gives many nonzero differentials in the motivic Adams spectral sequence for ${X}_{c=n}$. We may then pullback these Adams differentials to the motivic Adams spectral sequences for ${X}_{c \geq n}$, and push forward them to the motivic Adams spectral sequences for ${X}$.

The strategy for computing motivic Adams differentials for $X$ can be summarized in the following steps. See \S\ref{strategy} for more details.

\begin{enumerate}

\item Compute $\Ext_A^{*,*,*}(\textup{H}\Z/p_{*,*}, \textup{H}\Z/p_{*,*}X)$ over the $k$-motivic Steenrod algebra by a computer program. These $\Ext$-groups consist of the $E_2$-page of the motivic Adams spectral sequence for ${X}$ based on $\textup{H}\Z/p$.

\item Compute by a computer program the algebraic Novikov spectral sequence based on $\BP_{2*}\BP/I$ for the $\BP_{2*}\BP$-comodules $\BPGL_{2*+n,*}X$ for every $n$ in a reasonable range. 

\item Identify the $k$-motivic Adams spectral sequence based on $\textup{H}\Z/p$ for each $X_{c=n}$, with the algebraic Novikov spectral sequence based on $\BP_{2*}\BP/I$ for $\BPGL_{2*+n,*}X$, for $n \geq 0$. This computation includes an identification of the abutments with the multiplicative (and higher) structures.

\item Compute the mod $p$ motivic homology of $X_{c \geq n}$ using the universal coefficient spectral sequence (see Propositions~7.7 and 7.10 of \cite{dugger2005motivic}).
$$\bigoplus_{k \geq n}\textup{Tor}_{*,*}^{\BP_{2*}}(\BPGL_{2*+k,*}X, \ \Z/p) \implies \textup{H}\Z/p_{*,*} X_{c \geq n}  .$$

\item Compute by a computer program the $E_2$-pages of the motivic Adams spectral sequence for ${X_{c \geq n}}$ based on $\textup{H}\Z/p$, using the computation of $\textup{H}\Z/p_{*,*}X_{c \geq n}$ in step (4).

\item Pull back motivic Adams differentials for $X_{c=n}$ to motivic Adams differentials for $X_{c \geq n}$, and then push forward to motivic Adams differentials for $X$.

\item Deduce additional Adams differentials for $X$ with a variety of ad hoc arguments. The most important methods are Toda bracket shuffles and comparison to known results in the $\C$-motivic Adams spectral sequence.

\end{enumerate}

The inputs of our strategy are $\textup{H}\Z/p_{*,*}X$ and $\MGL_{*,*}X$ (or $\BPGL_{*,*}X$). In the case of the sphere spectrum $X = \1$, it is more plausible to know $\pi_{*,*} \BPGL_p^\comp$ (over some fields) rather than $\BPGL_{*,*}$. In this case, we modify our strategy by using pro-objects $(\1/p^n)_n$ as a substitute. See \S\ref{CTS ASS} for more details.





As an example, we can show that over finite fields, our method gives proofs of previously undetermined Adams differentials in \cite{WilsonOst}. See \S\ref{finite fields} for more details.



\subsection{Organization}
The rest of this article is organized as follows. In \S2, we define the Chow $t$-structure and discuss some of its basic properties. In \S3, we discuss the connection between the Chow $t$-structure and the homotopy $t$-structure. In particular, we prove one of our main theorems, Theorem~\ref{thm:intro-main} as Theorem~\ref{thm:main-computation}. In \S4, we use the theory of comonadic descent, the Barr--Beck--Lurie theorem and spectral Morita theory to give a concrete description of the Chow heart and to prove the algebraicity of the stable $\infty$-category $\1_{c=0}\Mod[1/e]$. In particular, we prove Theorems~\ref{thm:intro-heart}, \ref{intro:thm_algebraicity_of_1_module}, \ref{thm:intro-cell-heart} and \ref{cor:intro_wcellular_ext}. In \S5, we prove Theorem~\ref{thm:intro-iso-of-ss} and discuss applications toward computing motivic stable homotopy groups of spheres. In Appendices A, B and C, we recall well-known results regarding a cell structure of $\MGL$, a vanishing result of $\MGL$, and motivic cohomology of a point.

\subsection{Conventions and notation}
Following e.g. \cite[Definition 2.1]{nikolaus2017presentably}, we call a symmetric monoidal $\infty$-category $\scr C$ \emph{presentably symmetric monoidal} if $\scr C$ is presentable and the tensor product preserves colimits in each variable separately.

Given a morphism of schemes $f: X \to Y$, we denote by $f^*: \SH(Y) \to \SH(X)$ the induced base change functor, and by $f_*: \SH(X) \to \SH(Y)$ its right adjoint.
If $f$ is smooth, then $f^*$ has a left adjoint which we denote by $f_\#$.

Here is a table of further notation that we employ.

\NB{there is a tendency for this table to float away}
\NB{order}
\begin{longtable}{lll}
$\imap(\ph,\ph)$ & internal mapping object \\
$\Spc$ & $\infty$-category of spaces \\
$\SH$ & $\infty$-category of spectra \\
$\Ab$ & $1$-category of abelian groups \\
$\PSh_\Sigma, \PSh_\SH, \PSh_\Ab$ & product-preserving presheaves & \S\ref{subsub:spectral-morita} \\
$\CB(A)$ & cobar construction & \cite[Construction 2.7]{mathew2017nilpotence} \\
$A\Mod$ & $\infty$-category of modules over monoid $A$ & \cite[Definition 4.2.1.13]{HA} \\
$C\CoMod$ & $\infty$-category of comodules over comonad $C$ & \cite[Definition 4.2.1.13]{HA} \\
$\Stable(C)$ & stable category of comodules & \S\ref{subsub:compact-gen} \\
$\Alg(\ph), \CAlg(\ph)$ & (commutative) algebra objects  & \\
$\MU$ & complex cobordism spectrum \\
$\Map(\ph, \ph)$ & mapping space in an $\infty$-category \\
$\map(\ph, \ph)$ & mapping spectrum in a stable $\infty$-category \\
$[\ph,\ph]$ & homotopy classes of maps, i.e. $\pi_0\Map(\ph,\ph)$ \\

$\scr O, \scr O_X$ & trivial line bundle \\
$n\scr O, \scr O^n$ & trivial vector bundle of rank $n$ \\

$\Sm_S$ & category of smooth, finite type $S$-schemes \\
$\Spc(S)$ & $\infty$-category of motivic spaces over $S$ & \cite[\S2.2]{bachmann-norms} \\
$\SH(S)$ & $\infty$-category of motivic spectra over $S$ & \cite[\S4.1]{bachmann-norms} \\
$S^{p,q}$ & bigraded sphere $S^{p-q} \wedge \Gmp{q}$ \\
$\Sigma^{p,q}$ & functor of suspension by $S^{p,q}$ \\
$\Th(\xi) = \Th_S(\xi) \in \SH(S)$ & Thom spectrum of $\xi \in K(X), X \in \Sm_S$ & \cite[\S16.2]{bachmann-norms} \\
$\MGL$ & algebraic cobordism motivic spectrum & \cite[\S16.2]{bachmann-norms} \\
$E_a^\comp$ & $a$-completion of a (motivic) spectrum & \\
$E_\MGL^\comp$ & $\MGL$-nilpotent completion & \cite[\S7.1]{mantovani2021localizations} \\
$L_\MGL E$ & $\MGL$-localization & \cite[\S2.4]{mantovani2021localizations} \\
$\SH(S)_{\ge 0}$, $\SH(S)_{\le 0}$ & homotopy $t$-structure & \cite[\S B]{bachmann-norms} \\
$\pi_{p,q} E$ & bigraded homotopy groups \\
$\ul{\pi}_{p,q} E$ & bigraded homotopy sheaves \\

$\SH(S)_{c \ge 0}$, $\SH(S)_{c \le 0}$ & Chow $t$-structure & Definition \ref{def:chow-t} \\
$\SH(S)^\pure$ & & Definition \ref{def:SH-pure} \\
$\SH(S)^\rig$  & & Definition \ref{def:SH-rig} \\
$\SH(S)^{\pure\tate\ge d}$  & & \S\ref{app:cells} \\
$c(p,q)$ & Chow degree & Definition \ref{def:chow-degree} \\
\end{longtable}

\subsection{Acknowledgements}
We thank Marc Levine and the University of Duisburg-Essen for the conference SPP 1786 Jahrestagung in 2019, where part of this collaboration was started. We thank Mark Behrens, Robert Burklund, Lars Hesselholt, Dan Isaksen, Fangzhou Jin, Haynes Miller, 
Markus Spitzweck,
Glen Wilson,
Paul Arne {\O}stv{\ae}r
for conversations regarding this project. This material is based upon work supported by the National Science Foundation under Grant No. DMS-1926686. 
The third author was partially supported by grant NSFC-11801082 and the Shanghai Rising-Star Program under agreement No. 20QA1401600. The fourth author was partially supported by the National Science Foundation under Grant No. DMS-2105462.

\section{Elementary properties}
Let $S$ be a scheme.

\begin{definition} \label{def:chow-t}
Denote by $\SH(S)_{c \ge 0}$ the subcategory generated under colimits and extensions by motivic Thom spectra $\Th(\xi)$ for $X$ smooth and proper\NB{think about whether this could say projective} over $S$ and $\xi \in K(X)$ arbitrary.
This is the non-negative part of a $t$-structure on $\SH(S)$ \cite[Proposition 1.4.4.11(2)]{HA} called the \emph{Chow $t$-structure}.
We denote the non-positive part by $\SH(S)_{c \le 0}$, the heart by $\SH(S)^\heart$ and write \[ E \mapsto E_{c \ge 0}, E_{c \le 0}, E_{c = 0} \] for the truncation functors.
We also put $\SH(S)_{c \ge n} = \Sigma^n \SH(S)_{c \ge 0}$ and define $\SH(S)_{c \le n}, E_{c\ge n}$ etc. similarly.
\end{definition}

\begin{example}\label{ex:spheres_in_gezero}
	By definition, for any $n\in \mathbb{Z}$,  the bigraded sphere $S^{2n,n}\simeq \Th(n\scr O)$ is in $\SH(S)_{c \ge 0}$. 
\end{example}

\begin{remark}
\label{rem:chow-t-Amod}
For any ring spectrum $A \in \CAlg(\SH(S))$, we can define a Chow $t$-structure on the category of $A$-modules. Its non-negative part is the subcategory generated under colimits and extensions by the free $A$-modules $A\wedge \Th(\xi)$ for $X$ smooth and proper over $S$ and $\xi\in K(X)$.

\end{remark}

\subsection{First properties}
\begin{lemma} \label{lemm:t-structure-nonpos} \NB{surely this must be standard?!?}
Let $\scr C$ be a stable presentable category, $\scr S \subset \scr C$ a set of objects and denote by $\scr C_{\ge 0} \subset \scr C$ the subcategory generated under colimits and extensions by $\scr S$.
Then $E \in \scr C_{\le 0}$ if and only if $[\Sigma^i S, E] = 0$ for all $S \in \scr S$ and $i > 0$.
\end{lemma}
\begin{proof}
By definition $E \in \scr C_{\le 0}$ if and only if $\Map(\Sigma X, E) = 0$ for all $X \in \scr C_{\ge 0}$ (see e.g. \cite[Remark 1.2.1.3]{HA}).
Necessity is thus clear.
To prove sufficiency, let $E \in \scr C$ satisfy the stated condition, and write $\scr C_E \subset \scr C$ for the subcategory of all $F \in \scr C$ such that $\Map(\Sigma F, E) = 0$.
Then $\scr S \subset \scr C_E$ by assumption, and $\scr C_E$ is closed under colimits and extensions.
Thus $\scr C_{\ge 0} \subset \scr C_E$, which concludes the proof.
\end{proof}

Thus given $E \in \SH(S)$ we have $E \in \SH(S)_{c \le 0}$ if and only if $[\Sigma^{i} \Th(\xi), E] = 0$ for all $i>0$ (and $\xi$ a $K$-theory point on a smooth proper scheme $X$ over $S$).
We will use this characterization without further comment from now on.

The category $\SH(S)$ is symmetric monoidal, so carries a notion of strong duals (see e.g. \cite{dold1983duality}).
The following is a straightforward generalization of
\cite[Theorem 2.2]{Riou04}. 
\begin{lemma} \NB{surely there must be a reference?!?}
\label{lemma:dualizability}
For $f: X \to S$ smooth and proper and $\xi \in K(X)$, the spectrum $\Th_S(\xi) \in \SH(S)$ is strongly dualizable with dual $\Th_S(-T_X - \xi)$, where $T_X$ denote the tangent bundle.
More generally, if $E \in \SH(X)$ is strongly dualizable with dual $DE$, then $f_\#(E)$ is strongly dualizable with dual $f_\#(\Th(-T_X)\wedge DE )$.
\end{lemma}
\begin{proof}
Since the Thom spectrum functor turns sums into smash products (essentially by construction, see e.g. \cite[\S16.2]{bachmann-norms}), $\Th(\xi) \in \SH(X)$ is strongly dualizable (in fact invertible) with dual $\Th(-\xi)$.
By definition we have $\Th_S = f_\# \circ \Th$; it thus suffices to prove the second statement.
This follows from the natural equivalences \begin{gather*} \imap(f_\#(E), \ph) \wequi f_* \imap(E, f^*(\ph)) \wequi f_*(DE \wedge f^*(\ph)) \\ \wequi f_\#(\Th(-T_X) \wedge DE \wedge f^*(\ph)) \wequi f_\#(\Th(-T_X) \wedge DE) \wedge \ph; \end{gather*} here we have used the relationship between $f^!$ and $f^*$ and $f_*$ \cite[Theorem 6.18(1,2)]{hoyois-equivariant}\NB{i.e. ambidexterity} and the projection formula \cite[Theorem 1.1(7)]{hoyois-equivariant}.
\end{proof}

\begin{remark} \label{rmk:gens}
\begin{enumerate}
\item Note that the spectra $\Th_S(\xi)$ are compact, being strongly dualizable in a category with compact unit (see e.g. \cite[Proposition 6.4(3)]{hoyois-equivariant} for the latter).
\NB{This is true more generally if $X$ is only smooth (not necessarily proper), since $f_\#$ always preserves compact objects...}

\item The dualizability of the generators is important for many of the properties of the Chow $t$-structure.
\end{enumerate}
\end{remark}

\subsection{Compatibility with filtered colimits}
\begin{proposition} \label{prop:filtered}
The Chow $t$-structure is compatible with filtered colimits: $\SH(S)_{c \le 0}$ is closed under filtered colimits (and so are $\SH(S)_{c \ge 0}$ and $\SH(S)_{c =0}$).
\end{proposition}
\begin{proof}
The first claim is immediate from Remark \ref{rmk:gens}(1) (stating that $\Th(\xi)$ is compact) and Lemma \ref{lemm:t-structure-nonpos} (stating that $E \in \SH(S)_{c \le 0}$ if and only if $\Map(\Sigma \Th(\xi), E) \wequi 0$ for certain $X, \xi$).
By definition $\SH(S)_{c \ge 0}$ is even closed under all colimits, and finally $\SH(S)_{c =0}$ is closed under filtered colimits, being the intersection of $\SH(S)_{c \le 0}$ and $\SH(S)_{c \ge 0}$.
\end{proof}

\begin{corollary} \label{cor:filtered-trunctation}
The functor $E \mapsto E_{c \le 0}: \SH(S) \to \SH(S)$ preserves filtered colimits.
\end{corollary}
\begin{proof}
The functor factors as $\SH(S) \to \SH(S)_{c \le 0} \hookrightarrow \SH(S)$; here the first functor is a left adjoint by definition, and the second one preserves filtered colimits by Proposition \ref{prop:filtered}.
\end{proof}

\subsection{Interaction with tensor products}

\begin{proposition} \label{prop:tensor} \NB{In fact compatible with the normed structure...}
The Chow $t$-structure is compatible with the symmetric monoidal structure, i.e. the non-negative part of the Chow $t$-structure is closed under taking smash products: \[ \SH(S)_{c \ge 0} \wedge \SH(S)_{c \ge 0} \subset \SH(S)_{c \ge 0}. \]
\end{proposition}
\begin{proof}
Since $\SH(S)$ is presentably symmetric monoidal, it suffices to show that the generators of $\SH(S)_{c \ge 0}$ are closed under binary smash products.
This is clear since $\Th(\xi) \wedge \Th(\zeta) \wequi \Th(\xi \times \zeta)$, and smooth proper $S$-schemes are closed under fiber products over $S$.
\end{proof}

\begin{definition} \label{def:SH-pure}
Denote by $\SH(S)^\pure \subset \SH(S)$ the smallest subcategory that is closed under filtered colimits and extensions and contains $\Th(\xi)$ for any $K$-theory point $\xi$ on a smooth proper scheme $X$ over $S$.
\end{definition}
Note that by Definition~\ref{def:chow-t}, the subcategory $\SH(S)_{c \ge 0}$ is defined to be closed under arbitrary colimits. Thus $\SH(S)^\pure \subset \SH(S)_{c \ge 0}$.
By Proposition \ref{prop:tensor}, we have that $\SH(S)^\pure \wedge \SH(S)_{c \ge 0} \subset \SH(S)_{c \ge 0}$.
Interestingly, the same holds for the non-negative part of the Chow $t$-structure, essentially by duality (see Lemma \ref{lemma:dualizability}).

\begin{proposition} \label{prop:pure}
We have \[ \SH(S)^\pure \wedge \SH(S)_{c \le 0} \subset \SH(S)_{c \le 0}. \]
\end{proposition}
\begin{proof}
Let $\scr C \subset \SH(S)$ denote the subcategory of all spectra $Z$ such that $$Z \wedge \SH(S)_{c \le 0} \subset \SH(S)_{c \le 0}.$$
Since $\SH(S)_{c \le 0}$ is closed under filtered colimits (by Proposition \ref{prop:filtered}) and extensions (as is the non-positive part of any $t$-structure), so is the category $\scr C$. Hence to show that $\SH(S)^\pure \subset \scr C$, it suffices to show that $\Th(\zeta)\in \scr C$ for any $K$-theory point $\zeta$ on a smooth proper scheme $X$. 

Let $E$ be a spectrum in $\SH(S)_{c \le 0}$ and let $\xi$ be a $K$-theory point on a smooth proper scheme $Y$.
For any $i>0$ we find by duality (see Remark \ref{rmk:gens}(2)) that
\begin{align*}
	[\Sigma^i \Th(\xi), E \wedge \Th(\zeta)] &\wequi [\Sigma^i \Th(\xi) \wedge \Th(-T_X - \zeta), E] \\
	&\wequi [\Sigma^i \Th(\xi \times (- T_X - \zeta)), E] = 0.
\end{align*}
This concludes the proof.
\end{proof}

\begin{corollary} \label{cor:pure}
For $X \in \SH(S)^\pure, Y \in \SH(S)$ we have \[ Y_{c \le 0} \wedge X \wequi (Y \wedge X)_{c \le 0},\quad Y_{c \ge 0} \wedge X \wequi (Y \wedge X)_{c \ge 0} \] and \[ Y_{c=0} \wedge X \wequi (Y \wedge X)_{c=0}. \] 
In particular (taking $Y=\1$) we have \[ X_{c \le 0} \wequi X \wedge \1_{c \le 0}. \]
\end{corollary}
\begin{proof}
For any $E\in \SH(S)$, there is a unique up to equivalence cofiber sequence \[ E' \to E\to E'', \] where the first term is in $\SH(S)_{c\ge 0}$ and the last term is in $\SH(S)_{c < 0}$ \cite[Proposition 1.3.3(ii)]{beilinson1982faisceaux}.
Therefore, any stable endofunctor of $\SH(S)$ preserving the non-negative and non-positive parts of a $t$-structure must commute with truncation.
This holds for $\ph \wedge X$ by Propositions \ref{prop:pure} and \ref{prop:tensor}. 
\end{proof}

\subsection{Right completeness}
\begin{definition} \label{def:SH-rig}
Write $\SH(S)^\rig$ for the stable presentable subcategory generated by $\SH(S)^\pure$.
\end{definition}
Thus $\SH(S)_{c \ge 0} \subset \SH(S)^\rig$, and so it defines a $t$-structure on $\SH(S)^\rig$.
This is just the restriction of the Chow $t$-structure.

\begin{proposition} \label{prop:right-complete}
The Chow $t$-structure on $\SH(S)^\rig$ is right complete, i.e., we have \[ \cap_n \SH(S)^\rig_{c \le n} = 0. \]
\end{proposition}
The above definition of right completeness coincides with other possible ones, by (the dual of) \cite[Proposition 1.2.1.19]{HA} and Proposition \ref{prop:filtered}.
\begin{proof}
This is a formal consequence of the fact that $\SH(S)^\rig$ is generated by $\SH(S)_{c \ge 0}$.
\end{proof}

\begin{remark}
The Chow $t$-structure is typically not left complete, i.e. $\cap_n \SH(S)^\rig_{c \ge n} \neq 0.$
See Proposition \ref{prop:eta-periodic} below.
\end{remark}

\subsection{Interaction with $\eta$-periodization}
\begin{definition}
A spectrum $E\in \SH(S)^\rig$ is \emph{Chow-$\infty$-connective} if $E \in \SH(S)_{c \ge n}$ for all $n$.
\end{definition}

\begin{proposition} \label{prop:eta-periodic} \NB{We don't use this in the sequel. }
If $E \in \SH(S)^\rig$ is $\eta$-periodic, then $E$ is Chow-$\infty$-connective.
\end{proposition}
\begin{proof}
Right completeness and compatibility with filtered colimits (Proposition \ref{prop:filtered}) imply that that \[ E \wequi \colim_{n \to -\infty} E_{c \ge n}; \] indeed the cofiber of the natural map is in $\SH(S)_{c \le n}$ for every $n$.
Since $$ \Gmp{-k} \wequi \Sigma^k \Th(-k\scr O) \in \SH(S)_{c \ge k}, $$ we have $E_{c \ge n} \wedge \Gmp{-k} \in \SH(S)_{c \ge n+k}$ and hence $$ E \wequi E[1/\eta] \wequi \colim_{n\to-\infty} E_{c \ge n} [1/\eta] \wequi  \colim_{n\to-\infty}\colim_{k \to \infty} E_{c \ge n} \wedge \Gmp{-k} \in \SH(S)_{c \ge \infty}. $$
This was to be shown.
\end{proof} 

\begin{example}
We work entirely in $\SH(S)^\rig$.
For $E \in \SH(S)^\rig$, the fiber of the $\eta$-completion map $E \to E_\eta^\comp$ is $\eta$-periodic, and hence Chow-$\infty$-connective.
Thus $\eta$-completion induces an equivalence on Chow homotopy objects.
\end{example}

\begin{remark} \label{rmk:char-e-rig-gen}
Let $k$ be a field of exponential characteristic $e$.
Then $\SH(k)[1/e] \subset \SH(k)^\rig$ (combine \cite[Corollary 2.1.7]{elmanto2018perfection} and \cite[Proposition B.1]{levine2019algebraic}).
Hence the above results also apply to $\SH(k)[1/e]$.
\end{remark}

\subsection{Interaction with base change}
For future reference, we record the following facts about the interaction between the Chow $t$-structure and base change.
Recall the following standard definition.

\begin{definition}
Let $\scr C$ and $\scr D$ be two categories with $t$-structures. 
We say a functor $F:\scr C\to \scr D$ is \emph{right-$t$-exact (respectively left-$t$-exact)} if $F$ preserves the non-negative (respectively non-positive) parts of $t$-structures.
\end{definition}

\begin{proposition} \label{prop:bc} \NB{We don't use this.}
Let $f: S' \to S$ be a morphism of schemes.
\begin{enumerate}
\item The functor $f^*: \SH(S) \to \SH(S')$ is right-$t$-exact 
and $f_*: \SH(S') \to \SH(S)$ is left-$t$-exact. 
\item If $f$ is smooth and proper, then $f^*$ is also left-$t$-exact and its left adjoint $f_\#: \SH(S') \to \SH(S)$ is right-$t$-exact.
\end{enumerate}
\end{proposition}
\begin{proof}
The right adjoint of a right-$t$-exact functor is left-$t$-exact and vice versa (see e.g. \cite[Lemma 5]{bachmann-hurewicz}).
It thus suffices to show the right-$t$-exactness claims. 

(1) For any $X\in \Sm_S$ and $\xi\in K(X),$ we have $f^*\Th(\xi) \wequi \Th(f^*\xi)$.
It follows that $f^*$ sends the generators of $\SH(S)_{c \ge 0}$ into $\SH(S')_{c \ge 0}$.
The claim follows since $f^*$ also preserves colimits and extensions.

(2) For $X \in \Sm_{S'}$ and $\xi \in K(X)$ we have $f_\#(\Th(\xi)) = \Th_S(\xi)$.
Since $f$ is proper, if $X$ is proper over $S'$ then $X$ is proper over $S$; thus $f_\#$ preserves the generators of $\SH(\ph)_{c \ge 0}$.
We conclude as in (1).
\end{proof}
\begin{remark}
If $f$ is not smooth and proper, then $f^*$ is not in general left-$t$-exact.
Indeed if this was the case for all $f$ smooth, then $f_\#$ would be right-$t$-exact for all such $f$, and we would find that $\Sigma^{2n,n} \Sigma^\infty X_+ \in \SH(S)_{c \ge 0}$ for all $n \in \Z$ and all $X \in \Sm_S$. 
This would imply that $\SH(S)_{c \ge 0} = \SH(S)$, which is false in general (see e.g. Proposition \ref{prop:chow-MGL-homology}).
\end{remark}
\begin{example}
We give an explicit example showing that $f^*$ fails to be left-$t$-exact when $f$ is not smooth and proper.

Let $k$ be an infinite perfect field and let $U$ be an open subscheme of $\A^1_k$ such that $\emptyset \ne U \subsetneq \A^1_k$.
Then \[ \tilde\Sigma U \wequi \A^1_k/U \wequi \bigvee_{x \in \A_k^1 \setminus U} S^{2,1} \wedge x_+, \] where by $\tilde \Sigma$ we mean the unreduced suspension.\NB{I.e. take the pushout $* \coprod_{U} *$ in unpointed spaces and observe that this has a canonical base point.}
It follows from Proposition \ref{prop:chow-MGL-homology} that the spheres $S^{2,1} \wedge x_+$ are not in $\SH(S)_{c\ge 1},$ and consequently $\Sigma^\infty \tilde\Sigma U \not \in \SH(k)_{c \ge 1}$.
Since $k$ is infinite $U$ has a rational point and so $\Sigma^\infty_+ U \wequi \1 \vee \Sigma^{-1}\Sigma^\infty\tilde\Sigma U$.
Writing $f_U: U \to Spec(k)$ for the structure morphism, we deduce that $f_{U\#}(\1) \wequi \Sigma^\infty_+ U \not\in \SH(k)_{c \ge 0}$.
In particular $f^*_U$ is not left-$t$-exact.
\end{example}

\begin{example}
We now improve the construction of the previous example to show that even the base change along field extensions is not exact.
Suppose that $0 \not\in U$.
Then $\Sigma^\infty_+ U$ admits $\Sigma^{-1} S^{2,1}$ as a retract.
In fact we find that the pro-object \[ \text{``$\lim_{\emptyset \ne U}$''} \Sigma^\infty_+ U \] also admits $S^{2,1}$ as a retract.\NB{might find a reference for this in Bondarko's work on function fields}
Writing \[ f: Spec(k(t)) = \lim_{\emptyset \ne U \subset \A^1 \setminus 0} U \to Spec(k) \] we find that by continuity (see e.g. \cite[Lemma A.7]{hoyois-algebraic-cobordism}) \[ [\Sigma^{-2,-1}\1, f^* \Sigma^{-1} (\1_{c \le 0})]_{\SH(k(t))} \wequi \colim_U [\Sigma^{-2,-1}U_+, \Sigma^{-1} (\1_{c \le 0})]_{\SH(k)} \ne 0, \] since the latter group admits $[\1, \1_{c \le 0}]_{\SH(k)}$ as a summand (which is non-zero by Theorem \ref{thm:main-computation}).
It follows that $f^*$ is not $t$-exact.
\end{example}

\subsection{Cellular Chow $t$-structure}
\label{subsec:cellular-chow-t-str}
We denote by $\SH(S)^\cell \subset \SH(S)$ the subcategory of \emph{cellular spectra}, i.e. the localizing subcategory generated by the bigraded spheres $S^{p,q}$, $p, q \in \Z$.
\begin{definition} \label{def:cell-chow-t}
Denote by $\SH(S)^\cell_{c \ge 0}$ the subcategory generated under colimits and extensions by $S^{2n,n}$ for $n \in \Z$.
This is the non-negative part of a $t$-structure on $\SH(S)^\cell$ \cite[Proposition 1.4.4.11(2)]{HA} called the \emph{cellular Chow $t$-structure}.
We denote the non-positive part by $\SH(S)^\cell_{c \le 0}$.
\end{definition}

\begin{remark}
By definition we have \[ \SH(S)^\cell_{c \ge 0} \subset \SH(S)_{c \ge 0} \cap \SH(S)^\cell, \] but we do not know if the reverse inclusion holds, in general.
This would hold if cellularization is $t$-exact (rather than just left $t$-exact, which it always is).
Lemma \ref{lemm:cellularization-t-exact} \NB{This implicitly uses the main theorem.} implies that this is true over fields.\NB{The deduction is as follows: if $E \in \SH(k)_{c \ge 0} \cap \SH(k)^\cell$ then by the lemma we have $(E^\cell)_{c < 0} = (E_{c < 0})^\cell = 0$, but also $E = E^\cell$, so $E \in \SH(k)^\cell_{c \ge 0}$.}
\end{remark}

Lemma \ref{lemm:t-structure-nonpos} characterizes $\SH(S)^\cell_{c \le 0}$ in terms the vanishing of certain homotopy groups.

We now have enough notation in place to justify the multiplicative aspects of $\1_{c=0}$ and $\1_p^\comp/\tau$ over $\C$ more fully.
Specifically we make the following claims:
\begin{enumerate}
\item The truncation functor $\SH(k)_{c \ge 0} \to \SH(k)_{c=0}$ is symmetric monoidal.
\item Its right adjoint $\SH(k)_{c=0} \hookrightarrow \SH(k)_{c \ge 0}$ is lax symmetric monoidal.
\item The $p$-completion functor $\SH(k) \to \SH(k)_p^\comp$ is symmetric monoidal.
\item Its right adjoint $\SH(k)_p^\comp \hookrightarrow \SH(k)$ is lax symmetric monoidal.
\item The canonical inclusion $\SH(k)^\cell \hookrightarrow \SH(k)$ is symmetric monoidal.
\item Its right adjoint (cellularization) $\SH(k) \to \SH(k)^\cell$ is lax symmetric monoidal.
\end{enumerate}
The right adjoint of any symmetric monoidal functor is lax symmetric monoidal, whence statements (2), (4), and (6) follow from (1), (3) and (5), respectively, and (5) is obvious.
(1) and (3) hold because the functors are localizations compatible with the tensor product (i.e. the tensor product preserves weak equivalences in either variable), which for (1) follows from Proposition \ref{prop:tensor} and for (3) follows from the fact that the tensor product preserves cofibers in each variable.

\begin{proof}[Proof of Corollary~\ref{cor:intro-C}] 

Recall that over $\C$, $\tau: \Sigma^{0,-1} \1_p^\comp \to \1_p^\comp$ is the endomorphism of \cite{gheorghe2017motivic, Isa}. 
Since $S^{0,-1} = \Sigma^2 \Th(-\scr O) \in \Sigma\SH(k)_{c > 0}$ we find that the canonical map $\1_p^\comp \to (\1_{c=0})_p^\comp$ annihilates $\tau: \Sigma^{0,-1} \1_p^\comp \to \1_p^\comp$. There is thus an induced map \[ \1_p^\comp/\tau \to (\1_{c=0})_p^\comp, \] which by Theorem \ref{thm:intro-main} and \cite[Proposition 6.2.5]{Isa} induces an isomorphism on $\pi_{*,*}$.
We deduce that $\1_p^\comp/\tau$ is the $p$-complete cellularization of $\1_{c=0}$. The above discussion implies that $\1_{c=0}$ and thus $\1_p^\comp/\tau$ are both $\scr E_\infty$-rings, reproving \cite[Theorem 3.13]{gheorghe2017motivic}.
\end{proof}

Most results established above for the Chow $t$-structure on $\SH(S)$ also hold for the cellular version.

\begin{example} \label{ex:HZ-in-heart}
Let $k$ be a field of characteristic $\ne p$ with $k^\times/p \wequi \{1\}$, and containing a primitive $p$-th root of unity.
Then by \cite{voevodsky2011motivic}, we have $\pi_{*,*}(\textup{H}\Z/p) \wequi K_*^M(k)/p[\tau] \wequi \F_p[\tau]$.
This implies that $\textup{H}\Z/(p,\tau) \in \SH(k)^{\cell,\heart}$.
Indeed we have $\MGL\in \SH(k)^\cell_{c\ge 0}$ (Theorem \ref{thm:mgl-cells}). By the Hopkins--Morel isomorphism \cite{hoyois-algebraic-cobordism}, we have $\textup{H}\Z \in \SH(k)^\cell_{c\ge 0}$. Therefore $\textup{H}\Z/(p,\tau) \in \SH(k)^\cell_{c\ge 0}$. On the other hand $\textup{H}\Z/(p,\tau) \in \SH(k)^\cell_{c\le 0}$ just means that $\pi_{2m+i,m}(\textup{H}\Z/(p,\tau))$ vanishes for $m \in \Z, i>0$, which is clear since $K_*^M(k)/p \wequi \F_p$.
\end{example}
\begin{remark}
Continuing Example \ref{ex:HZ-in-heart}, write $P^i: \textup{H}\Z/p \to \Sigma^{2i(p-1),i(p-1)} \textup{H}\Z/p$ for the reduced power operation.
Since $P^i$ commutes with $\tau$ (note that if $p=2$ then $\sqrt{-1} \in k$ since $k^\times/2 \wequi \{1\}$) it defines \[ P^i/\tau: \textup{H}\Z/(p,\tau) \to \Sigma^{2i(p-1),i(p-1)} \textup{H}\Z/(p,\tau). \]
Noting that $\Sigma^{2n,n} \textup{H}\Z/(p,\tau) \in \SH(k)^{\cell, \heart}$, which is an additive $1$-category, we deduce that the operation $P^i/\tau$ is $\Z/p$-linear, i.e. comes from a map in the $\A^1$-derived category.\NB{I sort of knew this for $Sq^2$: the short exact sequence $I^{n+1} \to I^n \to k_n$ yields a reduced Bockstein $k_n \to \Sigma k_{n+1}$, which presumably is $Sq^2/\tau$ (even over any field...)}
\end{remark}

\section{Relationship to algebraic cobordism}

In this section, we show that the Chow $t$-structure is the colimit of a directed systems of $t$-structures, in the sense of Lemma \ref{lemm:chow-approx}.
The non-negative parts and the non-positive parts of these $t$-structures can be partially characterized by taking $\MGL$-homology (Lemma \ref{lemm:key}).
As a result, the homotopy groups of Chow homotopy objects can be computed by the motivic Adams--Novikov spectral sequence, which collapse at the $E_2$-page (Theorem \ref{thm:main-computation}).
In fact, we prove these results for a more general collection of $t$-structures which include Chow $t$-structure as a special case.

Throughout this section, we let $S=Spec(k)$ be the spectrum of a field of exponential characteristic $e$, and we implicitly invert $e$ throughout.
Thus we write $\SH(k)$ for $\SH(k)[1/e]$, and so on.
The main reason for this is that we need to use a vanishing result for algebraic cobordism (Theorem \ref{thm:MGL-vanishing}) which is proved by relating algebraic cobordism to higher Chow groups, and this relationship is currently only known away from the characteristic \cite{hoyois-algebraic-cobordism}.
Another simplification also occurs, which is that Remark \ref{rmk:char-e-rig-gen} applies.

\subsection{Chow degrees}
\begin{definition} \label{def:chow-degree}
Given a bidegree $(p,q) \in \Z \times \Z$, we call $c(p,q) = p - 2q$ the \emph{associated Chow degree}.
Given a bigraded abelian group $M_{{{*,*}}}$, we say that $M_{{{*,*}}}$ is concentrated in Chow degrees $\ge d$ (respectively $\le d$, $=d$) if $M_{p,q} = 0$ for $c(p,q) < d$ (respectively $> d$, $\ne d$).
\end{definition}

\begin{example}\NB{True over any base.}
\label{ex:chow_degree_lezero}
By Example \ref{ex:spheres_in_gezero}, spheres of non-negative Chow degree are in $\SH_{c \ge 0}(k)$. Therefore, if $E \in \SH(k)_{c \le 0}$, then $\pi_{{{*,*}}}E$ is concentrated in Chow degrees $\le 0$.
\end{example}

\begin{example}\label{ex:eta-comp}\NB{Move this elsewhere?}\NB{True over any base.}
Let $E \in \SH(k)$ with $\pi_{{{*,*}}} E$ concentrated in Chow degrees $\le 0$.
Consider the cofiber sequence $$\Gmp{n} \wedge E\xrightarrow{\eta^n} E\to E/\eta^n.$$
The homotopy groups $\pi_{{*,*}}(\Gmp{n} \wedge E)\wequi \pi_{*-n,*-n} E$ are concentrated in Chow degree $\le -n.$
Hence we have $\pi_{p,q}(E/\eta^n) \wequi \pi_{p,q}(E)$ for any fixed $(p,q)$ and $n$ sufficiently large.
As a result, $E \to E_\eta^\comp$ induces an isomorphism on $\pi_{{{*,*}}}$.
\end{example}

\subsection{$\MGL$-homology in a more general $t$-structure}
\label{subsec:MGL-hom}
In anticipation of our treatment of cellular objects later, we generalize our setting slightly.
Thus we fix \[ I \subset \SH(k)_{c \ge 0}. \]
We assume throughout that $I$ is closed under colimits and extensions and generated by a set of compact object; thus in particular $I$ defines a $t$-structure.
We denote its non-negative and non-positive parts by \[ \SH(k)_{I \ge 0} := I \text{ and } \SH(k)_{I \le 0}, \] and we denote the truncations by $\tau_{I \ge n}, \tau_{I \le n}$ and $\tau_{I=n}$.
We also assume that $\1 \in I$ and $\Sigma^{2,1} I = I$.

\begin{example}
\label{example:maximal_choice}
Taking $I = \SH(k)_{c \ge 0}$, we get $\tau_{I \ge 0} = (\ph)_{c \ge 0}$, and so on.
This is the maximal choice.
\end{example}
\begin{example}
Taking $I$ to be the subcategory generated under colimits and extensions by the $S^{2n,n}$, we obtain the smallest possible choice.
\end{example}

\begin{proposition} \label{prop:chow-MGL-homology} \NB{(2) $\ge 0, =0$ parts true over semilocal PID, rest true over any base}
Let $E \in \SH(k)_{I \ge 0}$ (resp. $E \in \SH(k)_{I \le 0}$, $E \in \SH(k)_{I=0}$).
Then
\begin{enumerate}
\item $\MGL \wedge E \in \SH(k)_{I \ge 0}$ (resp. $\MGL \wedge E \in \SH(k)_{I \le 0}$, $\MGL \wedge E \in \SH(k)_{I=0}$), and
\item $\MGL_{{{*,*}}} E$ is concentrated in Chow degrees $\ge 0$ (resp. $\le 0$, $=0$).
\end{enumerate}
\end{proposition}
\begin{proof}
(1) In the notation of \S\ref{app:cells}, we have $\MGL \in \SH(k)^{\pure\tate \ge 0}$ (see Theorem \ref{thm:mgl-cells}).
Thus (1) is immediate from (the proofs of) Propositions \ref{prop:tensor} and \ref{prop:pure}. 

(2) The $\le 0$ part of (2) follows from (1) and the fact that $S^{2n,n}\in I.$
It hence remains to prove the $\ge 0$ part of (2). It is sufficient to show the statement is true for the maximal choice of $I$ (Example \ref{example:maximal_choice}).

Let $\scr C \subset \SH(k)$ denote the subcategory of all spectra $E$ such that $\MGL_{{{*,*}}} E$ is concentrated in Chow degrees $\ge 0$.
It suffices to show that $\SH(k)_{c \ge 0} \subset \scr C$.
The category $\scr C$ is closed under small sums, cofibers, and extensions, and hence to show that $\SH(k)_{c \ge 0} \subset \scr C$ it is enough to show that $\Th(\xi) \in \scr C$ for every $K$-theory point $\xi$ on a smooth proper variety $X$.
We thus need to show that $\MGL_{2n+i,n} \Th(\xi) = 0$ for $i < 0$.
Using the Thom isomorphism and duality we rewrite this as \[ [S^{2n+i,n}, \MGL \wedge \Th(\xi)] \cong [\Sigma^{2n+i,n} \Th(-T_X - \xi), \MGL] \cong \MGL^{2s-i,s} X, \] where $s = \dim{X} + rank(\xi) - n$.
Thus the result follows from Theorem \ref{thm:MGL-vanishing}.
\end{proof}

\begin{corollary} \label{cor:chow-MGL-homology} \NB{True over simlocal PID.}
For $E \in \SH(k)$ arbitrary, we have \[ \MGL_{p,q} \tau_{I=i} E = \begin{cases} 0, &\text{if } c(p,q) \ne i \\ \MGL_{p,q}(E), &\text{if } c(p,q) = i. \end{cases}. \]
\end{corollary}
\begin{proof}
By replacing $E$ with $\Sigma^iE$, we reduce to the case that $i=0$.
 
The Chow degree nonzero part of the statement is immediate from Proposition \ref{prop:chow-MGL-homology}. 

For the Chow degree zero part, we consider two cofiber sequences \[ \tau_{I \ge 1} E \to \tau_{I \ge 0} E \to \tau_{I=0} E, \text{ and, } \tau_{I \ge 0} E \to E \to \tau_{I\le -1} E.\]
They yield associated long exact sequences \[\cdots\to \MGL_{2*,*} \tau_{I \ge 1}E \to \MGL_{2*,*} \tau_{I \ge 0} E \to \MGL_{2*,*} \tau_{I=0} E \to \MGL_{2*-1,*} \tau_{I \ge 1} E \to \cdots, \]
\[\cdots\to \MGL_{2*+1,*} \tau_{I \le -1} E \to  \MGL_{2*,*} \tau_{I \ge 0}E \to \MGL_{2*,*}  E \to \MGL_{2*,*} \tau_{I\leq -1} E \to \cdots. \]

By Proposition \ref{prop:chow-MGL-homology}, in each long exact sequence, the first and the last term in the above part vanish. Therefore, we have equivalences
$$\MGL_{2*,*} \tau_{I=0} E\cong \MGL_{2*,*} \tau_{I \ge 0} E \cong \MGL_{2*,*} E.$$

%
\end{proof}

\subsection{$t$-structure as a direct colimit}
\label{subsec:colimit-diagram}
We now bring the homotopy $t$-structure into play.

\begin{definition}
	Let $\SH(k)_{\ge 0}$ denote the subcategory of $\SH(k)$ generated under colimits and extensions by $\Sigma^\infty_+ X \wedge \Gmp{n}$ for $X \in \Sm_k$ and $n \in \Z$. This defines the non-negative part of the homotopy $t$-structure. 
\end{definition}


Since we are working over fields, the subcategory $\SH(k)_{\ge 0}$ is characterised by a vanishing of homotopy sheaves \cite[Theorem 2.3]{hoyois-algebraic-cobordism}: \begin{equation} \label{eq:char-ge0} \SH(k)_{\ge 0} = \left\{ E \in \SH(k) \mid \ul{\pi}_{i+n,n}(E) = 0 \text{ for } i < 0, n \in \Z \right\}, \end{equation} \[ \SH(k)_{\le 0} = \left\{ E \in \SH(k) \mid \ul{\pi}_{i+n,n}(E) = 0 \text{ for } i > 0, n \in \Z \right\}. \]
\NB{This only holds over fields! Replacement over finite dimensional base scheme: \cite[Proposition B.3]{bachmann-norms}.}

Recall that the notation $\SH(k)_{\ge d}$ denotes the subcategory $\Sigma^{d}\SH(k)_{\ge 0}$.
\begin{thmdefn}
Let $I^d$ be the category $\SH(k)_{\ge d} \cap I$. This is the non-negative part of a $t$-structure on $\SH(k)$.
\end{thmdefn}

\begin{proof}
We need to show that $I^d$ is closed under colimits and extensions and is presentable \cite[Proposition 1.4.4.11(1)]{HA}.
This is true for $\SH(k)_{\ge d}$ and $I$ by \cite[Proposition 1.4.4.11(2)]{HA}; the case of $I^d$ follows since subcategories with these properties are closed under limits (and so in particular intersections; use that limits of presentable categories along left adjoint functors\NB{also true for right adjoint functors} are presentable \cite[Proposition 5.5.3.13]{HTT}).
\end{proof}

\begin{example} \label{ex:Id} \NB{True over any base.}
Let $I$ be $\SH(k)_{c\ge 0}$.
For a smooth proper variety $X$ and a $K$-theory point $\xi$ on $X$ of rank $\ge d$, we have $\Th(\xi) \in I^d$.
Indeed it suffices to check that $\Th(\xi) \in \SH(k)_{\ge d}$, which follows from \cite[Lemma 13.1]{bachmann-norms}.
\end{example}

\begin{example} \label{ex:gens-bdded}
Let $E \in I$ be compact.
Then $E \in I^d$ for $d$ sufficiently small.
This follows from the facts that (1) the subcategory of bounded below spectra is thick, (2) the subcategory of compact spectra is the thick subcategory generated by spectra of the form $\Sigma^{n,n} \Sigma^\infty_+ X$ for $n \in \Z$ and $X \in \Sm_k$ \cite[Lemma 4.4.5]{neeman2014triangulated} \cite[Theorem 9.2]{dugger2005motivic} and (3) these generators are bounded below, by definition.
\end{example}

We use the notations $\tau^d_{\le n}$, $\tau^d_{\ge n}$, $\tau^d_{= n}$ to denote the truncation functors associated to the $t$-structure with non-negative part $I^d$. 

The following is a key technical result.
\begin{proposition} \label{lemm:key} \NB{(1) true over any base. (2) true over semilocal PID. (3,4) true only over fields? Maybe \cite[Proposition 3.7]{schmidt2018stable} helpful?}
Let $E \in I^d$.
Then we have the following:
\begin{enumerate}
\item $\tau^d_{\le 0} E \in I^d$,
\item $\MGL_{2*,*} \tau^d_{\le 0} E \cong \MGL_{2*,*} E$,
\end{enumerate}
Suppose now that $E \in I^{d+1}$.
Then we have additionally the following:
\begin{enumerate}
\setcounter{enumi}{2}
\item $\pi_{{*,*}} \tau^d_{\le 0} E$ is concentrated in Chow degrees $\le 0$, and
\item $\MGL_{{*,*}} \tau^d_{\le 0} E$ is concentrated in Chow degree $0$.
\end{enumerate}
\end{proposition}

\begin{proof}
(1) Immediate from the cofiber sequence $\tau^d_{\ge 1} E \to E \to \tau^d_{\le 0} E$, since $E \in I^d$ and $\tau^d_{\ge 1} E \in \Sigma I^d \subset I^d$.

(2) Consider the long exact sequence
\[\cdots\to \MGL_{2*,*} \tau^d_{\ge 1}E \to \MGL_{2*,*} E \to \MGL_{2*,*} \tau^d_{\le 0} E \to \MGL_{2*-1,*} \tau^d_{\ge 1} E \to \cdots. \]
It suffices to show that $\MGL_{2*+i,*} \tau^d_{\ge 1} E = 0$ for $i = 0, -1$. Since $\tau^d_{\ge 1} E \in \Sigma I^d \subset \SH(k)_{c \ge 1}$, the result follows from Proposition \ref{prop:chow-MGL-homology}.

(3) We need to show that $\pi_{2n+i,n}(\tau^d_{\le 0} E) = 0$ for $i > 0$.
This vanishing arises for two slightly different reasons, depending on if (a) $n+i > d$ or (b) $n+i \le d$.

(a) When $n+i > d$, we have $S^{2n+i,n}= \Sigma^{n+i} \Gmp{n} \in \Sigma \SH(k)_{\ge d}$. Therefore we have $S^{2n+i,n} \in \Sigma I^d$. 
Since $\tau^d_{\le 0} E$ is in the $\le 0$ part of the $t$-structure corresponding to $I^d$, the homotopy groups $\pi_{2n+i,i} (\tau^d_{\le 0} E)$ vanishes by definition.

(b) When $n+i \le d$, by Equation \eqref{eq:char-ge0}, if $F \in \SH(k)_{\ge d+1}$ then $\pi_{2n+i,n}(F) = 0$. Therefore it suffices to show that $\tau^d_{\le 0} E \in \SH(k)_{\ge d+1}.$ Consider the cofiber sequence in (1).
 Since $\tau^d_{\ge 1} E \in \Sigma I^d \subset \SH(k)_{\ge d+1}$ and $E \in I^{d+1} \subset \SH(k)_{\ge d+1}$, the result follows.

(4)  By (1) we have $\tau^d_{\le 0} E \in \SH(k)_{c \ge 0}$. Hence $\MGL_{*,*} \tau^d_{\le 0} E$ is concentrated in Chow degrees $\ge 0$, by Proposition \ref{prop:chow-MGL-homology}(2).
It remains to show that $\MGL_{{*,*}} \tau^d_{\le 0} E$ is concentrated in Chow degrees $\le 0$.

Let $\scr C \subset \SH(k)$ be the subcategory consisting all spectra $M$ such that $\pi_{{*,*}}(M \wedge \tau_{\le 0}^d E)$ is concentrated in Chow degrees $\le 0$.
It suffices to show that $\MGL \in \scr C$.
Note that $\scr C$ is closed under filtered colimits, extensions, wedge sums, and $\Sigma^{2n,n}$ for any $n \in \Z$.
By (3), $\1 \in \scr C$, and hence (in the notation of \S\ref{app:cells}) $\SH(k)^{\pure\tate\ge 0} \subset \scr C$.
The result thus follows from Theorem \ref{thm:mgl-cells}.
\end{proof}

\begin{proposition} \label{lemm:chow-approx} \NB{true over any base}
Let $E \in \SH(k)$.
There are directed systems \[ \tau^n_{\ge 0} E \to \tau^{n-1}_{\ge 0} E \to \dots \to \tau_{I \ge 0} E \] and \[ \tau^n_{\le 0} E \to \tau^{n-1}_{\le 0} E \to \dots \to \tau_{I \le 0} E \] which are in fact colimit diagrams.
\end{proposition}
\begin{proof}
Since $I^n \subset I^{n-1} \subset \dots \subset \SH(k)_{I \ge 0}$, the directed systems exist.

To see that the second one is a colimit diagram, it suffices to show that $\colim_n \tau^n_{\le 0} E \in \SH(k)_{I \le 0}$.
(Indeed the fiber of $E \to \colim_n \tau^n_{\le 0} E$ is $\colim_n \tau^n_{>0} E \in \SH(k)_{I > 0}$, so we conclude by \cite[Proposition 1.3.3(ii)]{beilinson1982faisceaux}.)
This follows from Example \ref{ex:gens-bdded}, which shows that every generator of $\SH(k)_{I \ge 0}=I$ is in $I^n$ for $n$ sufficiently small.

Taking the fiber of the constant colimit diagram $E$ mapping to the second one yields the first  one (up to a shift), which is thus also a colimit diagram.
This concludes the proof.
\end{proof}

\subsection{$\MGL$-completion and Adams--Novikov spectral sequence}
\label{subsec:ANSSmain-computation}

We come to the main result of this section. 

Recall from \cite[\S7]{dugger2010motivic} or \cite{hu2011remarks} the usual $\MGL$ based motivic Adams--Novikov spectral sequence
$$\Ext^{*,*,*}_{\MGL_{{*,*}}\MGL}(\MGL_{{*,*}}, \MGL_{{*,*}}X)\implies \pi_{{*,*}}X^{\wedge}_{\MGL}.$$
Recall that $(\MGL_{2*,*},\MGL_{2*,*}\MGL)$ is a Hopf algebroid canonically isomorphic to $(\MU_{2*},\MU_{2*}\MU)$ (combine \cite[Corollary 6.7, Lemma 6.4]{naumann2009motivic} and \cite[Corollary 6.7]{SpitzweckMGL} and recall our convention that we implicitly invert $e$). \NB{True over Dedekind domain with vanishing Picard group}
In particular for $E \in \SH(k)$, the graded abelian group $\MGL_{2*,*} E$ is canonically a comodule over $\MU_{2*}\MU$.
\begin{theorem} \label{thm:main-computation}
Let $E \in \SH(k)_{I \ge 0}$.
Then
\begin{enumerate}
\item the canonical map $\tau_{I\le 0}E \to \tau_{I \le 0}(E)_\MGL^\comp$ to the $\MGL$-nilpotent completion induces an isomorphism on $\pi_{*,*}$, and
\item $\pi_{2w-s, w} \tau_{I \le 0} E \cong \Ext^{s, 2w}_{\MU_*\MU}(\MU_*,\MGL_{2*,*}E).$\NB{indexing?}
\end{enumerate}
\end{theorem}

\begin{proof}
We first prove part (2) of the theorem.

We have isomorphisms \[ \tau_{I \le 0} E \stackrel{(*)}{\wequi} \tau_{I \le 0} \tau_{I \ge 0} E \stackrel{L.\ref{lemm:chow-approx}}{\wequi} \tau_{I \le 0} \colim_d \tau_{\ge 0}^d E \stackrel{C.\ref{cor:filtered-trunctation}}{\wequi} \colim_d \tau_{I \le 0} \tau_{\ge 0}^d E \stackrel{L.\ref{lemm:chow-approx}}\wequi \colim_{n,d} \tau_{\le 0}^n \tau_{\ge 0}^d E, \] where (*) holds because $E \in \SH(k)_{I \ge 0}$. The third equivalence uses Corollary~\ref{cor:filtered-trunctation} which is stated for the largest $I$. By a similar proof we can show it also works for other choices of $I$.
Since the right hand side of the isomorphism in part $(2)$ is compatible with filtered colimits in $E$ \cite[Lemma 3.2.2(b)]{Hovey}, we may replace $E$ by $\tau_{\ge 0}^d E$ and so we may assume that $E \in I^d$.
Let $n<d$.
We shall show that $\pi_{2w-s, w} \tau^n_{\le 0} E \wequi \Ext^{s, 2w}_{\MU_*\MU}(\MGL_{2*,*} E)$.
Taking the (constant) colimit as $n \to -\infty$ will yield the result.

By Lemma \ref{lemm:key}~(2, 4), we find that $\MGL_{2*+i,*} \tau^n_{\le 0} E = \MGL_{2*,*} E$ for $i=0$ and vanishes else.
Hence the Adams--Novikov spectral sequence for $\tau^n_{\le 0} E$ again collapses, and it suffices to show that this spectral sequence converges to $\pi_{*,*} \tau^n_{\le 0} E$, or in other words that $\tau^n_{\le 0} E$ is $\MGL$-nilpotent complete (on homotopy groups).
By Lemma \ref{lemm:key}~(1), $\tau^n_{\le 0}(E)$ is connective in the homotopy $t$-structure; hence by \cite[\S5.1 and Theorem 7.3.5]{mantovani2021localizations} we have \[ \tau^n_{\le 0}(E)_\MGL^\comp \wequi L_\MGL \tau^n_{\le 0}(E) \wequi \tau^n_{\le 0}(E)_\eta^\comp. \]
It is thus sufficient to show that $\tau^n_{\le 0}(E) \to \tau^n_{\le 0}(E)_\eta^\comp$ induces an equivalence on $\pi_{*,*}$.
This follows from Lemma \ref{lemm:key}(3) and Example \ref{ex:eta-comp}.

This concludes the proof of part (2).

For part $(1)$, Corollary \ref{cor:chow-MGL-homology} implies that the Adams--Novikov spectral sequence for $\tau_{I \le 0} E$ collapses and $$\pi_{2w-s, w} \tau_{I \le 0}(E)_\MGL^\comp \cong \Ext^{s,2w}_{\MU_*\MU}(\MU_*, \MGL_{2*,*}E).$$\NB{details?}
Hence part (1) follows from part (2).
\end{proof}

\subsection{First consequences}
Theorem \ref{thm:main-computation} is the central result of this paper.
We establish here some of its immediate consequences, many of which will be amplified in \S\ref{sec:reconstruction}.

\begin{lemma} \label{lem:smash-with-thom}
Let $E \in \SH(k)$.
If $\pi_{i,0}(E\wedge \Th(\xi))=0$ for all $i \in \Z$ and $K$-theory points $\xi$ on smooth proper varieties $X$, then $E \wequi 0$.
\end{lemma}
\begin{proof}
By duality, the assumption is equivalent to $[\Sigma^{i}\Th(\xi), E]=0$ for all $i \in \Z$ and $K$-theory points $\xi$ on smooth proper varieties $X$.
Taking $\xi = n\scr O$ shows that $[\Sigma^{i,j} \Sigma^\infty_+ X, E]=0$ for all $i,j \in \Z$ and all $X$ smooth and proper.
This implies that $E=0$ by Remark \ref{rmk:char-e-rig-gen}.
\end{proof}

\begin{remark}
Let $E \in \SH(k)^\heart$.
By Theorem \ref{thm:main-computation}, the map $\alpha: E \to E_\MGL^\comp$ induces an isomorphism on $\pi_{*,*}$.
Let $X$ be smooth projective and $\xi \in K(X)$.
By Corollary \ref{cor:pure} we have $E \wedge \Th(\xi) \in \SH(k)^\heart$, and hence $E \wedge \Th(\xi) \to (E \wedge \Th(\xi))_\MGL^\comp$ again induces an equivalence on $\pi_{*,*}$.
But tensoring with the strongly dualizable object $\Th(\xi)$ preserves limits, so $(E \wedge \Th(\xi))_\MGL^\comp \wequi E_\MGL^\comp \wedge \Th(\xi)$.
Lemma \ref{lem:smash-with-thom} now implies that $\alpha$ is an equivalence.
This is a special case of Proposition \ref{prop:comonadic-reconstruction} in the next section.
\end{remark}

We deduce the following strengthening of Proposition \ref{prop:eta-periodic}.
\begin{corollary} \label{cor:detect-infinity-connective}
Let $E \in \SH(k)$.
Then $E$ is Chow-$\infty$-connective if and only if $E \wedge \MGL \wequi 0$.
\end{corollary}
\begin{proof}
By right completeness (Proposition \ref{prop:right-complete}), $E \in \SH(k)_{c \ge n}$ if and only if $E_{c=i} = 0$ for $i < n$.
(Indeed necessity is clear, and for sufficiency note that $E_{c<n} \in \cap_{m} \SH(k)_{c < m} = 0$.)
Consequently $E$ is Chow-$\infty$-connective if and only if $E_{c=i}=0$ for all $i$.

We hence find:
\begin{gather*}
\text{$E$ is Chow-$\infty$-connective} \\
\Leftrightarrow E_{c=i} = 0 \text{ for all $i$} \\
\stackrel{L.\ref{lem:smash-with-thom}}{\Leftrightarrow} \pi_{*,*}(E_{c=i} \wedge \Th(\xi)) = 0 \text{ for all $i$, $X$ smooth proper and $\xi \in K(X)$} \\
\stackrel{C.\ref{cor:pure}}{\Leftrightarrow} \pi_{*,*} (E \wedge \Th(\xi))_{c=i} = 0 \text{ for all $i$} \\
\stackrel{T.\ref{thm:main-computation}+T.\ref{thm:mgl-cells}}{\Leftrightarrow} \MGL_{*,*}(E \wedge \Th(\xi))_{c=i} = 0 \text{ for all $i$} \\
\stackrel{C.\ref{cor:chow-MGL-homology}}{\Leftrightarrow} \MGL_{2*+i,*}(E \wedge \Th(\xi)) = 0 \text{ for all $i$}\\
\Leftrightarrow \pi_{*,*}(\MGL \wedge E \wedge \Th(\xi)) = 0 \\
\stackrel{L.\ref{lem:smash-with-thom}}{\Leftrightarrow} \MGL \wedge E = 0
\end{gather*}
\end{proof}

\section{Reconstruction theorems} \label{sec:reconstruction}
We keep the conventions from the last section: $S=Spec(k)$ is the spectrum of a field of exponential characteristic $e$, which is implicitly inverted throughout.

In this section we will amplify the results of the last section by deducing fairly explicit descriptions of the categories $\SH(k)^\heart$ and $\1_{c=0}\Mod$. 
We begin in \S\ref{subsec:comonadic} by describing these categories in terms of a certain comonad $C$ on $\MGL_{c=0}\Mod$.
Then in \S\ref{subsec:MGL-mod} we explain how work of Bondarko supplies a very explicit description of the category $\MGL_{c=0}\Mod$, and we also identify the comonad $C$.
Finally in \S\ref{subsec:chow-heart} we use this to provide an explicit description of $\SH(k)^\heart$.

\subsection{Comonadic descent} \label{subsec:comonadic}
\subsubsection{}
Suppose given a presentably symmetric monoidal category $\scr D$ and $A \in \CAlg(\scr D)$.
We obtain a free-forgetful adjunction \[ F: \scr D \adj A\Mod: U, \] and the endofunctor \[ C := FU = \otimes A: A\Mod \to A\Mod \] which acquires the structure of a comonad \cite[Proposition 4.7.3.3]{HA}.

We denote by $C\CoMod$ the category of comodules under $C$ \cite[Definition 4.2.1.13]{HA}, and hence obtain a factorization \cite[\S4.7.4]{HA} \[ \scr D \adj C\CoMod \adj A\Mod, \] where $C\CoMod \to A\Mod$ is the forgetful functor which we denote by $H$ (with right adjoint the cofree comodule functor) and $\scr D \to C\CoMod$ sends $X$ to $X \otimes A$ with its canonical comodule structure.

\subsubsection{Cobar resolution}
We can form the cobar resolution \cite[Construction 2.7]{mathew2017nilpotence} \[ \1_{\scr{D}} \to \CB(A) := \left( A \rightrightarrows A \otimes A \triplearrows \dots \right); \] this as a coaugmented cosimplicial object in $\CAlg(\scr D)$.
Taking module categories, we obtain a coaugmented cosimplicial category\NB{note that the degeneracies are not the right adjoints, but rather left Kan extension along multiplication} \begin{equation} \label{eq:C-nilp} \scr D \to \CB(A)\Mod := \left( A\Mod \rightrightarrows (A \otimes A)\Mod \triplearrows \dots \right). \end{equation}
It follows from \cite[Theorem 4.7.6.2]{HA} that \[ \lim_{\Delta} \CB(A)\Mod \wequi C\CoMod. \]

Let $\Delta_s$ be the subcategory of $\Delta$ with the same objects, but where the morphisms are
given by injective order preserving maps between nonempty linearly ordered sets. We will often be interested in the restriction of diagram \eqref{eq:C-nilp} to the coinitial subcategory $\Delta_s \hookrightarrow \Delta$ \cite[Lemma 6.5.3.7]{HTT}, i.e. view this as a coaugmented \emph{semi}-cosimplicial object.
The limit of any cosimplicial object is the same as the limit of its associated semi-cosimplicial object. Therefore we also have 
\[ \lim_{\Delta_s} \CB(A)\Mod \wequi C\CoMod. \]

\subsubsection{Monoidal structure}
The category $C\CoMod$ is in fact symmetric monoidal, and the left adjoint functors above are symmetric monoidal.
Indeed \eqref{eq:C-nilp} is a diagram of symmetric monoidal categories and symmetric monoidal functors (since it comes from a diagram of commutative rings and commutative ring maps).
Since limits of symmetric monoidal categories are computed on the underlying categories \cite[Corollary 3.2.2.5]{HA}, it follows that the limit $C\CoMod$ is symmetric monoidal, as desired.


\subsubsection{$t$-structure}
\label{subsec:t_structure}
Suppose that $\scr D$ carries a $t$-structure, and $\otimes A: \scr D \to \scr D$ is $t$-exact.
We can give each of the categories $A^{\otimes n}\Mod$ the $t$-structure detected by the forgetful functor to $\scr D$ (see e.g. \cite[Lemma 29]{bachmann-tambara}).

When viewed as a semi-cosimplicial category, the maps in \eqref{eq:C-nilp} become $t$-exact functors.

For $-\infty \le m \le n \le +\infty$, we denote by $A^{\otimes p}\Mod_{[m,n]}$ the subcategory of objects bounded in the $t$-structure.
Then $\CB(A)\Mod_{[m,n]}$ is a full subdiagram of \eqref{eq:C-nilp}.
Since limits preserve fully faithful functors\NB{because mapping spaces in limits of categories are computed in the obvious way}, we find that \[ C\CoMod_{[m,n]} := \lim_{\Delta_s} \CB(A)\Mod_{[m,n]} \] is a full subcategory of $C\CoMod$; in fact $C\CoMod_{[m,n]}$ is equivalent to ${H}^{-1}(A\Mod_{[m,n]})$, where $H: C\CoMod \to A\Mod$ is the forgetful functor.

Let $C_{[m,n]}: A\Mod_{[m,n]} \to A\Mod_{[m,n]}$ denote the restriction of $C$, and $C^\heartsuit = C_{[0,0]}$.
We note that by construction \begin{equation} \label{eq:heart-comod} C\CoMod_{[m,n]} \wequi C_{[m,n]}\CoMod. \end{equation}

One checks immediately that \[ C\CoMod_{\ge 0} := C\CoMod_{[0,\infty]} \quad\text{and}\quad C\CoMod_{\le 0} := C\CoMod_{[-\infty,0]} \] define a $t$-structure on $C\CoMod$. 

\subsubsection{Compact generation} \label{subsub:compact-gen}
We consider the $t$-structure on $C\CoMod$ defined in \S\ref{subsec:t_structure}. 

Suppose that $C\CoMod^\heartsuit$ is compactly generated. Let $C\CoMod^{\heartsuit\omega}$ denote the category of its compact objects.

We consider the category \[ \Stable(C) := \Ind(\Thick(C\CoMod^{\heartsuit\omega})), \] where $\Thick(C\CoMod^{\heartsuit\omega})$ denotes the thick subcategory of $C\CoMod$ generated by $C\CoMod^{\heartsuit\omega}$, and $\Ind$ denote the category obtained by freely adding filtered colimits (see \cite[5.3.5.1]{HTT}).
We obtain an adjunction \cite[Propositions 5.3.5.10 and 5.3.5.13]{HTT} \[ \Stable(C) \adj C\CoMod. \]

\begin{remark}
The assumption that $C\CoMod^\heartsuit$ is compactly generated need not imply that $C\CoMod$ is compactly generated.
Moreover if $E \in C\CoMod^\heartsuit$ is compact, it need not be the case that $E \in C\CoMod$ is compact.
\end{remark}

\begin{remark}
If $C$ is the comonad describing comodules over some Hopf algebroid $\Psi$, then under mild assumptions $\Stable(C)$ coincides with Hovey's category $\Stable(\Psi)$ \cite{Hovey}.
See \cite[\S3]{barthel2018algebraic} for a treatment in the language of $\infty$-categories.
In this situation we will use the notations $\Stable(C)$ and $\Stable(\Psi)$ interchangeably.
\end{remark}

\subsubsection{Modules over $\1_{c=0}$}
We apply the above discussion with $\scr D = \1_{c=0}\Mod$ and $A = \MGL \wedge \1_{c=0}$ (which is equivalent to $\MGL_{c=0}$ by Corollary \ref{cor:pure}).
We obtain \[ \1_{c=0}\Mod \adj C\CoMod \adj \MGL_{c=0}\Mod. \]
We consider the $t$-structure on $\1_{c=0}\Mod$ induced by the Chow $t$-structure (Remark \ref{rem:chow-t-Amod}).

\begin{proposition} \label{prop:comonadic-reconstruction}
\hfill
\begin{enumerate}
\item The free functor $\bar F: \1_{c=0}\Mod \to C\CoMod$ is $t$-exact and symmetric monoidal.
\item For $-\infty \le m \le n < \infty$, the restriction $\1_{c=0}\Mod_{[m,n]} \to C\CoMod_{[m,n]}$ is an equivalence.
  In particular \[ \1_{c=0}\Mod^\heart \wequi C\CoMod^\heart. \]
\item The functor $\bar F$ induces a symmetric monoidal equivalence $\1_{c=0}\Mod \wequi \Stable(C)$.
\end{enumerate}
\end{proposition}
\begin{proof}
(1) Clear by construction.

(2) We shall apply the Barr-Beck-Lurie theorem \cite[Theorem 4.7.3.5]{HA}.
It hence suffices to show that (a) $F: \1_{c=0}\Mod_{[m,n]} \to \MGL_{c=0}\Mod_{[m,n]}$ is conservative, and (b) $F$-split totalizations exist in $\1_{c=0}\Mod_{[m,n]}$ and are preserved by $F$. 

(a) For $E \in \1_{c=0}\Mod$ we have $E \wedge_{\1_{c=0}} \MGL_{c=0} \wequi E \wedge \MGL,$ and this is zero if and only if $E$ has vanishing Chow homotopy objects (Corollary \ref{cor:detect-infinity-connective}).
Thus if $E$ is additionally Chow bounded from the left, then $E \wequi 0$ by right completeness.

(b) Since our categories are presentable, totalizations (and in fact all small limits) exist \cite[Corollary 5.5.2.4]{HTT}.
The truncation functors $(\ph)_{c\ge m}$ commute with limits, and the subcategories of $n$- coconnective objects are preserved by limits.
It follows that limits in $\1_{c=0}\Mod_{[m,n]}$ are computed by first computing in $\1_{c=0}\Mod$ and then applying $(\ph)_{c \ge m}$, and similarly for $\MGL_{c=0}\Mod_{[m,n]}$.
Being $t$-exact, $F$ commutes with $(\ph)_{c \ge m}$, and hence it is enough to show that $F: \1_{c=0}\Mod_{[-\infty,n]} \to \MGL_{c=0}\Mod_{[-\infty,n]}$ preserves totalizations.
Since the forgetful functors are conservative and preserves limits, it is enough to show that $\wedge \MGL$ preserves totalizations in $\SH(k)_{c \le n}$, i.e. $\MGL \wedge \lim_{n \in \Delta} E^n\simeq \lim_{n \in \Delta} \MGL \wedge E^n$.
We are dealing with Chow-bounded above objects, so it suffices to show that the induced maps on Chow homotopy objects are isomorphisms (Proposition \ref{prop:right-complete}).
This follows from $t$-exactness of smashing with $\MGL$ (Corollary \ref{cor:pure}) and \cite[Proposition 1.2.4.5(5)]{HA}\NB{or rather its dual}, i.e. the fact that totalizations of bounded above spectra behave like finite limits on homotopy objects.
More precisely, for $E^\bullet \in \SH(k)_{c \le n}$ and $i \in \Z$ there exists $N \gg 0$ such that
\begin{align*}
	(\MGL \wedge \lim_{n \in \Delta} E^n)_{c=i} &\wequi \MGL \wedge (\lim_{n \in \Delta} E^n)_{c=i} \\
	&\wequi \MGL \wedge (\Tot^{N}E^\bullet)_{c=i} \\
	&\wequi (\MGL \wedge \Tot^{N}E^\bullet)_{c=i} \\
	&\wequi (\Tot^{N}(\MGL \wedge E^\bullet))_{c=i} \wequi (\lim_{n \in \Delta} \MGL \wedge E^n)_{c=i};
\end{align*}
here $\Tot^N$ refers to the $N$-th partial totalization, i.e. the limit over $\Delta_{\le N}$.

(3) Let $\scr G \subset \1_{c=0}\Mod$ the class of objects of the form $\Th(\xi) \wedge \1_{c=0} \wequi \Th(\xi)_{c=0}$ (see Corollary \ref{cor:pure}).
They form a compact generating family, and $\scr G \subset \1_{c=0}\Mod^\heart$.
There is an induced symmetric monoidal functor \[ \Ind(\Thick(\scr G)) \to \Ind(\Thick(\bar F \scr G)), \] which is an equivalence by (2).
The left hand side is $\1_{c=0}\Mod$.
The right hand side is $\Stable(C)$, by definition (and (2)).
This is the desired result.
\end{proof}

\begin{corollary} \label{corr:heart-abstract}
There are canonical symmetric monoidal equivalences \[ \SH(k)^\heart \wequi \1_{c=0}\Mod^\heart \wequi C^\heart\CoMod, \] where $C^\heart$ is the comonad on $\MGL_{c=0}\Mod^\heart$ obtained by restricting $C$.
\end{corollary}
\begin{proof}
The functor $\SH(k) \to \1_{c=0}\Mod$ induces an equivalence of the hearts, e.g. by \cite[Lemma 29]{bachmann-tambara}. 
The second equivalence follows from Proposition \ref{prop:comonadic-reconstruction}(2) via Equation \eqref{eq:heart-comod}.
\end{proof}

\subsection{Modules over $\MGL_{c=0}$} \label{subsec:MGL-mod}
\subsubsection{Pure $\MGL$-motives}
For a smooth proper variety $X$ and $i \in \Z$, denote by $X\{i\} \in \MGL_{c=0}\Mod$ the object $(\Sigma^{2i,i} X_+ \wedge \MGL)_{c=0} \wequi \Sigma^{2i,i} (X_+)_{c=0} \wedge \MGL$.
By the Thom isomorphism, these are generators of $\MGL_{c=0}\Mod$ (as a localizing subcategory).
Write $\PM_\MGL(k) \subset \MGL_{c=0}\Mod$ for the smallest idempotent complete additive subcategory containing the objects $X\{i\}$.
For now we view this as a spectrally enriched category.
By duality and adjunction, we have 
\begin{align*}
	\Map_{\PM_\MGL(k)}(X\{i\}, Y\{j\}) &\wequi \Map_{\PM_\MGL(k)}(\MGL_{c=0}, (X \times Y)\{j-d_X-i\})\\
	&\wequi \Map_{\SH(k)}(S^{2(i+d_X-j),(i+d_X-j)}, ((X \times Y)_+)_{c=0} \wedge \MGL).
\end{align*}
Proposition \ref{prop:chow-MGL-homology} thus implies that \[ {\pi}_{*}\Map_{\PM_\MGL(k)}(X\{i\}, Y\{j\}) \wequi \MGL_{*+2(i+d_X-j),(i+d_X-j)}(((X \times Y)_+)_{c=0}) \] is concentrated in degree $0$. 
In other words our spectrally enriched category $\PM_\MGL(k)$ is just an additive ordinary $1$-category.
The above computation together with Corollary \ref{cor:chow-MGL-homology} shows that 
\begin{equation}
\label{eq:graded_Hom}
\begin{aligned} 	\Hom_{\PM_\MGL(k)}(X\{i\}, Y\{j\}) &\wequi \MGL_{2(i+d_X-j),(i+d_X-j)}(X \times Y) \\ &\wequi \MGL^{2(d_Y+j-i),(d_Y+j-i)}(X \times Y). \end{aligned}
\end{equation}

\begin{remark} \label{rmk:PM-A}
For future use, we point out the following generalization.
If $B \in \SH(k)$ is any oriented ring spectrum, we can form a category $\PM_B$ of \emph{pure $B$-motives}.
It is the idempotent complete additive $1$-category generated by objects $X\{i\}_B$ and morphisms \[ \Hom_{\PM_B}(X\{i\}_B, Y\{j\}_B) = [\Sigma^{2i,i} \Sigma^\infty_+ X\wedge B, \Sigma^{2j,j} \Sigma^\infty_+ Y\wedge B]_{B\Mod}. \]
This construction enjoys the following properties:
\begin{enumerate}
\item If $M$ is in $B\Mod$, then the functor $M_*(X):=[\Sigma^{2*,*} \Sigma^\infty_+ X \wedge B, M]_{B\Mod}$ defines a linear presheaf on $\PM_B$.
\item If $u: A \to B$ is a morphism of oriented ring spectra and $F$ is any $A$-module, then $B \otimes_A F$ is a $B$-module and we have a canonical map $u_F: F \to B \otimes_A F$.
  Given $\alpha: X \to Y \in \PM_A$ and $s \in F_*$ we have $u_{F*}(\alpha^* s) = u_*(\alpha)^*(u_{F*} s)$.
\end{enumerate}
\end{remark}

\begin{remark}
Given an oriented cohomology theory $A^*$ on smooth proper $k$-varieties, one can define a $1$-category of pure $A$-motives \cite[\S6]{nenashev2006oriented}.
It is generated as an idempotent complete additive $1$-category by objects $X\{i\}$ with sets of maps \[ [X\{i\}, Y\{j\}] = A^{j+d_Y-i}(X \times Y), \] and composition given by convolution.
Taking $A^* = B^{2*,*}$ for some oriented ring spectrum $B$, we recover the category $\PM_B$ (in particular taking $A^* = \MGL^{2*,*}$ we recover $\PM_\MGL(k)$). Given the above description, the only part of this assertion which we have not proved yet is that composition in $\PM_A$ is given by convolution; this is a purely formal consequence of the rigidity of the generators.
\end{remark}

\subsubsection{Spectral Morita theory} \label{subsub:spectral-morita}

By Remark \ref{rmk:char-e-rig-gen} and the Thom isomorphism, 
$\PM_\MGL(k)$ compactly generates   
$\MGL_{c=0}\Mod$. Since we have an explicit description of $\PM_\MGL(k)$ as a spectrally enriched category, we should be able to recover $\MGL_{c=0}\Mod$ by a variant of Morita theory, such as \cite{schwede2003stable}.
For an $\infty$-category $\scr D$ with finite coproducts, we use the notation \[ \PSh_\Sigma(\scr D) = \Fun^\times(\scr D^\op, \Spc), \quad \PSh_\SH(\scr D) = \Fun^\times(\scr D^\op, \SH), \quad\text{and}\quad \PSh_\Ab(\scr D) = \Fun^\times(\scr D^\op, \Ab), \]
where $\Fun^\times$ denotes the category of product-preserving functors.
Provided that $\scr D$ is additive\NB{In this case also $\Fun^\times(\scr D^\op, \Set) \wequi \Fun^\times(\scr D^\op, \Ab)$}, there are equivalences \cite[Remark C.1.5.9]{SAG} \[ \PSh_\SH(\scr D)_{\ge 0} \wequi \PSh_\Sigma(\scr D) \quad\text{and}\quad \PSh_\SH(\scr D)^\heartsuit \wequi \PSh_\Ab(\scr D). \]

Note that $\PSh_\SH(\scr D)$ has a natural (pointwise) $t$-structure.
\begin{lemma} \label{lemm:spectra-presheaves-t-str} \NB{surely there must be a reference?}
Let $\scr D$ be a small semi-additive $\infty$-category.
The full subcategory $\PSh_\SH(\scr D)_{\ge 0}$ consisting of functors $F: \scr D^\op \to \SH_{\ge 0} \subset \SH$ is generated under colimits and extension by the image of the canonical functor $\scr D \to \PSh_\SH(\scr D)$.
In particular $\PSh_\SH(\scr D)_{\ge 0}$ is the non-negative part of a $t$-structure.
Its non-positive part consists of the functors $\scr D^\op \to \SH_{\le 0} \subset \SH$.
\end{lemma}
\begin{proof}
For $d \in \scr D$ write $R_d \in \PSh_\SH(\scr D)$ for the ``representable functor''.
We have\NB{ref?} \[ (*) \quad \map(R_d, R_e) \wequi \Map(d,e)^{gp} \in \SH_{\ge 0}, \] where the superscript $gp$ denotes group completion of the additive $\scr E_\infty$-monoid structure.
Consider the $t$-structure on $\PSh_\SH(\scr D)$ generated by the objects $R_d$ for $d \in \scr D$.
We have an adjunction $\SH \adj \PSh_\SH(\scr D): ev_d$. The left adjoint is right-$t$-exact, since it sends the generator $\1 \in \SH$ to $R_d$, whence the right adjoint $ev_d$ is left-$t$-exact.
By $(*)$ it is also right-$t$-exact.
Thus the functors $ev_d$ for $d \in \scr D$ form a conservative $t$-exact collection, whence the non-negative and non-positive parts of the $t$-structure are as claimed.
\end{proof}

\begin{proposition} \label{prop:MGLc0-mod}
We have a canonical $t$-exact, symmetric monoidal equivalence \[ \MGL_{c=0}\Mod \wequi \PSh_\SH(\PM_\MGL(k)). \]
\end{proposition}
\begin{proof}
This is standard.
To be more precise, the symmetric monoidal category $\PSh_\SH(\PM_\MGL(k))$ can be obtained as the stabilization of a localization of $\PSh(\PM_\MGL(k))$ (see e.g. \cite[Remark 2.10]{aoki2020weight}).
The universal properties of presheaves, day convolution, localization and stabilization thus imply that there exists a unique cocontinuous symmetric monoidal functor \[ \PSh_\SH(\PM_\MGL(k)) \to \MGL_{c=0}\Mod \] extending the inclusion $\PM_\MGL(k) \to \MGL_{c=0}\Mod$.\NB{surely there must be a direct reference for this}
Since it induces equivalences on mapping spectra between compact generators, it is an equivalence.
By Lemma \ref{lemm:spectra-presheaves-t-str}, this equivalence identifies the non-negative parts of the $t$-structures, and hence is $t$-exact.
\end{proof}

Given $E \in \MGL_{c=0}\Mod$, denote by $\ul{\pi}_i^c(E) \in \PSh_\Ab(\PM_\MGL(k))$ the presheaf with $\ul{\pi}_i^c(E)(X) = \pi_i \Map_{\MGL_{c=0}\Mod}(X, E), ~~\forall X\in \PM_\MGL(k)$.

\begin{corollary}
\begin{enumerate}
\item The functor $\ul{\pi}_0^c: \MGL_{c=0}\Mod^\heart \to \PSh_\Ab(\PM_\MGL(k))$ is an equivalence.
\item For $E \in \MGL_{c=0}\Mod$ we have $E \in \MGL_{c=0}\Mod_{c \ge 0}$ (respectively $E \in \MGL_{c=0}\Mod_{c \le 0}$) if and only if $\ul{\pi}^c_i(E) = 0$ for all $i < 0$ (respectively $i > 0$).
\item The Chow $t$-structure on $\MGL_{c=0}\Mod$ is non-degenerate \cite[p. 32]{beilinson1982faisceaux}.
\end{enumerate}
\end{corollary}
\begin{proof}
By Proposition \ref{prop:MGLc0-mod}, all statements translate into assertions about $\PSh_\SH(\PM_\MGL(k))$, which are easily verified.
\end{proof}

\begin{remark} \label{rmk:MGL-c0-mod-explicit-heart}
An object $F \in \MGL_{c=0}\Mod^\heart \wequi \PSh_\Ab(\PM_\MGL(k))$ consists of the following data:
\begin{itemize}
\item For every smooth proper variety $X$ a graded abelian group $F(X)_* = \ul{\pi}^c_0(F)(X\{*\})$.
\item For every graded $\MGL$-correspondence $\alpha: X \to Y$ (i.e. $\alpha \in \MGL^{2*,*}(X \times Y)$) a homomorphism $\alpha^*: F(Y)_* \to F(X)_*$,
\end{itemize}
subject to the conditions that:
\begin{itemize}
\item for composable $\MGL$-correspondences $\alpha, \beta$ we have $\alpha^*\beta^* = (\beta \alpha)^*$,
\item $\id^* = \id$ and $0^* = 0$, as well as
\item for parallel $\MGL$-correspondences $\alpha, \beta$ we have $\alpha^* + \beta^* = (\alpha + \beta)^*$.
\end{itemize}
For example, since $\MGL_{2*,*} \cong \MU_{2*}$, each $F(X)_*$ is an $\MU_{2*}$-module, and all the $\alpha^*$ are automatically $\MU_{2*}$-module maps.
\end{remark}

\begin{remark} \label{rmk:MGL-c-D}
The category $\PSh_\Ab(\PM_\MGL(k))$ has enough projective objects, namely the representable presheaves.
When viewed as objects of $\PSh_\SH(\PM_\MGL(k))$, they have mapping spectra concentrated in degree $0$ (see the proof of Lemma \ref{lemm:spectra-presheaves-t-str}).
Using \cite[Proposition 1.3.3.7]{HA} (or spectra Morita theory) this implies that $\PSh_\SH(\PM_\MGL(k)) \wequi D(\PSh_\Ab(\PM_\MGL(k)))$.
\end{remark}

To summarize, we have 
$$\begin{tikzcd}
 \PM_\MGL(k)\ar[r,hook] & \MGL_{c=0}\Mod^\heart \ar[r,"\simeq"] \ar[d,hook]&\PSh_\Ab(\PM_\MGL(k)) \ar[d,hook]  \\
 &	\MGL_{c=0}\Mod\ar[r,"\simeq"']&\PSh_\SH(\PM_\MGL(k))   & \ar[l,"\wequi"] D(\PSh_\Ab(\PM_\MGL(k)))
\end{tikzcd}$$
where the downwards arrows are the inclusions of the hearts.

\subsubsection{Identification of the monad}
Denote by \[ \Fun^L_0(\MGL_{c=0}\Mod, \MGL_{c=0}\Mod) \subset \Fun(\MGL_{c=0}\Mod, \MGL_{c=0}\Mod) \] the subcategory of those functors $F$ which preserve colimits and such that $F(\MGL_{c=0}\Mod^\heart) \subset \MGL_{c=0}\Mod^\heart$.
\begin{lemma} \label{lemm:fun0}
The restriction \[ \Fun^L_0(\MGL_{c=0}\Mod, \MGL_{c=0}\Mod) \to \Fun^L(\MGL_{c=0}\Mod^\heart, \MGL_{c=0}\Mod^\heart) \] is an equivalence.
\end{lemma}
\begin{proof}
By \cite[Lemma 3.2]{sosnilo2017theorem} we have \[ \Fun^L(\MGL_{c=0}\Mod, \MGL_{c=0}\Mod) \wequi \Fun^\oplus(\PM_\MGL(k), \MGL_{c=0}\Mod), \]
where $\Fun^\oplus$ denotes the category of biproduct preserving functors.
Therefore, we deduce that \[ \Fun^L_0(\MGL_{c=0}\Mod, \MGL_{c=0}\Mod) \wequi \Fun^\oplus(\PM_\MGL(k), \MGL_{c=0}\Mod^\heart). \]
This latter category identifies with $\Fun^L(\PSh_\Sigma(\PM_\MGL(k)), \MGL_{c=0}\Mod^\heart)$ by \cite[Proposition 5.5.8.15]{HTT} and \cite[Lemma 2.8]{bachmann-norms}\NB{more direct reference?}.
Finally we have \[ \Fun^L(\PSh_\Sigma(\PM_\MGL(k)), \MGL_{c=0}\Mod^\heart) \wequi \Fun^L(\PSh_\Ab(\PM_\MGL(k)), \MGL_{c=0}\Mod^\heart), \] since $\MGL_{c=0}\Mod^\heart$ is a $1$-category and $\PM_\MGL(k)$ is additive.
This concludes the proof.
\end{proof}

\begin{corollary} \label{cor:comonad-discrete}
Restriction induces an equivalence \begin{gather*} \{ \text{cocontinuous comonads on $\MGL_{c=0}\Mod$ preserving the heart} \} \\ \wequi \{ \text{cocontinous comonads on $\MGL_{c=0}\Mod^\heart$} \}. \end{gather*}
\end{corollary}
\begin{proof}
Since comonads on $\scr D$ are by definition given by $\Alg(\Fun(\scr D, \scr D)^\op)$, this follows from the fact that the restriction equivalence of Lemma \ref{lemm:fun0} is compatible with the composition monoidal structures\NB{even on the level of simplicial sets...}.
\end{proof}
Under the above equivalence, the comonad $C$ corresponds to its restriction to the heart, which we denote by $C^\heart$ or also by $C$ when it is clear in the context.
We describe this restriction.

By Equation (\ref{eq:graded_Hom}), $[X\{i\}, Y\{j\}]\simeq [X\{0\}, Y\{j-i\}]$. Therefore we can view $[X\{*\}, Y\{*\}]$ as a single graded group by taking the first $*$ to be $0$.
Observe that for smooth proper varieties $X, Y$, the mapping set, $[X\{0\}, Y\{*\}]_{\PM_\MGL(k)}$, is an $\MU_{2*}\MU$-comodule; indeed we have seen that up to some shift in degrees it identifies with $\MGL^{2*,*}(X \times Y)$.
In other words, for any graded $\MGL$-correspondence $\alpha: X \to Y \in [X\{0\}, Y\{*\}]_{\PM_\MGL(k)}$ we obtain \begin{gather*} \Delta(\alpha) = \sum_i p_i \otimes \alpha_i \\ \in \MU_{2*}\MU \otimes_{\MU_{2*}} \MGL_{2*,*}(X \times Y) \\ \wequi \MU_{2*}\MU \otimes_{\MU_{2*}} [X\{0\}, Y\{*\}]_{\PM_\MGL(k)}. \end{gather*}

Recall our description of $\MGL_{c=0}\Mod^\heart$ in Remark \ref{rmk:MGL-c0-mod-explicit-heart}.
Let $F \in \MGL_{c=0}\Mod^\heart$; thus $F$ is a kind of presheaf on smooth proper varieties together with some extra data, namely an action by $\MGL$-correspondences.
We wish to describe $CF \in \MGL_{c=0}\Mod^\heart$, again this is a presheaf with an action by $\MGL$-correspondences.

\begin{remark}
We have $CF = \MGL_{c=0} \wedge_{\1_{c=0}} F$, which has \emph{two} structures as an $\MGL_{c=0}$-module.
Since the underlying spectra are the same, the right module structure has the same value on sections as the left module structure when viewed as an object of $\PSh_\SH(\PM_\MGL(k))$, however the action by graded $\MGL$-correspondences differs.
The ``correct'' action is on the left, and given a correspondence $\alpha: X \to Y$ we denote it by $\alpha_L^*: CF(Y) \to CF(X)$.
\end{remark}

\begin{proposition} \label{prop:C-heart-explicit} \hfill
\begin{enumerate}
\item Given $F \in \MGL_{c=0}\Mod^\heart \wequi \PSh_\Ab(\PM_\MGL(k))$, the object $CF$ is given on sections by \[ (CF)(X)_* = \MU_{2*}\MU \otimes_{\MU_{2*}} F(X)_*. \]
  Given an $\MGL$-correspondence $\alpha: X \to Y$, the action $\alpha_L^*: CF(Y)_* \to CF(X)_*$ is given by $\Delta(\alpha)^*$.
  In other words for $s \in F(Y)$ and $p \in \MU_{2*}\MU$ we have \[ \alpha_L^*(p\otimes s)=\sum_i pp_i \otimes \alpha_i^*(s), \] in the notation for $\Delta(\alpha)$ of above.
\item The counit map $CF \to F$ is given on sections by $p \otimes s \mapsto \epsilon(p)s$, where $\epsilon$ is the counit of the Hopf algebroid $(\MU_{2*}, \MU_{2*}\MU).$
\item The comultiplication map $CF \to C^2F$ is given on sections by $p \otimes s \mapsto \Delta(p) \otimes s$.
\end{enumerate}
\end{proposition}
\begin{proof}
(1) Given $F \in \MGL_{c=0}\Mod$ we have \[ CF = \MGL_{c=0} \wedge_{\1_{c=0}} F \wequi (\MGL_{c=0} \wedge_{\1_{c=0}} \MGL_{c=0}) \wedge_{\MGL_{c=0}} F. \]
Since also $\MGL_{c=0} \wequi \MGL \wedge \1_{c=0}$ (Corollary \ref{cor:pure}) we get \[ \MGL_{c=0} \wedge_{\1_{c=0}} \MGL_{c=0} \wequi (\MGL \wedge \MGL) \wedge \1_{c=0} \wequi \MGL_{c=0}[a_1, a_2, \dots], \] where $\MU_{2*}\MU = \MU_*[a_1, a_2, \dots]$ (use \cite[Lemma 6.4]{naumann2009motivic}).
The description of the sections of $CF$ follows.

To obtain the description of $\alpha_L^*$, we apply Remark \ref{rmk:PM-A} with $A=\MGL_{c=0}$, $M=CF$, $B=\MGL_{c=0} \wedge_{\1_{c=0}} \MGL_{c=0}$ and $u$ the map inserting a unit on the right.
Part (1) of the remark implies that $\alpha_L^*$ is $\MU_{2*}\MU$-linear.
Thus to determine $\alpha_L^*(p \otimes s)$, we may assume that $p=1$.
We find that $B \wedge_A F \wequi CF$, $u_{F*}(s) = 1 \otimes s$ and $u_*(\alpha) = \Delta(\alpha)$.
Part (2) of the remark thus asserts the claimed formula $\alpha_L^*(1 \otimes s) = \Delta(\alpha)^*(1 \otimes s)$.

(2) The counit map is the module action \[ CF = \MGL \wedge F \to F, \] which under $\MGL \wedge F \wequi (\MGL \wedge \MGL) \wedge_{\MGL} F$ corresponds to the multiplication map $\MGL \wedge \MGL \to \MGL$ followed by the old module action on $F$.
The result follows since the multiplication in $\MGL$ corresponds to the counit in the associated Hopf algebroid.

(3) By definition, the comultiplication $\MGL \wedge \MGL \to \MGL \wedge \MGL \wedge \MGL$ is by inserting a unit in the middle. The proof is similar to (2). 
\end{proof}

\begin{example}
\label{example:cellular_case}
In the special case when $\alpha \in \MU_{2*} = [\1\{0\}, \1\{*\}]_{\PM_\MGL(k)}$ we get $\Delta(\alpha) = \eta_L(\alpha) \otimes \id$, where $\eta_L$ is the left unit of the Hopf algebroid $(\MU_{*}, \MU_{*}\MU)$. Therefore the $\MU_{2*}$-module structure on $(CF)(X)_*$ is indeed the left $\MU_{2*}$-module structure on $\MU_{2*}\MU \otimes_{\MU_{2*}} F(X)_*$.
\end{example}

\subsection{The Chow heart} \label{subsec:chow-heart}
Using Corollary \ref{corr:heart-abstract} and Proposition \ref{prop:C-heart-explicit}, we can now describe the category $\SH(k)^\heart$ explicitly, by spelling out what modules under the comonad $C^\heart$ mean in terms of Remark \ref{rmk:MGL-c0-mod-explicit-heart}. Namely, we have $\SH(k)^\heart \wequi C^\heart\CoMod$, where we view $C^\heart$ as a comonad on $\PSh_\Ab(\PM_\MGL(k))$.
An object $F \in C^\heart\CoMod$ is given by the following data:
\begin{itemize}
\item for every smooth proper variety $X$ a graded $\MU_{2*}\MU$-comodule $F(X)_*:=[\Sigma^{2*,*}X_+,F\wedge \MGL]$, and
\item for every graded $\MGL$-correspondence $\alpha: X \to Y$ an $\MU_{2*}$-linear map $\alpha^*: F(Y)_* \to F(X)_*$,
\end{itemize}
subject to the following conditions:
\begin{itemize}
\item $(\alpha \beta)^* = \beta^* \alpha^*$, $(\alpha + \beta)^* = \alpha^* + \beta^*$, $\id^* = \id$, $0^* = 0$, and
\item for every graded $\MGL$-correspondence $\alpha: X \to Y$ the following diagram commutes
\begin{equation*}
\begin{CD}
F(Y)_* @>{\alpha^*}>> F(X)_* \\
@V{\Delta_{F,Y}}VV  @V{\Delta_{F,X}}VV \\
\MU_{2*}\MU \otimes_{\MU_{2*}} F(Y)_* @>{\Delta(\alpha)^*}>> \MU_{2*}\MU \otimes_{\MU_{2*}} F(X)_*.
\end{CD}
\end{equation*}
\end{itemize}
A morphism from $F$ to $G$ consists of an $\MU_{2*}\MU$-comodule map $F(X)_* \to G(X)_*$ for every smooth proper variety $X$, compatible with pullback along graded $\MGL$-correspondences in the evident way.\\


We thank an anonymous referee for pointing out the following.
\begin{remark}
\label{rem:concise_description}
As we have observed above, the mapping sets in $\PM_\MGL(k)$ are naturally $\MU_{2*}\MU$-comodules.
In fact this makes $\PM_\MGL(k)$ into a category \emph{enriched} in $\MU_{2*}\MU$-comodules.
The compatibility condition displayed above precisely means that $\SH(k)^\heart[1/e]$ is equivalent to the category of \emph{enriched} presheaves on $\PM_\MGL(k)$ (see e.g. \cite[\S3.5]{riehl2014categorical}).
\end{remark}

We can also prove the following.
\begin{proposition}
There is a canonical\tom{For simplicity, I have removed ``symmetric monoidal''. I'm sure we can prove it if we want to.} equivalence \[ \1_{c=0}\Mod[1/e]_b \wequi D^b(\SH(k)^\heart). \]
\end{proposition}
\begin{proof}
We implicitly invert $e$ throughout.
We know by Remark \ref{rmk:MGL-c-D} that $\MGL_{c=0}\Mod \wequi D(\MGL_{c=0}\Mod^\heart)$.
Consider the comonadic adjunction \[ F: \1_{c=0}\Mod_b \adj \MGL_{c=0}\Mod_b \wequi D^b(\MGL_{c=0}\Mod^\heart): R. \]
The functor $F$ is $t$-exact and conservative by Proposition \ref{prop:comonadic-reconstruction}, and $R$ is $t$-exact since $\MGL_{c=0} \in \SH(k)_{c \ge 0}$.
The category $\MGL_{c=0}\Mod^\heart$ has enough injectives, being a presheaf category.
Since $F$ is a conservative exact left adjoint, the category $\1_{c=0}\Mod^\heart$ also has enough injectives, namely those of the form $RI$ for $I \in \MGL_{c=0}\Mod^\heart$ injective.
To conclude, by the dual of \cite[Proposition 1.3.3.7]{HA}, it suffices to prove that if $X \in \1_{c=0}\Mod^\heart$ then $[X, \Sigma^i RI]_{\1_{c=0}\Mod_b} = 0$ for $i > 0$.
But this is the same as $[FX, \Sigma^i I]_{D^b(\MGL_{c=0}\Mod^\heart)}$, which vanishes since $F$ is $t$-exact and $I$ is injective.
\NB{the proof seems to work more generally for any exact comonad on an abelian category with enough injectives...}
\end{proof}

\subsection{$W$-Cellular objects} \label{subsec:cellular-objects}
\begin{definition}
\label{def:wcellular}
	Let $W$ be a set of smooth proper schemes over $k$ that contains $Spec (k)$ and is closed under finite products. Define the 
	\emph{$W$-cellular category}, denoted by $\SH(k)^\wcell$ to be the subcategory of $\SH(k)$ generated under taking colimits and desuspensions by objects of the form $\Th(\xi)$ for $\xi\in K(X)$ and $X\in W$.
\end{definition}

\begin{remark}
\label{rmk:wcellular_special_case}
When $W$ contains only $Spec (k)$, this specializes to the cellular subcategory in \S \ref{subsec:cellular-chow-t-str}. Since this is the smallest choice, all cellular objects are $W$-cellular for any choice of $W$.
\end{remark}
 
Fix $W$. In the discussion below, cellular means $W$-cellular.

\subsubsection{The $W$-cellular Chow $t$-structure}

Given a stable presentable category $\scr D$ under $\SH(k)$, we denote by $\scr D^{\wcell}$ the subcategory generated under colimits and desuspensions by the images of $\SH(k)^\wcell$.
Similarly we write $\scr D^{\wcell}_{\ge 0}$ for the subcategory generated under colimits and extensions by images of the objects of the form $\Th(\xi)$ for $\xi\in K(X)$ and $X\in W$.
This defines a $t$-structure on $\scr D^{\wcell}$, whose truncation we denote by $\tau_{c \le 0}^{\wcell}, \tau_{c \ge 0}^{\wcell}, \tau_{c = 0}^{\wcell}$. We call it cellular Chow $t$-structure.
The inclusion $\scr D^{\wcell} \hookrightarrow \scr D$ admits a right adjoint, called the cellularization functor $\scr D\to \scr D^{\wcell}$. In the following discussion, we take $\scr D$ to be $\SH(k)$ or $\SH(k)\xrightarrow{\wedge A} A\Mod$ for some $A\in \CAlg(\SH(k)).$

We have the following results in the cellular case, analogous to Lemma \ref{lem:smash-with-thom}.

\begin{lemma} \label{lem:smash-with-thom-cell}
Let $E \in \SH(k)^\wcell$.
If $\pi_{i,0}(E\wedge \Th(\xi))=0$ for all $i \in \Z$ and $K$-theory points $\xi$ on smooth proper varieties $X\in W$, then $E \wequi 0$.
\end{lemma}

\begin{proof}
	The proof is similar to that of Lemma \ref{lem:smash-with-thom}.
\end{proof}

Let $E\mapsto E^{\wcell}$ denote the cellularization of $E\in \SH(k).$ A key fact is as follows.
\begin{lemma} \label{lemm:cellularization-t-exact}
The cellularization functor $\SH(k) \to \SH(k)^{\wcell}$ is $t$-exact for the Chow $t$-structures. Equivalently for $E \in \SH(k)_{c \ge 0}$ we have $E^{\wcell} \in \SH(k)^{\wcell}_{c \ge 0}$.
\end{lemma}
\begin{proof}
The equivalence follows from the fact that $\SH(k)^{\wcell}_{c \ge 0} \subset \SH(k)_{c \ge 0}$, whence cellularization is always left-$t$-exact, being right adjoint to a right-$t$-exact functor.

Now let $E \in \SH(k)_{c \ge 0}$.
To prove that $E^{\wcell} \in \SH(k)^{\wcell}_{c \ge 0}$ it is enough to show that $\tau^{\wcell}_{c<0}E^{\wcell} \wequi 0$.
Similar to the proof of Proposition \ref{prop:right-complete}, the cellular Chow $t$-structure is right complete. Therefore we only need to show that $\tau^{\wcell}_{c=i}E^{\wcell} \wequi 0$ for all $i<0$. 

By Lemma
\ref{lem:smash-with-thom-cell}, it suffice to show $\pi_{s,0}(\tau^{\wcell}_{c=i}E^{\wcell}\wedge \Th(\xi))=0$ for all $s \in \Z$ and $K$-theory points $\xi$ on smooth proper varieties $X\in W$.
Theorem \ref{thm:main-computation} applied to $I=\SH(k)^{\wcell}_{c \ge 0}$ shows that $$\pi_{*,*}(\tau^{\wcell}_{c=i}E^{\wcell}\wedge \Th(\xi)) \wequi \Ext^{*,*}_{\MU_{2*}\MU}(\MU_{2*},\MGL_{2*+i,*}E^{\wcell}\wedge \Th(\xi)).$$
Here we also use the cellular version of Proposition \ref{prop:pure} which gives $(\tau^{\wcell}_{c=i}E^{\wcell})\wedge \Th(\xi)\simeq \tau^{\wcell}_{c=i}(E^{\wcell}\wedge \Th(\xi))$.
Since $\MGL$ is cellular, we have $\MGL_{2*+i,*}E^{\wcell}\wedge \Th(\xi) \wequi \MGL_{2*+i,*} E\wedge \Th(\xi)$, which vanishes by Proposition \ref{prop:chow-MGL-homology}(2). Therefore when $i\leq 0$, the bigraded homotopy groups of $\tau^{\wcell}_{c=i}E^{\wcell}\wedge \Th(\xi)$ all vanish. The result follows. 
\end{proof}

\begin{lemma} \label{lemm:t-exact-escalation}
Let $A \in \CAlg(\SH(k)_{c \ge 0})$ and consider the following commutative diagram of left adjoint functors
\begin{equation*}
\begin{tikzcd}
\SH(k) \ar[r,"\wedge A^{\wcell}"] &A^{\wcell}\Mod \ar[r,"\wedge_{A^{\wcell}}A"] & A\Mod\\
\SH(k)^{\wcell} \ar[r,"\wedge A^{\wcell}"] \ar[u, hook] &A^{\wcell}\Mod^{\wcell} \ar[r,"\wedge_{A^{\wcell}}A", "a"'] \ar[u,hook] & A\Mod^{\wcell} \ar[u,hook]
\end{tikzcd}
\end{equation*}
All of the right adjoints are $t$-exact, all the horizontal right adjoints are conservative, and the bottom right horizontal functor $a$ is an equivalence.
\end{lemma}
\begin{proof}
If $F: \scr C \to \scr D$ is a functor of presentable stable $\infty$-categories such that $\scr D$ is generated under colimits by the essential image of $F$, and $F$ admits a right adjoint $G$, then $G$ is conservative.
Indeed given $\alpha: X \to Y \in \scr D$ such that $G\alpha$ is an equivalence, the class of objects $T \in \scr D$ such that $\Map(T, \alpha)$ is an equivalence is closed under colimits and contains the essential image of $F$, hence is all of $\scr D$, and so $\alpha$ is an equivalence.
Conservativity of the horizontal right adjoints follows from this observation, taking $F$ to be a horizontal left adjoint.
This also means that $a$ is an equivalence if and only if it is fully faithful, which we may test on the compact generators $\Sigma^{p,q}A^{\wcell} \wedge \Th(\xi)$, for $K$-theory points $\xi$ on smooth proper varieties $X\in W$.
In other words we need to show that $\pi_{*,*}(A^{\wcell}\wedge \Th(\xi)) \wequi \pi_{*,*}(A\wedge \Th(\xi))$. This holds by definition (and dualizability of $\Th(\xi)$).

We prove the $t$-exactness of the right adjoints. Since all the left adjoints are right $t$-exact, the right adjoints are left $t$-exact. Therefore it suffices to show that all the right adjoints are right $t$-exact.

First we prove the statement for the top horizontal arrows. We denote the functors by notations as in the following diagram
\begin{equation*}
\begin{tikzcd}
	F: \SH(k) \ar[r,"F'",yshift={.5ex}]\ar[r,leftarrow, "U'"', yshift={-0.5ex}] 
	& A^{\wcell}\Mod \ar[r,"L",yshift={.5ex}] \ar[r,leftarrow,"R"', yshift={-0.5ex}] 
	& A\Mod: U
\end{tikzcd}
\end{equation*}
where $F$ and $U$ are the composites.
%
By assumption $A\in \SH(k)_{c\geq 0}.$ Thus for any $E\in \SH(k)_{c \ge 0} $, $UFE \wequi A \wedge E \in \SH(k)_{c \ge 0}$. Since the non-negative part of $A\Mod$ is generated by $F(\SH(k)_{c \geq 0})$ and $U$ preserves colimits (its left adjoint preserving compact generators), $U$ is right-$t$-exact.
By Lemma \ref{lemm:cellularization-t-exact}, \[ A^{\wcell} \in \SH(k)^{\wcell}_{c \ge 0} \subset \SH(k)_{c \ge 0}.\]
We may thus apply the same argument to $A^{\wcell}\Mod$, and deduce that the $U'$ is $t$-exact. Now we show the right $t$-exactness of $R$. Given $E\in A\Mod_{\geq 0}$, by $t$-exactness of $U$ and $U'$, we have $0\simeq U(\tau_{ <0}E)\simeq (U'R(E))_{<0}\simeq U' \tau_{<0}R(E).$ Using conservativity of $U{'}$, we conclude that $\tau_{< 0}R(E)\simeq 0$ and equivalently $R(E)\in A^{\wcell}\Mod_{\ge 0}.$

The same argument proves the $t$-exactness statements for the bottom left arrow (and the bottom right arrows are $t$-exact, being equivalences).

It remains to check that the vertical right adjoints (cellularizations) are $t$-exact.
This is true for the left hand vertical right adjoint by Lemma \ref{lemm:cellularization-t-exact}.
Exactness of the other two vertical right adjoints follows formally, since a functor is $t$-exact if and only if its composite with a $t$-exact conservative functor is.
\end{proof}

\subsubsection{$W$-cellular reconstruction}

\begin{corollary} \label{cor:cellular-comonadic-technical}
Let $-\infty \le m \le n \le \infty$ and $A \in \CAlg(\SH(k)_{c \ge 0})$.
The composite functor \[ A\Mod^{\wcell}_{[m,n]} \hookrightarrow A\Mod \xrightarrow{\tau_{[m,n]}} A\Mod_{[m,n]} \] is fully faithful.
\end{corollary}
\begin{proof}
Since cellularization preserves $\ge m$ part by Lemma \ref{lemm:t-exact-escalation}, we have an adjunction \[ A\Mod^{\wcell}_{[m,\infty]} \adj A\Mod_{[m,\infty]}: (\ph)^{\wcell}. \]
Consider the following diagram of adjunction pairs:
\begin{equation*}
	\begin{tikzcd}
		A\Mod^{\wcell}_{[m,\infty]} \ar[r,hook,"",yshift={.5ex}]\ar[r,leftarrow,"(-)^{\wcell}"',yshift={-0.5ex}]\ar[d,"\tau_{\leq n}^\wcell"',xshift={-0.5ex}] & A\Mod_{[m,\infty]}\ar[d,"\tau_{\leq n}"',xshift={-0.5ex}]\\
		A\Mod^{\wcell}_{[m,n]} \ar[u,"i"',xshift={0.5ex},hook]& A\Mod_{[m,n]}\ar[u,xshift={0.5ex},hook]
	\end{tikzcd}.
\end{equation*}
We denote by $L: A\Mod^{\wcell}_{[m,\infty]}\to A\Mod_{[m,n]}$ the composite of the left adjoints and $R$ its right adjoint.
Note that the vertical adjunctions are in fact localizations, namely annihilating the $> n$ parts.
It follows that there is an induced adjunction of the localizations 
\[ \tau_{\le n}: A\Mod^{\wcell}_{[m,n]} \adj A\Mod_{[m,n]}: (\ph)^{\wcell}, \] 
where left adjoint is the composite $L\circ i$, and thus is equivalent to the functor we are required to show is fully faithful.
It thus suffices to show that $RL(E) \wequi E$ for any $E \in A\Mod^{\wcell}_{[m,n]}\subset A\Mod^{\wcell}_{[m,\infty]}$. 
Since cellularization is $t$-exact by Lemma \ref{lemm:t-exact-escalation}, we have 
$$RL(E)\simeq (\tau_{\le n} E)^{\wcell} \wequi \tau_{\le n}^{\wcell}(E^{\wcell}) \wequi \tau_{\le n}^{\wcell}(E) \wequi E. $$
This concludes the proof.
\end{proof}

\begin{corollary} \label{cor:comonadic-cellular}
Let $-\infty \le m \leq n < \infty$. 
\begin{enumerate}
\item The following diagram is a pullback square
\begin{equation*}
\begin{CD}
\1_{c=0}\Mod_{[m,n]}^{\wcell} @>>> \1_{c=0}\Mod_{[m,n]} \\
@VVV                               @VVV \\
\MGL_{c=0}\Mod_{[m,n]}^{\wcell} @>>> \MGL_{c=0}\Mod_{[m,n]}.
\end{CD}
\end{equation*}

\item The comonad $C$ on $\MGL_{c=0}\Mod$ restricts to a comonad $C^{\wcell}_{[m,n]}$ on $\MGL_{c=0}\Mod_{[m,n]}^{\wcell}$ and 
\[ C^{\wcell}_{[m,n]}\CoMod \wequi \1_{c=0}\Mod_{[m,n]}^{\wcell}. \]
\end{enumerate}
\end{corollary}
\begin{proof}
(1) In other words, given $E \in \1_{c=0}\Mod_{[m,n]}$, $E$ is in (the essential image of the fully faithful inclusion of) $\1_{c=0}\Mod_{[m,n]}^{\wcell}$ if and only if $E \wedge \MGL$ is in (the essential image of) $\MGL_{c=0}\Mod_{[m,n]}^{\wcell}$.
Since necessity is clear, we show sufficiency. \NB{Smashing with $\MGL$ instead of $\MGL_{c=0}$ over $\1_{c=0}$ ... doesn't make a difference but maybe mention somewhere?}
Since $\wedge \MGL$ is conservative on $\1_{c=0}\Mod_{[m,n]}$ as we proved in Proposition \ref{prop:comonadic-reconstruction}, it is enough to show that $E^{\wcell} \wedge \MGL \simeq E \wedge \MGL$.
Since $\MGL$ is cellular\NB{I have a hard time figuring out if this is true without $\MGL$ being cellular, but it is certainly obvious in that case} we have $E^{\wcell} \wedge \MGL \wequi (E \wedge \MGL)^{\wcell}$, and $E \wedge \MGL$ is cellular in $\MGL_{c=0}\Mod_{[m,n]}$ by assumption, so the result follows.

(2) $C$ restricts as claimed since $\MGL_{c=0} \wedge \MGL \in \MGL_{c=0}\Mod^{\wcell,\heart}$.
Now by construction, under the equivalence $C_{[m,n]}\CoMod \wequi \1_{c=0}\Mod_{[m,n]}$, $C^{\wcell}_{[m,n]}\CoMod$ is equivalent to the full subcategory of $\1_{c=0}\Mod_{[m,n]}$ spanned by those objects $E$ such that $E \wedge \MGL$ is cellular in $\MGL_{c=0}\Mod_{[m,n]}$; this is $\1_{c=0}\Mod_{[m,n]}^{\wcell}$ by (1).
\end{proof}

Arguing as in the proof of Proposition \ref{prop:MGLc0-mod}, we have the following cellular analogue.

\begin{proposition}
\label{prop:MGLc0-mod-cell}
	We have a canonical $t$-exact, symmetric monoidal equivalence 
	$$\MGL_{c=0}\Mod^\wcell \wequi \PSh_\SH(\PM_\MGL^\wcell (k)), $$
	where $\PM_\MGL^\wcell (k)\subset \PM_\MGL(k)$ denote the full subcategory spanned by $X\{i\}$ for $X\in W \text{ and } i\in \Z.$ 
\end{proposition}
\begin{proof}
	The proof is similar to that of Proposition \ref{prop:MGLc0-mod}, by replacing the whole categories with the cellular categories.
\end{proof}

\subsubsection{Field extensions}




We discuss the reconstruction theorems in the cellular cases for different choices of $W$.

First, set $W$ to be $\{Spec (l)\vert  ~~ l/k \textit{ is a finite separable extension} \}$.
Let $G = Gal(k)$ be the absolute Galois group.
Recall that the stable category of genuine $G$-spectra $\SH(BG)$ \cite[Example 9.12]{bachmann-norms} admits a $t$-structure with heart the category of $G$-Mackey functors \cite[Proposition 9.11]{bachmann-norms} \[ \PSh_\Ab(\Span(\Fin_G)) =: \Mack_G. \]
Here $\Fin_G$ denotes the category of finite discrete $G$-sets.
The canonical symmetric monoidal cocontinuous functor $\SH \to \SH(BG)$ induces the ``constant Mackey functor'' \[ \Ab \to \Mack_G, A \mapsto \ul{A}. \]
This is symmetric monoidal and so preserves rings, Hopf algebroids etc.
In particular, from the usual Hopf algebroid $(\MU_{2*}, \MU_{2*}\MU)$ we obtain the constant Hopf algebroid in (graded) Mackey functors $(\ul{\MU_{2*}}, \ul{\MU_{2*}\MU})$.

\begin{corollary}
\label{cor:wcellular_heart_absolute}
Set $W$ to be $\{Spec (l)\vert  ~~ l/k \textit{ is a finite separable extension} \}$ and let $G = Gal(k)$ be the absolute Galois group.
We have 
\begin{align*} 
\SH(k)^{\AT,\heart} & \wequi \ul{\MU_{2*}\MU}\CoMod,\\ 
\1_{c=0}\Mod^\AT & \wequi \Stable(\ul{\MU_{2*}\MU}), \\ 
\MGL_{c=0}\Mod^{\AT, \heart} & \wequi \ul{\MU_{2*}}\Mod. 
\end{align*}

Here $\CoMod$, $\Stable$ and $\Mod$ are performed relative to $\Mack_G$.
\end{corollary}
\begin{proof}
Recall the cocontinuous symmetric monoidal functor $c: \SH(BG) \to \SH(k)$ \cite[Proposition 10.6]{bachmann-norms}.
For $X \in \Fin_G$ it satisfies $c(\Sigma^\infty_+ X) \wequi \Sigma^\infty_+ cX$ for some $cX$ finite {\'e}tale over $k$, and hence $c(\SH(BG)_{\ge 0}) \subset \SH(k)_{c \ge 0}$.
There is thus an induced functor \[ c: \Mack_G \wequi \SH(BG)^\heartsuit \to \MGL_{c=0}\Mod^{\AT, \heart}. \]
Let $\Mack_G^\Z$ denote the category of $\Z$-graded Mackey functors.
The invertible object $\Sigma^{2,1}\MGL_{c=0} \in \MGL_{c=0}\Mod^{\AT,\heart}$ determines a symmetric monoidal functor $c^\Z: \Mack_G^\Z \to \MGL_{c=0}\Mod^{\AT, \heart}$, which admits a right adjoint $c^*_\Z$.

We claim that $c^*_\Z(\MGL_{c=0}) \wequi \ul{\MU_{2*}}$ and $\MGL_{c=0}\Mod^{\AT, \heart} \wequi \ul{\MU_{2*}}\Mod$.

Since both sides are of the form $\PSh_\Ab$ and the functor is surjective on strongly dualizable generators, the second claim follows from the first.
Let $X$ be the finite $G$-set $G/H$ for some subgroup $H\subset G$. We compute using \S\ref{subsec:MGL-mod} that 
\begin{align*}
	c^*_\Z(\MGL_{c=0})(X)_i 
	&\wequi \Map_{\MGL_{c=0}\Mod}(c(X)(i), \MGL_{c=0}) \\
	&\wequi \MGL_{2i,i}(Spec(l^H))\simeq \MU_{2i}.
\end{align*}
This is indeed the formula expected for the sections of a constant Mackey functor.
In other words, $c^*_\Z(\MGL_{c=0})(-)_*$ and $\ul{\MU_{2*}} \in \Mack_G = \PSh_\Ab(\Span(\Fin_G))$ have canonically identical sections.
What we need to check is that given $\alpha: X \to Y \in \Span(\Fin_G)$, the pullback maps $\alpha^*$ of the two presheaves agree.
Any map in $\Span(\Fin_G)$ is a composite of a ``usual'' map $f: X \to Y \in \Fin_G$, and a ``reverse'' map $g^t: Y \to Z \in \Span(\Fin_G)$ corresponding to $g: Z \to Y \in \Fin_G$.
It suffices to separately compare pullbacks along $f$ (restriction maps) and pullbacks along $g^t$ (transfers).
Our formula for $c^*_\Z(\MGL_{c=0})(X)_*$ is actually functorial in $X \in \Fin_G$, i.e. the restriction maps are as expected.
It remains to compare the transfers.
We employ the following trick.
The category $\Fin_G$ can be made into a site, where the coverings are those maps of finite $G$-sets which are surjective (ignoring the $G$-actions); this is often called the \emph{étale topology}.
Denote by $\Shv_\Ab(\Span(\Fin_G)) \subset \PSh_\Ab(\Span(\Fin_G))$ the full subcategory of those Mackey functors whose underlying presheaves on $\Fin_G$ are étale sheaves, and by $\Shv_\Ab(\Fin_G) \subset \PSh_\Ab(\Fin_G)$ the category of étale sheaves.
Note that the underlying presheaf $\ul{\MU_{2*}} \in \PSh_\Ab(\Fin_G)$ is a sheaf in the étale topology.
By \cite[Proposition C.11, Corollary C.13]{bachmann-norms}, the forgetful functor $\Shv_\Ab(\Span(\Fin_G)) \to \Shv_\Ab(\Fin_G)$ is an equivalence. In other words, the transfers are uniquely determined by the restriction maps.

\NB{We can give a slightly slicker proof as follows: let $\SH(BG)^{\mathbb{S}}$ denote the category of sphere-graded genuine $G$-spectra. Then there is a functor $c^\Z:\SH(BG)^{\mathbb{S}} \to \SH(k)$ and $\MGL_{c=0}\Mod \wequi c_\Z^*(\MGL_{c=0})\Mod$ for formal reasons. This spectrum is identified with $\ul{\MU_{2*}}$ as before. I figured it was best not to introduce the additional complication of sphere-graded objects.}

We have thus established the third equivalence.
The other two follow formally from Proposition \ref{cor:comonadic-cellular} and Proposition \ref{prop:C-heart-explicit}, i.e. the identification of $\SH(k)^{\AT, \heart}$ with $C^\AT\CoMod$, and the explicit description of $C$, translated into the category $\MGL_{c=0}\Mod^{\AT, \heart} \wequi \ul{\MU_{2*}}\Mod$ (and similarly for $\1_{c=0}\Mod^\AT$).
\end{proof}

We have similar results for another two choices of $W$:

\begin{corollary}
\label{cor:wcellular_heart_fin_Galois_ext}
Let $l/k$ be a finite Galois extension with Galois group $G$. Let $W$ be $\{Spec (l') ~~\vert ~~ l'$ is a subextension of $l/k \}$. 
We have 
\begin{align*} 
\SH(k)^{\AT,\heart} & \wequi \ul{\MU_{2*}\MU}\CoMod,\\
\1_{c=0}\Mod^\AT & \wequi \Stable(\ul{\MU_{2*}\MU}), \\ 
 \MGL_{c=0}\Mod^{\AT, \heart} & \wequi \ul{\MU_{2*}}\Mod. 
\end{align*}
Again, everything is relative to $\Mack_G$, but this time for the \emph{finite} group $G$.
\end{corollary}
\begin{proof}
	Similarly, by \cite[Proposition 10.6]{bachmann-norms}, we have a cocontinuous symmetric monoidal functor $c_{l/k}: \SH(BG) \to \SH(k)$. The results follow from the same arguments as in the proof of Corollary \ref{cor:wcellular_heart_absolute}.
\end{proof}

As mentioned in Remark \ref{rmk:wcellular_special_case}, we can also obtain results for the ordinary cellular category by taking $W$ to be $\{Spec (k)\}$.

\begin{corollary}
\label{cor:wcellular_heart_ordinary_cellular}

We have canonical, symmetric monoidal, $t$-exact equivalences 
\begin{align*} 
\SH(k)^{\cell,\heart} & \wequi \MU_{2*}\MU\CoMod, \\ 
\1_{c=0}\Mod^\cell & \wequi \Stable(\MU_{2*}\MU),\\ 
\MGL_{c=0}\Mod^{\cell,\heart} & \wequi \MU_{2*}\Mod. 
\end{align*}

%
\end{corollary}

\begin{proof}
	It follows from Corollary \ref{cor:wcellular_heart_fin_Galois_ext} by taking the trivial extension.
\end{proof}

\section{Towards Computations of Motivic Stable Stems}
In this section, we discuss a strategy for computing differentials in the motivic Adams spectral sequence based on $\textup{H}\Z/p$ for a general motivic spectrum $X$, and in particular consequentially computing motivic stable stems. This strategy has worked very well for $\1$ over $\C$ in the work \cite{GWX, IWX, IWX2}. Since the main structural theorems in this paper work for any field $k$, whose exponential characteristic is different from the prime $p$, we can now generalize the strategy over $\C$ to the field $k$. We will first discuss the general strategy over $k$, and then take a closer look at the cases $k = \C, \mathbb{R}$, and finite fields.


\subsection{Algebraicity of the motivic Adams spectral sequence in the Chow heart}
The crucial point of the strategy for computing Adams differentials over $\C$ is the following theorem \cite[Theorem~1.3]{GWX}, describing the algebraicity of the motivic Adams spectral sequence of $\1^\comp_p/\tau$. 

\begin{theorem}[Gheorghe--Wang--Xu] \label{GWX iso ss}
For each prime p, there is an isomorphism of spectral sequences, between the $\textup{H}\Z/p$-based motivic Adams spectral sequence for $\1^\comp_p/\tau$, and the algebraic Novikov spectral sequence for $\BP_{2*}$.  
\end{theorem}

Since the algebraic Novikov spectral sequence for $\BP_{2*}$ can be completely computed by a computer program \cite{IWX2} within any reasonable range, Theorem~\ref{GWX iso ss} provides a source of nontrivial differentials in the motivic Adams spectral sequence for $\1^\comp_p/\tau$. We can then pullback these differentials along the map of spectral sequences induced by the unit map $\1^\comp_p \rightarrow \1^\comp_p/\tau$:
$$\mathbf{motASS}(\1^\comp_p) \longrightarrow \mathbf{motASS}(\1^\comp_p/\tau),$$
and push forward these differentials along the map of spectral sequences induced by the quotient map $\1^\comp_p/\tau \rightarrow \Sigma^{1, -1} \1^\comp_p$:
$$\mathbf{motASS}(\1^\comp_p/\tau) \longrightarrow \mathbf{motASS}(\Sigma^{1, -1} \1^\comp_p).$$

 As a consequence, we obtain differentials in the motivic Adams spectral sequence for $\1^\comp_p$.

Over an arbitrary field $k$, the substitute for $\1^\comp_p/\tau$ over $\C$ are the objects $\1_{c=n}$ for $n\geq 0$. They all live in a suspension of the Chow heart.  Theorem~\ref{thm: iso of ss} below is a generalization of Theorem~\ref{GWX iso ss}.

In general, let $F \in \SH(k)^\heart$. We can attempt to compute $\pi_{*,*} F$ by running the $\textup{H}\Z/p$-based motivic Adams spectral sequence.
This is the trigraded spectral sequence obtained from the graded cosimplicial spectrum \[ \map(S^{0,*}, F \wedge \CB(\textup{H}\Z/p)) \wequi \map(S^{0,*}, F \wedge_{\1_{c=0}} \CB_{\1_{c=0}}(\1_{c=0} \wedge \textup{H}\Z/p)). \]
The Hopkins--Morel isomorphism (with the exponential characteristic $e$ implicitly inverted) \cite{hoyois-algebraic-cobordism} \[ \textup{H}\Z \wequi \MGL/(a_1, a_2, \dots) \] implies that \[ \1_{c=0} \wedge \textup{H}\Z/p \wequi \1_{c=0} \wedge \MGL/(p, a_1, a_2, \dots) \] is cellular in $\1_{c=0}\Mod$, where $p \neq e$.
Under the equivalence with $\Stable(\MU_{2*}\MU)$ (see Corollary \ref{cor:wcellular_heart_ordinary_cellular}), $\1_{c=0} \wedge \textup{H}\Z/p$ thus corresponds to $\MU_{2*}\MU/(p, a_1, a_2, \dots)$.

We give a more general definition of the algebraic Novikov spectral sequence in the category of $\Stable(\MU_{2*}\MU)$.

\begin{definition}\label{def: iso of ss}
	Let $X$ be an object in $\Stable(\MU_{2*}\MU)$ and let $H\in \Stable(\MU_{2*}\MU) $ be a commutative monoid. Define the \emph{algebraic Novikov spectral sequence} based on $H$ for $X$ to be the spectral sequence associated to the cosimplicial object $X \wedge \CB(H)$ in $\Stable(\MU_{2*}\MU)$, of the form
$$E_1^{s,i,t}=\pi_s\map_{\Stable(\MU_{2*}\MU)}(\Sigma^{t} \MU_{2*}, X \wedge \mathrm{CB}^i(H))\implies  \pi_{s-i}\map_{\Stable(\MU_{2*}\MU)}(\Sigma^{t} \MU_{2*}, X),$$
with differentials $d_r: E_r^{s,i,t} \rightarrow E_r^{s+r-1,i+r,t}$. Here $\Sigma^{t} \MU_{2*}$ denote $\MU_{2*}$ with the internal degree of every element shifted positively by $t$.
A similar definition of the algebraic Novikov spectral sequence applies to the category of $\Stable(\BP_{2*}\BP)$.
\end{definition}

\begin{theorem} 
\label{thm: iso of ss_cellular}
Let $F \in \SH(k)^{\cell, \heart}$.
Write $M$ for the $\MU_{2*}\MU$-comodule associated to $F$ under the equivalence in Corollary \ref{cor:wcellular_heart_ordinary_cellular} and $H$ for $ \MU_{2*}\MU/(p, a_1, a_2, \dots)$.
The trigraded motivic Adams spectral sequence for $F$ based on $\textup{H}\Z/p$ is equivalent (with all higher and multiplicative structure) to the algebraic Novikov spectral sequence based on $H$ for $M$.
\end{theorem}
\begin{proof}
We have canonical equivalences of graded cosimplicial spectra:
\begin{align*}
\map(\1, F \wedge \CB(\textup{H}\Z/p))
\simeq & \ \map(\1, \ F \wedge_{\1_{c=0}} \CB_{\1_{c=0}}(\1_{c=0} \wedge \textup{H}\Z/p))\\
\simeq & \ \map_{\1_{c=0}\Mod}(\1_{c=0}, \ F \wedge_{\1_{c=0}} \CB_{\1_{c=0}}(\1_{c=0} \wedge \textup{H}\Z/p))& \\
\simeq & \ \map_{\Stable(\MU_{2*}\MU)}(\MU_{2*}, \ M \wedge \CB(H));
\end{align*}
here the first equivalence follows from the fact that $F$ is a module over $\1_{c=0}$, the second is by adjunction, and the third is by Corollary \ref{cor:wcellular_heart_ordinary_cellular}.
\end{proof}

More generally we have the following.
\begin{theorem} 
\label{thm: iso of ss}
Let $F \in \SH(k)^\heart$. Write $M$ for the $\MU_{2*}\MU$-comodule associated to $F^{\cell}$ under the equivalence in Corollary \ref{cor:wcellular_heart_ordinary_cellular} and $H$ for $ \MU_{2*}\MU/(p, a_1, a_2, \dots)$. The same equivalence of spectral sequences holds.
\end{theorem}
\begin{proof}
Indeed for general $F \in \SH(k)^\heart$, its $\textup{H}\Z/p$-based motivic Adams spectral sequence coincides with that for $F^\cell$.
\end{proof}

The category of $p$-local $\MU_{2*}\MU$-comodules and the category of $\BP_{2*}\BP$-comodules are equivalent as abelian categories (see \cite[Proposition~1.2.3]{Morava} for example). The comodule $\MU_{2*}\MU/(p, a_1, a_2,\dots)$ is $p$-local, and it corresponds to
$$\BP_{2*} \otimes_{\MU_{2*}} \MU_{2*}\MU/(p, a_1, a_2,\dots) \cong \BP_{2*}\BP/I$$
as a $\BP_{2*}\BP$-comodule, where $I = (p, v_1, v_2, \dots)$ is the augmentation ideal of $\BP_{2*}$.
Therefore, the algebraic Novikov spectral sequence of an $\MU_{2*}\MU$-comodule $X$ based on $H = \MU_{2*}\MU/(p, a_1, a_2, \dots)$, is equivalent to the algebraic Novikov spectral sequence of its $p$-localization $X_{(p)}$ based on $H$, and is equivalent to the algebraic Novikov spectral sequence of its corresponding $\BP_{2*}\BP$-comodule based on $\BP_{2*}\BP/I$.

\begin{corollary} \label{algNSS BP}
Let $F \in \SH(k)^\heart$. Then the trigraded motivic Adams spectral sequence for $F$ based on $\textup{H}\Z/p$ is equivalent (with all higher and multiplicative structure) to the trigraded algebraic Novikov spectral sequence for $\BPGL_{2*,*}F$ based on $\BP_{2*}\BP/I$.
\end{corollary}
\begin{proof}
Immediate from Theorem~\ref{thm: iso of ss} and the above discussion.
\end{proof}


The original definition of the algebraic Novikov spectral sequence (see \cite{Miller, Novikov}) is defined at each prime $p$ in the category of $\BP_{2*}\BP$-comodules, and only for $\BP_{2*}$. It converges to the $E_2$-page of the Adams-Novikov spectral sequence $\Ext^{*, *}_{\BP_{2*}\BP}(\BP_{2*}, \BP_{2*})$ for the ($p$-completed) sphere spectrum. It is defined by filtering the cobar complex using the powers of the augmentation ideal $I = (p, v_1, v_2, \dots)$. 

In the case of $\BP_{2*}$, our Definition~\ref{def: iso of ss} of the algebraic Novikov spectral sequence based on $\BP_{2*}\BP/I$ is equivalent to the original one. This is proved in \S9 of \cite{GWX}.

A natural question is the following.

\begin{question} \label{two def}
For a $\BP_{2*}\BP$-comodule $Y$, we have the following two definitions of the algebraic Novikov spectral sequence.
\begin{itemize}
\item The one using the filtration of the powers of the augmentation ideal $I$, and
\item the one in Definition~\ref{def: iso of ss} based on $\BP_{2*}\BP/I$.
\end{itemize} 
When are they equivalent?
\end{question}

We give a sufficient condition for such an equivalence in Proposition~\ref{prop levin index}.

We first recall the following definition in homological algebra (see e.g. \cite{Sega}).

\begin{definition}
A $\BP_{2*}$-module $Y$ \emph{has Levin index 1}, if the following maps that are induced by $I^{n}\rightarrow I^{n-1}$ are zero for all $n \ge 1$:
$$\Ext^{{*,*}}_{\BP_{2*}}(I^{n-1}Y, \mathbb{F}_p) \rightarrow \Ext^{{*,*}}_{\BP_{2*}}(I^{n}Y, \mathbb{F}_p),$$
where $I=(p, v_1, v_2, \dots)$.	
\end{definition}

It is proved in \cite{Sega} and \cite[Lemma~9.7]{GWX} that $\BP_{2*}$ itself has Levin index 1.

\begin{proposition} \label{prop levin index}
Let $Y$ be a $\BP_{2*}\BP$-comodule whose underlying $\BP_{2*}$-module has Levin index 1.

Then the two definitions of the algebraic Novikov spectral sequence for $Y$ in Question~\ref{two def} are equivalent.
\end{proposition}

\begin{proof}
  The proof is essentially identical to the proof of Theorem~8.3 of \cite{GWX}, where the technical condition that $\BP_{2*}$ has Levin index 1 (Lemma~9.7 of \cite{GWX}) is replaced by the general condition that $Y$ has Levin index 1.
\end{proof}

\begin{remark}
In general, without the condition that $Y$ has Levin index 1, the two definitions of the algebraic Novikov spectral sequences are not equivalent. In fact, for $F \in \SH(k)^\heart$, suppose that $\BPGL_{{2*,*}}F$ has Levin index at least 2. Then the algebraic Novikov spectral sequence for $\BPGL_{{2*,*}}F$ defined by powers of $I$ is isomorphic to a modified version of the motivic Adams spectral sequence in the sense of \cite[\S3]{BHHM}.
\end{remark}


\subsection{The strategy} \label{strategy}
For a motivic spectrum $X$ with the property that $X \simeq X_{c \geq 0}$, we consider its Postnikov--Whitehead tower with respect to the Chow $t$-structure.

\begin{center}
\begin{tikzcd}
& \arrow[d] &\\
& {X}_{c \geq 2} \arrow[d] \arrow[r] & {X}_{c=2}   \\
& {X}_{c \geq 1} \arrow[d] \arrow[r] & {X}_{c=1}  \\
X \arrow[r, equal] & X_{c \geq 0} \arrow[r] & {X}_{c=0}
\end{tikzcd}
\end{center}

\begin{remark} \label{rem chow tower}
\begin{itemize}
\item
By Corollary~\ref{cor:chow-MGL-homology} and for degree reasons, if we take $\MGL_{{*,*}}$ of each term of this Postnikov--Whitehead tower, the corresponding Chow spectral sequence computing $\MGL_{{*,*}}X$ collapses at its $E_1$-page. 
\item
If we take $\pi_{{*,*}}$ of each term of this Postnikov--Whitehead tower, by Theorem~\ref{thm:main-computation}, we have
$$\pi_{*,*}{X}_{c=n} \wequi \Ext^{*,*}_{\MU_{2*}\MU}(\MU_{2*}, \MGL_{2*+n,*}X).$$
The $E_1$-page of the corresponding tower spectral sequence computing $\pi_{{*,*}}{X}$ is isomorphic to the $E_2$-page of the motivic Adams--Novikov spectral sequence. Over $\mathbb{C}$, for the case $X  = \1^\comp_p$, it was proved in \cite{Isa} that after a re-grading, these two spectral sequences are isomorphic. We suspect that this is an isomorphism of spectral sequences after a re-grading over a general field.

\end{itemize}
\end{remark}

Now consider the motivic Adams spectral sequences based on $\textup{H}\Z/p$ for every term in the Postnikov--Whitehead tower with respect to the Chow $t$-structure, and the induced maps among them. This gives us the following tower of motivic Adams spectral sequences.
\begin{center}
\begin{tikzcd}
& \arrow[d] &\\
& \mathbf{motASS}({X}_{c \geq 2}) \arrow[d] \arrow[r] & \mathbf{motASS}({X}_{c=2}) \arrow[r, equal] & \mathbf{algNSS}(\BPGL_{{2*+2, *}}X)   \\
& \mathbf{motASS}({X}_{c \geq 1}) \arrow[d] \arrow[r] & \mathbf{motASS}({X}_{c=1}) \arrow[r, equal] & \mathbf{algNSS}(\BPGL_{{2*+1,*}}X)  \\
\mathbf{motASS}({X}) \arrow[r, equal] & \mathbf{motASS}({X_{c \geq 0}}) \arrow[r] & \mathbf{motASS}({X}_{c=0}) \arrow[r, equal] & \mathbf{algNSS}(\BPGL_{{2*, *}}X)
\end{tikzcd}
\end{center}

We now have a strategy for computing differentials in the motivic Adams spectral sequence for $X$ based on $\textup{H}\Z/p$. The input of our computations is $\BPGL_{*,*}X$ and $\textup{H}\Z/p_{*,*}X$ over the base field $k$.

For every $\BP_*\BP$-comodule $\BPGL_{{2*+n,*}}X$, its algebraic Novikov spectral sequence can be computed in a large range by a computer program. This gives many nonzero differentials in the motivic Adams spectral sequence for ${X}_{c=n}$. We may then pullback these Adams differentials to the motivic Adams spectral sequences for ${X}_{c \geq n}$, and push forward them to the motivic Adams spectral sequences for ${X}$.

Here are detailed steps of this strategy.

\begin{enumerate}

\item Compute $\Ext_A^{*,*,*}(\textup{H}\Z/p_{*,*}, \textup{H}\Z/p_{*,*}X)$ over the $k$-motivic Steenrod algebra by a computer program. These $\Ext$-groups consist of the $E_2$-page of the motivic Adams spectral sequence for ${X}$ based on $\textup{H}\Z/p$.

\begin{itemize}
\item Over $\C$ for the spectrum $\1^\comp_2$, this was computed in a large range in \cite{IWX, IWX2}.

\item Over $\mathbb{R}$ for the spectrum $\1^\comp_2$, this was computed in a range in \cite{DugIsa, BelIsa} using the $\rho$-Bockstein spectral sequence and computations over $\C$. This $\rho$-Bockstein computation can also be automated.

\item Over finite fields for the spectrum $\1^\comp_2$, this was computed in \cite{WilsonOst} by a long exact sequence with the other two terms the $\C$-motivic $\Ext$. The boundary homomorphism corresponds to the Steenrod operation action on powers of $\tau$.


\end{itemize} 

\item Compute by a computer program the algebraic Novikov spectral sequence based on $\BP_{2*}\BP/I$ for the $\BP_{2*}\BP$-comodules $\BPGL_{2*+n,*}X$ for every $n$ in a reasonable range. The structures of these comodules $\BPGL_{2*+n,*}X$ depend on each field $k$ for $n \geq 1$. For $n = 0$, $\BPGL_{2*+n,*} = \BP_{2*}$ after $p$-completion for all fields $k$. The case for $\BP_{2*}$ was computed in a large range in \cite{IWX, IWX2}. This computation includes all differentials, and the multiplicative structure of the cohomology of the Hopf algebroid $(\BP_{2*}, \BP_{2*}\BP)$.

\item Identify the $k$-motivic Adams spectral sequence based on $\textup{H}\Z/p$ for each $X_{c=n}$, with the algebraic Novikov spectral sequence based on $\BP_{2*}\BP/I$ for $\BPGL_{2*+n,*}X$, for $n \geq 0$. This computation includes an identification of the abutments with the multiplicative (and higher) structures.

\item Compute the mod $p$ motivic homology of $X_{c \geq n}$ using the universal coefficient spectral sequence (see Propositions~7.7 and 7.10 of \cite{dugger2005motivic}).
$$\bigoplus_{k \geq n}\textup{Tor}_{*,*}^{\BP_{2*}}(\BPGL_{2*+k,*}X, \ \Z/p) \implies \pi_{*,*} X_{c \geq n} \wedge \textup{H}\Z/p.$$
Here we use the fact that $\1_{c=0} \wedge \BPGL/(p, v_1, \dots)$ is a cellular module over $\1_{c=0} \wedge \BPGL$ and the  equivalences 
\begin{align*}
X_{c \geq n} \wedge \textup{H}\Z/p & \simeq X_{c \geq n} \wedge \BPGL/(p, v_1, \dots)\\ 
& \simeq (X_{c \geq n} \wedge \BPGL) \wedge_{(\1_{c = 0} \wedge \BPGL)} (\1_{c=0} \wedge \BPGL/(p, v_1, \dots)),
\end{align*}
and the isomorphisms
\begin{align*}
\pi_{*,*} \1_{c=0} \wedge \BPGL & \cong \BP_{2*},\\
\pi_{*,*}X_{c \geq n} \wedge \BPGL & \cong \bigoplus_{k \geq n}\BPGL_{2*+k,*}X,\\
\pi_{*,*} \1_{c=0} \wedge \BPGL/(p, v_1, \dots) & \cong  \Z/p.
\end{align*}

\item Compute by a computer program the $E_2$-pages of the motivic Adams spectral sequence for ${X_{c \geq n}}$ based on $\textup{H}\Z/p$, using the computation of $\pi_{*,*} X_{c \geq n} \wedge \textup{H}\Z/p$ in step (4).

\item Pull back motivic Adams differentials for $X_{c=n}$ to motivic Adams differentials for $X_{c \geq n}$, and then push forward to motivic Adams differentials for $X$.

\item Deduce additional Adams differentials for $X$ with a variety of ad hoc arguments. The most important methods are Toda bracket shuffles and comparison to known results in the $\C$-motivic Adams spectral sequence.

\end{enumerate}
Most of the computations in steps (1) -- (5) can be automated. In the rest of this section, we will elaborate on step $(6)$ of this strategy, in the cases of $X = \1$ and $k = \C, \ \mathbb{R}$ and finite fields. 

\subsection{Continuous Adams spectral sequences} \label{CTS ASS}
For the sphere spectrum $\1$, the inputs of our strategy are $\textup{H}\Z/p_{*,*}$ and $\BPGL_{*,*}$. In practice, it is more plausible to know $\pi_{*,*}(\BPGL_p^\comp)$ rather than $\BPGL_{*,*}$. In order to take the $p$-completion into account, we consider pro-objects as a substitute.

\begin{definition}
Let $(X_i)_{i \in I}$ be a pro-object in $\SH(k)$.
We call the cosimplicial object \[ \Delta \ni n \mapsto \lim_i [X_i \wedge (\textup{H}\Z/p)^{\wedge n}] \] the \emph{continuous $\textup{H}\Z/p$-resolution} of $(X_i)_{i \in I}$, and the associated spectral sequence the \emph{continuous Adams spectral sequence} of $(X_i)_{i \in I}$.
\end{definition}

\begin{remark} \label{rmk:pro-p-cts}
Given $X \in \SH(k)$, write \[ F_p(X) = (\dots \to X \xrightarrow{p} X \xrightarrow{p} X) \in \Pro(\SH(k)), \] and $(X/p^n)_n \in \Pro(\SH(k))$ for the pro-$p$-completion.
Then there is a fiber sequence \[ F_p(X) \to X \to (X/p^n)_n, \] where in the middle we mean the constant pro-object.
Thus if $Y \in \SH(k)$ such that $p: Y \to Y$ is equivalent to the trivial map, then $F_p(X) \wedge Y \wequi 0$ in $\Pro(\SH(k))$ and so $X \wedge Y \wequi (X/p^n \wedge Y)_n $ in $ \Pro(\SH(k))$.
In particular the continuous $\textup{H}\Z/p$-resolution of the pro-$p$-completion $(X/p^n)_n$ is just the usual $\textup{H}\Z/p$-resolution of $X$.
\end{remark}

We now consider the following tower of pro-objects \[ \dots \to ((\1/p^n)_{c \ge i})_n \to \dots \to ((\1/p^n)_{c \ge 1})_n \to ((\1/p^n)_{c \ge 0})_n \wequi (\1/p^n)_n. \]
The layers are given by \[ ((\1/p^n)_{c=i})_n \in \Pro(\SH(k)^{\heart}). \]
Their continuous Adams spectral sequences are related in the way outlined above to the continuous Adams spectral sequence for $((\1/p^n)_{c \ge 0})_n$, 
which is just the pro-$p$-completion of $\1$, and so by Remark \ref{rmk:pro-p-cts} we recover information about the usual motivic Adams spectral sequence.
It remains to identify the continuous Adams spectral sequence of the layers; we can achieve this under some simplifying assumptions.
 
\begin{proposition}
Suppose that $\pi_{*,*} \BPGL_p^\comp$ has bounded $p$-torsion.
Put \[ B^i_* = \pi_{2*+i,*} \BPGL_p^\comp, \] viewed as a $\BP_{2*}\BP$-comodule.
Then $(\pi_{*,*}(\1/p^n)_{c=i})_n$ corresponds to the pro-$p$-completion of $B_*^i \in \Stable(\MU_{2*}\MU)$.
In particular the continuous Adams spectral sequence of $((\1/p^n)_{c=i})_n$ is isomorphic to the algebraic Novikov spectral sequence of $B^i_*$.
\end{proposition}
\begin{proof}
We may replace $\BP$ by $\MU$ throughout.
We know that $(\1/p^n)_{c=i} \in \SH(k)^{\cell, \heart}$ corresponds to the $\MU_{2*}\MU$-comodule $X^i_n = \pi_{2*+i,*} (\MGL/p^n)$.
Let $K^i_n$ be the kernel of multiplication by $p^n$ on $B^i_*$.
Then we have a short exact sequence \[ 0 \to (B^i/p^n)_n \to (X^i_n)_n \to (K^{i-1}_n)_n \to 0 \] of pro-objects in $\MU_{2*}\MU$-comodules.
The bonding maps in $(K^{i-1}_n)_n$ are given by multiplication by $p$, and since the $p$-torsion is bounded, sufficiently long composites of bonding maps are zero.
It follows that $(K^{i-1}_n)_n \wequi 0$ as a pro-object, and so $(B^i/p^n)_n \wequi (X^i_n)_n$ as pro-objects.
Let \[Y^i_n = cof(B^i \xrightarrow{p^n} B^i) \in \Stable(\MU_{2*}\MU); \] in other words $(Y^i_n)_n$ is the pro-$p$-completion of $B^i$ in the derived category of comodules.
We have a cofiber sequence \[ \Sigma (K^{i}_n)_n \to (Y^i_n)_n \to (B^i/p^n)_n, \] and so arguing as before and combining with the previous result we find that \[ (Y^i_n)_n \wequi (B^i/p^n)_n \wequi (X^i_n)_n. \]

The last sentence follows from Remark \ref{rmk:pro-p-cts}.
\end{proof}

\subsection{Over the complex numbers}

Over the complex numbers, this strategy has been proved to be very successful. In \cite{IWX, IWX2}, using this method, the knowledge of both $\C$-motivic and classical stable homotopy groups of spheres at the prime 2 was extended from 60 to 90. It is also used in \cite[\S10]{GWX} to provide one-line proofs for many historically hard Adams differentials.

In fact, over $\C$, we have $\pi_{{*, *}}\BPGL^\comp_p = (\BP_{2*})^\comp_p [\tau] = \Z^\comp_p [\tau, v_1, v_2, \cdots]$, with $|\tau| = (0,-1), \ |v_i| = (2p^i-2, p^i-1)$. In other words, we have
\[\pi_{{2*+n, *}} \BPGL^\comp_p  \cong \begin{cases} \Sigma^{0,-m} (\BP_{2*})^\comp_p , & \text{if } n = 2m \ \text{and} \ m \geq 0, \\ 0, &\text{otherwise}, \end{cases} \]
where $\Sigma^{0,-m} (\BP_{2*})^\comp_p$ means every element in $(\BP_{2*})^\comp_p$ is shifted by bidegree $(0, -m)$.

This corresponds to the fact that
\[  \pi_{*,*}(\1^\comp_p)_{c=n} \cong \begin{cases} \pi_{*,*} \Sigma^{0,-m} \1^\comp_p/\tau, & \text{if } n = 2m \ \text{and} \ m \geq 0, \\ 0, &\text{otherwise}. \end{cases} \]

By Proposition~7.2 of \cite{GWX}, we have $\MGL_{*,*}\1^\comp_p \cong \pi_{{*, *}}\MGL^\comp_p$ over $\C$. By a similar proof to the one of Proposition~7.2 of \cite{GWX}, we have $\BPGL_{*,*}\1^\comp_p \cong \pi_{{*, *}}\BPGL^\comp_p$. 
Then the tower of motivic Adams spectral sequences based on $\textup{H}\Z/p$ becomes
\begin{center}
\begin{tikzcd}
& \arrow[d] &\\
& \mathbf{motASS}({(\1^\comp_p)}_{c \geq 4}) \arrow[d] \arrow[r] & \mathbf{motASS}({(\1^\comp_p)}_{c=4}) \arrow[r, equal] & \mathbf{algNSS}(\Sigma^{0,-2} (\BP_{{2*}})^\comp_p )   \\
& \mathbf{motASS}((\1^\comp_p)_{c \geq 2}) \arrow[d] \arrow[r] & \mathbf{motASS}({\1^\comp_p}_{c=2}) \arrow[r, equal] & \mathbf{algNSS}(\Sigma^{0,-1} (\BP_{{2*}})^\comp_p )  \\
\mathbf{motASS}({\1}^\comp_p) \arrow[r, equal] & \mathbf{motASS}({\1}^\comp_p) \arrow[r] & \mathbf{motASS}({(\1^\comp_p)}_{c=0}) \arrow[r, equal] & \mathbf{algNSS}((\BP_{{2*}})^\comp_p)
\end{tikzcd}
\end{center}

Here we may view $\Sigma^{m,n} (\BP_{2*})^\comp_p$ as an element in the derived category of $p$-completed $\BP_{2*}\BP$-comodules: It is a cochain complex concentrated in cohomological degree $2n-m$, with the comodule $\Sigma^{2n}(\BP_{2*})^\comp_p$ in that cohomological degree. Under the equivalence of the categories in Corollary~\ref{cor:wcellular_heart_ordinary_cellular}, $\Sigma^{m,n} (\BP_{2*})^\comp_p$ corresponds to $\Sigma^{m,n} \1^\comp_p/\tau$. 

In this case, all horizontal arrows are identical, up to a shift of degrees. So the information above the bottom row is not of much use. This is indeed the case in the work of \cite{IWX, IWX2}, where only the naturality of the motivic Adams spectral sequences of the unit map $\1^\comp_p \rightarrow \1^\comp_p/\tau$ was used.


\subsection{Over the real numbers} \label{subsec:real-computation}

Over the real numbers, at the prime 2, recall that we have
$$\pi_{{*,*}}\textup{H}\Z/2 = \Z/2[\rho, \tau],$$
with $|\rho| = (-1,-1), \ |\tau| = (0,-1)$. The class $\rho$ is a map $$\rho: \1 \rightarrow \Gm = \Sigma^{1,1}\1,$$
that corresponds to $-1 \in \mathbb{R}^\times$ (see \cite{bachmannrealetale} for example),
and can be further viewed as a map 
$$\rho: \Sigma^{-1, -1} \1^\comp_2 \rightarrow \1^\comp_2.$$
We denote by $\1^\comp_2/\rho$ the cofiber of $\rho$.  

By Corollary~1.9 of \cite{BJ}, there is an equivalence between the category of cellular spectra over $\mathbb{C}$ and the category of cellular modules over $\1^\comp_2/\rho$ over $\mathbb{R}$, under which the $\mathbb{C}$-motivic sphere $\1^\comp_2$ corresponds to $\1^\comp_2/\rho$ over $\mathbb{R}$.

The class $\tau$ then corresponds to a map
$$\tau: \Sigma^{0,-1}\1^\comp_2/\rho \rightarrow \1^\comp_2/\rho.$$  
We denote by $\1^\comp_2/(\rho, \tau)$ the cofiber of $\tau$.

The $\mathbb{R}$-motivic $\pi_{{*, *}}\BPGL^\comp_2$ was computed by Hu--Kriz \cite{HuKriz} and Hill \cite{Hill}. 

\begin{proposition} \label{prop bpgl over R}
$$
\pi_{{*, *}}\BPGL^\comp_2 \cong \mathbb{Z}^\comp_{2}\left[\begin{array}{cccccc}
\rho, & \\	
v_0, & \tau^2 v_0, & \tau^4 v_0, & \tau^6 v_0, & \tau^8 v_0, & \cdots\\
v_1, & & \tau^4 v_1, & &\tau^8 v_1, & \cdots\\
v_2, & & & & \tau^8 v_2, & \cdots\\
\cdots & 
\end{array}\right]/
\left[\mkern-10mu\begin{array}{c}
v_0 = 2\\  \rho v_0 = 0 \\ \rho^3 v_1 = 0 \\  \rho^7 v_2 = 0 \\ \cdots  \end{array}\mkern-10mu\right]$$
and the generators satisfy the further relations	 
$$\tau^{2^{i+1}\cdot j} v_i \cdot \tau^{2^{k+1}\cdot l} v_k  = \tau^{2^{i+1}(j+2^{k-i}l)} v_i v_k$$
when $i \leq k$, as if the class $\tau$ were an element in this ring.

Here the bidegrees of the generators are 
$$|\tau| = (0,-1), \ |\rho| = (-1,-1), \ |v_n| = (2^{n+1}-2, 2^n -1).$$
\end{proposition}

For the Chow degrees, we have $\tau$ and $\rho$ in degrees 2 and 1, and we have all $v_n$'s in degree 0. Therefore, we can read off the Chow degree $n$ part of the $\mathbb{R}$-motivic $\pi_{{*, *}}\BPGL^\comp_2$. The first few are the following.\\

\begin{itemize}
\item {\bf Chow = 0}: $\mathbb{Z}^\comp_2[v_0, v_1, \cdots] = (\BP_{2*})^\comp_2$.\\
\item {\bf Chow = 1}: $\rho \cdot \mathbb{Z}^\comp_2[v_0, v_1, \cdots]/ \rho v_0 = \Sigma^{-1,-1}\BP_{2*}/2$.\\
\item {\bf Chow = 2}: $\rho^2 \cdot \mathbb{Z}^\comp_2[v_0, v_1, \cdots]/ \rho v_0 = \Sigma^{-2,-2}\BP_{2*}/2$.\\
\item {\bf Chow = 3}: $\rho^3 \cdot \mathbb{Z}^\comp_2[v_0, v_1, \cdots]/ (\rho v_0, \rho^3 v_1) = \Sigma^{-3,-3}\BP_{2*}/(2, v_1)$.\\
\item {\bf Chow = 4}: $\rho^4 (\BP_{2*})^\comp_2 \oplus \tau^2v_0(\BP_{2*})^\comp_2 = \Sigma^{-4,-4}\BP_{2*}/(2, v_1) \oplus \Sigma^{0,-2}(\BP_{2*})^\comp_2$.\\
\item {\bf Chow = 5}: $\rho^5 (\BP_{2*})^\comp_2  = \Sigma^{-5,-5}\BP_{2*}/(2, v_1) $.\\
\item {\bf Chow = 6}: $\rho^6 (\BP_{2*})^\comp_2  = \Sigma^{-6,-6}\BP_{2*}/(2, v_1) $.\\
\item {\bf Chow = 7}: $\rho^7 (\BP_{2*})^\comp_2  = \Sigma^{-7,-7}\BP_{2*}/(2, v_1, v_2) $.\\
\item {\bf Chow = 8}: 
$$\rho^8 (\BP_{2*})^\comp_2 \oplus \frac{\tau^4v_0(\BP_{2*})^\comp_2 + \tau^4v_1(\BP_{2*})^\comp_2}{\tau^4v_0\cdot v_1 - \tau^4v_1\cdot v_0}$$ 
$$ = \Sigma^{-8,-8}\BP_{2*}/(2, v_1, v_2) \oplus (\Sigma^{0,-4}(\BP_{2*})^\comp_2 + \Sigma^{2,-3}\BP_{2*}/2). $$
\end{itemize}
On the spectrum level, this corresponds to the fact that
$$\pi_{*,*}(\1^\comp_2)_{c=0} \cong \pi_{*,*} \1^\comp_2/(\rho, \tau).$$
In the category of $\BP_{2*}\textup{BP}$-comodules, we have short exact sequences

\begin{center}
\begin{tikzcd}
0 \arrow[r] & \BP_{2*}/I_n \arrow[r, "v_{n+1}"] & \BP_{2*}/I_n \arrow[r] & \BP_{2*}/I_{n+1} \arrow[r] & 0,
\end{tikzcd}
\end{center}
where $I_n = (2, v_1, \cdots, v_n)$. This implies that we could realize the comodule $\BP_{2*}/I_n$ inductively by the spectrum $\1^\comp_2/(\rho, \tau, 2, v_1, \cdots, v_n)$. These spectra are all $\scr E_\infty$.

In higher Chow degrees, these comodules are more complicated. However, Corollary~\ref{cor:wcellular_heart_ordinary_cellular} tells us they correspond to module spectra over $\1^\comp_2/(\rho, \tau)$. By Proposition~\ref{prop bpgl over R}, these comodules can be filtered such that the subquotients are shifts of BP$_{2*}/I_n$. Then by naturality, we can use the motivic Adams spectral sequences for shifts of the spectra $\1^\comp_2/(\rho, \tau, 2, v_1, \cdots, v_n)$ to study the one for $(\1^\comp_2)_{c=n}$.

For example in Chow degree 8, by Corollary~\ref{cor:wcellular_heart_ordinary_cellular}, the second summand of the comodule, denoted by $M_8$, corresponds to a module spectrum over $\1^\comp_2/(\rho, \tau)$, denoted by $X_8$. Then the short exact sequence of comodules
\begin{center}
\begin{tikzcd}
0 \arrow[r] & \Sigma^{0,-4}(\BP_{2*})^\comp_2 \arrow[r] & M_8 \arrow[r] & \Sigma^{2,-3}\BP_{2*}/2 \arrow[r] & 0
\end{tikzcd}
\end{center}
gives us a cofiber sequence of module spectra over $\1^\comp_2/(\rho, \tau)$
\begin{center}
\begin{tikzcd}
\Sigma^{0,-4}\1^\comp_2/(\rho, \tau) \arrow[r] & X_8 \arrow[r] & \Sigma^{2,-3}\1^\comp_2/(\rho, \tau, 2).
\end{tikzcd}
\end{center}
This cofiber sequence gives us maps of motivic Adams spectral sequences
\begin{center}
\begin{tikzcd}
\mathbf{motASS}(\Sigma^{0,-4}\1^\comp_2/(\rho, \tau)) \arrow[r] & \mathbf{motASS}(X_8) \arrow[r] & \mathbf{motASS}(\Sigma^{2,-3}\1^\comp_2/(\rho, \tau, 2)).
\end{tikzcd}
\end{center}
We anticipate that this would allow us to compute motivic Adams differentials for $X_8$ from the ones for $\Sigma^{0,-4}\1^\comp_2/(\rho, \tau)$ and $\Sigma^{2,-3}\1^\comp_2/(\rho, \tau, 2)$, which are purely algebraic.



This gives us a tower of motivic Adams spectral sequences. 

\begin{center}
\begin{tikzcd}
& \arrow[d] &\\
& \mathbf{motASS}((\1^\comp_2)_{c \geq 3}) \arrow[d] \arrow[r] & \mathbf{motASS}((\1^\comp_2)_{c=3}) \arrow[r, equal] & \shortstack{$\mathbf{algNSS}$\\$(\Sigma^{-3,-3}\textup{BP}_{2*}/(2,v_1))$}   \\
& \mathbf{motASS}((\1^\comp_2)_{c \geq 2}) \arrow[d] \arrow[r] & \mathbf{motASS}((\1^\comp_2)_{c=2}) \arrow[r, equal] & \shortstack{$\mathbf{algNSS}$\\$(\Sigma^{-2,-2}\textup{BP}_{2*}/2)$}   \\
& \mathbf{motASS}((\1^\comp_2)_{c \geq 1}) \arrow[d] \arrow[r] & \mathbf{motASS}((\1^\comp_2)_{c=1}) \arrow[r, equal] & \shortstack{$\mathbf{algNSS}$\\$(\Sigma^{-1,-1}\textup{BP}_{2*}/2)$}   \\
\mathbf{motASS}(\1^\comp_2) \arrow[r, equal] & \mathbf{motASS}(\1^\comp_2) \arrow[r] & \mathbf{motASS}((\1^\comp_2)_{c=0}) \arrow[r, equal] & \mathbf{algNSS}((\BP_{2*})^\comp_2)
\end{tikzcd}
\end{center}
Comparing to the case over $\mathbb{C}$, it is clear that over $\mathbb{R}$ different rows contain different algebraic information of motivic Adams differentials. We anticipate that this allows us to obtain motivic Adams differentials for $\1^\comp_2$ from the algebraic Novikov spectral sequences for various $\BP_{2*}\textup{BP}$-comodules.

\subsection{Over finite fields} \label{finite fields}

Next, we discuss the cases of finite fields $\mathbb{F}_q$, where $q$ is a power of an odd prime. We are going to concentrate on the computations at the prime 2.

Recall from \cite{Hoyois2017} and Theorem~8.2 of \cite{Kylling} (see also \cite{WilsonOst} and Appendix C) that we have 

\begin{itemize}
\item  $\pi_{{*,*}}\textup{H}\Z/2 = 
\begin{cases} 
\Z/2[\tau, u]/u^2,  \ Sq^1\tau=0 &\text{if } q \equiv 1 \ \textup{mod} \ 4, \\ 
\Z/2[\tau, \rho]/\rho^2, \ Sq^1\tau=\rho &\text{if } q \equiv 3 \ \textup{mod} \ 4. 
\end{cases} $
$$\text{where} \ |\tau|=(0,-1), \ |u|= |\rho|=(-1,-1).$$
\item $\pi_{{*,*}}\textup{H}\mathbb{Z}^\comp_2 = 
\begin{cases} \mathbb{Z}^\comp_2 &\text{if } ({*,*}) = (0,0), \\ (\mathbb{Z}/(q^w-1))^\comp_2 &\text{if } ({*,*}) = (-1,-w) \ \text{and} \ w\geq 1,\\
0 &\text{otherwise}.
 \end{cases} $

\item $\pi_{{*,*}}\BPGL^\comp_2= \pi_{{*,*}}\textup{H}\mathbb{Z}^\comp_2 \otimes \BP_{2*}$. 	
\end{itemize}
We can read off the Chow degree $n$ part of $\pi_{{*,*}}\textup{BPGL}^\comp_2$. In particular, it is concentrated in Chow degrees 0 and $2w-1$ for $w \geq 1$. In positive Chow degrees, we have
\begin{align*}
\BPGL_{*,*} (\1^\comp_2)_{c=2w} & = 0, \\
\BPGL_{*,*} (\1^\comp_2)_{c=2w-1} & = \Sigma^{-1,-w}\BP_{2*}/2^{\nu_2(q^w-1)},
\end{align*}
where $\nu_2(-)$ is the 2-adic valuation function. 

\begin{remark}
We remark that in positive chow degrees, most of these $\BP_{2*}$-modules do not have Levin index 1. For example in Chow degree 1, it is $\Sigma^{-1,-1}\BP_{2*}/2^{\nu_2(q-1)}$. In particular, when $q \equiv 5 \ \textup{mod} \ 8$, we have $\Sigma^{-1,-1}\BP_{2*}/4$, and when $q \equiv 1 \ \textup{mod} \ 8$, we have $\Sigma^{-1,-1}\BP_{2*}/8\cdot2^m$ for $m \geq 0$. Both of them do not have Levin index 1. In fact, one can show that $\BP_{2*}/2^{\nu_2(q^w-1)}$ has Levin index 1 if and only if $\nu_2(q^w-1) =1$.
\end{remark}

For every spectrum $(\1^\comp_2)_{c=n}$, by a change-of-ring isomorphism and the description of $\BPGL_{*,*} (\1^\comp_2)_{c=n}$, the universal coefficient spectral sequence collapses for degree reasons (see Propositions~7.7 and 7.10 of \cite{dugger2005motivic} and step $(4)$ of our strategy in \S\ref{strategy}). 
$$\textup{Tor}_{*,*}^{(\BP_{2*})^\comp_2}(\BPGL_{*,*} (\1^\comp_2)_{c=n}, \ \Z/2) \implies \textup{H}\Z/2_{*,*} (\1^\comp_2)_{c=n}.$$
We have that for $w \geq 1$,
\begin{align*}
\textup{H}\Z/2_{*,*} (\1^\comp_2)_{c=2w} & = 0, \\
\textup{H}\Z/2_{*,*} (\1^\comp_2)_{c=2w-1} & = \begin{cases} 
\Z/2\{\tau^w\} \oplus  \Z/2\{u \tau^{w-1}\}  &\text{if } q \equiv 1 \ \textup{mod} \ 4, \\ 
\Z/2\{\tau^w\} \oplus  \Z/2\{\rho \tau^{w-1}\}  &\text{if } q \equiv 3 \ \textup{mod} \ 4. 
\end{cases}
\end{align*}
Here $\tau^w, \ u \tau^{w-1}, \ \rho \tau^{w-1}$ denote the generators of the groups $\Z/2$, in bidegrees $(0,-w), \ (-1,-w), \ (-1,-w)$. Note that 
$$\textup{H}\Z/2_{*,*} (\1^\comp_2)_{c=0} = \textup{H}\Z/2\{1\},$$
generated by 1 in bidegree $(0,0)$. By Corollary~\ref{cor:wcellular_heart_ordinary_cellular}, these spectra $(\1^\comp_2)_{c=n}$ are cellular modules over $(\1^\comp_2)_{c=0}$. Therefore, we learn that for $w \geq 1$,
\begin{align*}
\pi_{*,*} (\1^\comp_2)_{c=2w} & = 0, \\
\pi_{*,*} (\1^\comp_2)_{c=2w-1} & = \pi_{*,*} \Sigma^{-1,-w}(\1^\comp_2)_{c=0}/2^{\nu_2(q^w-1)}.
\end{align*}
In particular, the spectrum $(\1^\comp_2)_{c=2w-1}$ has the same homotopy groups as 
a 2-cell complex over $(\1^\comp_2)_{c=0}$, with cells in bidegrees $(-1, -w)$ and $(0, -w)$, and attaching map $2^{\nu_2(q^w-1)}$, or equivalently $q^w-1$.

We emphasis that the spectrum $(\1^\comp_2)_{c=0}$ is very well understood independent of the base field (whose characteristics is not 2, which includes these finite fields $\F_q$). By Corollaries~\ref{cor:wcellular_heart_ordinary_cellular} and \ref{algNSS BP}, the motivic Adams spectral sequence of $(\1^\comp_2)_{c=0}$ is isomorphic to the algebraic Novikov spectral sequence for $\BP_{2*}$, which is again independent of the base field. Over $\C$, this spectral sequence is isomorphic to the motivic Adams spectral sequence of $\1^\comp_2/\tau$, and is well understood in a large range of dimensions \cite{IWX, IWX2}. For the spectrum $(\1^\comp_2)_{c=0}/2^{\nu_2(q^w-1)}$, its motivic Adams spectral sequence is isomorphic to the algebraic Novikov spectral sequence for $\BP_{2*}/2^{\nu_2(q^w-1)}$, which can also be understood through known computations over $\C$.

As an application of the Postnikov--Whitehead tower of the Chow $t$-structure, we compute all Adams differentials on powers of $\tau$, reproving Corollary~7.12 of \cite{WilsonOst}. By Leibniz's rule, we only need to compute all nonzero differentials on classes of the form $\tau^{2^n}$.

\begin{proposition} \label{diff tau power}
Let $m(n) = \nu_2(q^{2^{n}}-1)$, where $\nu_2(-)$ is the 2-adic valuation function.
\begin{enumerate}
\item When $q \equiv 1 \ \textup{mod} \ 4$, all nonzero Adams differentials on the classes $\tau^{2^n}$ are 
$$d_{m(n)}(\tau^{2^n}) = h_0^{m(n)} u \tau^{2^n-1}, \ \text{for} \ n \geq 0.$$

\item When $q \equiv 3 \ \textup{mod} \ 4$, all nonzero Adams differentials on the classes $\tau^{2^n}$ are 
$$d_{m(n)}(\tau^{2^n}) = h_0^{m(n)} \rho \tau^{2^n-1}, \ \text{for} \ n \geq 1.$$
\end{enumerate}
\end{proposition}

\begin{proof}
Suppose that $q \equiv 1 \ \textup{mod} \ 4$. For degree reasons, the only nonzero Adams differential on $\tau^{2^n}$ has the form
$$d_{m}(\tau^{2^n}) = h_0^{m} u \tau^{2^n-1},$$
for some integer $m$, depending on $n$.

Let $w = 2^{n}$. From the above discussion, we have
$$\textup{H}\Z/2_{*,*} (\1^\comp_2)_{c=2^{n+1}-1} = \Z/2\{\tau^{2^n}\} \oplus  \Z/2\{u \tau^{2^n-1}\}.$$
So the differentials at interest are present in the motivic Adams spectral sequence for $(\1^\comp_2)_{c=2^{n+1}-1}$, and the ones for $\1^\comp_2$ follows from naturality of the motivic Adams spectral sequences and the following zigzag in the Postnikov--Whitehead tower.

\begin{center}
\begin{tikzcd}
\shortstack{ $\mathbf{motASS}$\\$((\1^\comp_2)_{c \geq 2^{n+1}-1})$ } \arrow[d] \arrow[r] &
 \shortstack{ $\mathbf{motASS}$\\$((\1^\comp_2)_{c=2^{n+1}-1})$ } \arrow[r, equal] & \shortstack{$\mathbf{algNSS}$\\$(\BPGL_{*,*} (\1^\comp_2)_{c=2^{n+1}-1})$}   \\
 \mathbf{motASS}(\1^\comp_2)  &  & 
\end{tikzcd}
\end{center}
We have
$$\BPGL_{*,*} (\1^\comp_2)_{c=2^{n+1}-1} = \Sigma^{-1,-2^n}\BP_{2*}/2^{\nu_2(q^{2^n}-1)}.$$
Therefore, in the algebraic Novikov spectral sequence, we have 
$$d_{m(n)}(\tau^{2^n}) = h_0^{m(n)} u \tau^{2^n-1},$$
where $m(n) = \nu_2(q^{2^{n}-1})$. This completes the proof for the cases $q \equiv 1 \ \textup{mod} \ 4$.

For the cases $q \equiv 3 \ \textup{mod} \ 4$, note that $\tau^2$ is an indecomposable element on the motivic Adams $E_2$-page. The rest of the proof works in the exact same way.
\end{proof}

In \cite{WilsonOst}, Wilson--\O stv\ae r computed $\pi_{n, 0} {\1}^\comp_2$ for $0 \leq n \leq 20$ when $q \equiv 1 \ \textup{mod} \ 4$ and for $0 \leq n \leq 18$ when $q \equiv 3 \ \textup{mod} \ 4$. The difference of ranges is due to the computation of an Adams $d_2$-differential from the 20-stem to the 19-stem in the cases $q \equiv 1 \ \textup{mod} \ 4$, while the answer for an Adams $d_2$-differential is unknown in the cases $q \equiv 3 \ \textup{mod} \ 4$. This $d_2$-differential is also the hardest one in the computation up to the 20-stem.
In fact, Wilson--\O stv\ae r proved

\begin{proposition} [Proposition~7.16 and Remark~7.19 in \cite{WilsonOst}] \label{prop:WO diff}

Over $\mathbb{F}_q$, where $q$ is a power of an odd prime, in the motivic Adams spectral sequence for ${\1}^\comp_2$, we have
\begin{enumerate}
\item $d_2(\tau g) = u h_0^2g$, when $q \equiv 5 \ \textup{mod} \ 8$,
\item $d_2(\tau g) = 0$, when $q \equiv 1 \ \textup{mod} \ 8$,
\item $d_2(\tau^2 g) = \rho\tau h_0^2g$ or $0$, when $q \equiv 3 \ \textup{mod} \ 4$.
\end{enumerate}
\end{proposition}

As an application of Proposition~\ref{diff tau power} and the Chow $t$-structure method, we give much simpler proofs of the Adams $d_2$-differentials in the cases $(1)$ and $(2)$ in Proposition~\ref{prop:WO diff}. We also compute the unknown Adams $d_2$-differential in the cases $q \equiv 3 \ \textup{mod} \ 4$ in Proposition~\ref{prop:WO diff}. The intuitions of our proofs are explained in Remarks~\ref{remark 1} and \ref{remark 2}.

\begin{proof}[Proof of Proposition~\ref{prop:WO diff} (1) and (2)]
In both cases $q \equiv 1, \ 5 \ \textup{mod} \ 8$, as discussed in \cite{WilsonOst}, for degree reasons, the only possibilities are
$$d_2(\tau g) = u h_0^2g \ \text{or} \ 0.$$
On the motivic Adams $E_2$-page, we have
\begin{align*}
\tau g \cdot d_0 & = \tau e_0^2 \neq 0,\\
u h_0^2g \cdot d_0 & = u h_0^2 e_0^2 \neq 0.
\end{align*}
See \cite[Proposition~7.1]{WilsonOst} for the multiplicative structure of the Adams $E_2$-page.

By Proposition~\ref{diff tau power}~(1) for the case $n=0$, we have
$$d_2(\tau) = \begin{cases} 
0  &\text{if } q \equiv 1 \ \textup{mod} \ 8, \\ 
u h_0^2  &\text{if } q \equiv 5 \ \textup{mod} \ 8. 
\end{cases}
$$
By Leibniz's rule, we have $d_2(e_0^2) = 0$. Thus,
$$d_2(\tau \cdot e_0^2) = d_2(\tau) \cdot e_0^2 = \begin{cases} 
0  &\text{if } q \equiv 1 \ \textup{mod} \ 8, \\ 
u h_0^2 e_0^2  &\text{if } q \equiv 5 \ \textup{mod} \ 8. 
\end{cases}
$$
The element $d_0$ is a permanent cycle for degree reasons. So we have $d_2(\tau g \cdot d_0) = d_2(\tau g) \cdot d_0$, and
$$d_2(\tau g) = \begin{cases} 
0  &\text{if } q \equiv 1 \ \textup{mod} \ 8, \\ 
u h_0^2 g  &\text{if } q \equiv 5 \ \textup{mod} \ 8. 
\end{cases}
$$

\end{proof}

\begin{remark} \label{remark 1}
There are two reasons that direct proofs of Proposition~\ref{prop:WO diff} (1) and (2) are harder to come up with. 
\begin{itemize}
\item The element $\tau g$ is indecomposable on the motivic Adams $E_2$-page.
\item On the $E_2$-pages, the element $\tau g$ is detected by $(\1^\comp_2)_{c=0}$, while its potential target $u h_0^2g$ is detected by $(\1^\comp_2)_{c=1}$. They live in different Chow degrees.
\end{itemize}
After multiplication by $d_0$, both $\tau g \cdot d_0 = \tau e_0^2$ and $u h_0^2g \cdot d_0 = u h_0^2 e_0^2$ are detected by $(\1^\comp_2)_{c=1}$ on the motivic Adams $E_2$-pages, so this $d_2$-differential can be obtained through the algebraic Novikov spectral sequence, something purely algebraic.
\end{remark}

\begin{proposition}
In the motivic Adams spectral sequence for ${\1}^\comp_2$ over $\mathbb{F}_q$, , where $q$ is a power of an odd prime, and $q \equiv 3 \ \textup{mod} \ 4$, we have
$$d_2(\tau^2 g) = 0.$$
\end{proposition}

\begin{proof}
For degree reasons, the only possibilities are
$$d_2(\tau^2 g) = \rho\tau h_0^2g \ \text{or} \     0.$$
On the motivic Adams $E_2$-page, we have
\begin{align*}
\tau^2 g \cdot d_0 & = \tau^2 e_0^2 \neq 0,\\
\rho\tau h_0^2g \cdot d_0 & = \rho\tau h_0^2 e_0^2 \neq 0.
\end{align*}
See \cite[Proposition~7.1]{WilsonOst} for the multiplicative structure of the Adams $E_2$-page.

By Proposition~\ref{diff tau power}~(2) for the case $n=1$, we have
$$d_2(\tau^2) = 0.$$
(In fact, we have $d_3(\tau^2) = \rho\tau h_0^3$ when $q \equiv 3 \ \textup{mod} \ 8$, and $\tau^2$ supports nonzero $d_4$ or higher when $q \equiv 7 \ \textup{mod} \ 8$.)

By Leibniz's rule, we have
$$d_2(\tau^2 g) \cdot  d_0 = d_2(\tau^2 g \cdot d_0) = d_2 (\tau^2 \cdot e_0^2) = d_2 (\tau^2) \cdot e_0^2 = 0.$$
The element $d_0$ is a permanent cycle for degree reasons. Since $\rho\tau h_0^2g \cdot d_0 \neq 0$, we have
$$d_2(\tau^2 g) = 0.$$
\end{proof}

\begin{remark} \label{remark 2}
Similarly to Remark~\ref{remark 1}, on the $E_2$-pages, the element $\tau^2 g$ is detected by $(\1^\comp_2)_{c=1}$, while its potential target $\rho\tau h_0^2g$ is detected by $(\1^\comp_2)_{c=3}$. They live in different Chow degrees. After multiplication by $d_0$, they are both detected by $(\1^\comp_2)_{c=3}$, and the $d_2$-differential can be obtained through the algebraic Novikov spectral sequence.
\end{remark}

\begin{remark}
We comment that by Leibniz's rule, $d_2(\tau^2 g) = 0$ in all cases for $q$.
\end{remark}

\begin{appendices}
\section{A cell structure for $\MGL$} \label{app:cells}
We record the fact that the Thom spectrum $\MGL$ admits a very well-behaved cell structure.
We suspect that this is well-known, but do not know a reference.

Write $\SH(S)^{\pure\tate\ge d}$ for the subcategory generated under filtered colimits, extensions and sums by $\Th_S(\scr O^n) \simeq S^{2n,n}$ with $n \ge d$.
The result we need is as follows.
\begin{theorem} \label{thm:mgl-cells}
We have $\MGL \in \SH(S)^{\pure\tate\ge 0}$.
The same is true for the motivic Thom spectra $\mathrm{MSL}, \mathrm{MSp}$.
\end{theorem}

We need a few preparations for the proof.
\begin{lemma} \label{lemm:thom-cof}
Let $X \in \Sm_S$, $U \subset X$ open with smooth complement $Z \subset X$.
Let $V$ be a vector bundle on $X$.
Then there is a cofibration sequence \[ \Th(V|_U) \to \Th(V) \to \Th(V|_Z \oplus N_{Z/X}) \in \Spc(S)_*. \]
\end{lemma}
\begin{proof}
We have \[ \Th(V)/\Th(V|_U) = \frac{V/V \setminus 0}{V|_U / V|_U \setminus 0} \wequi \frac{V}{V \setminus 0 \cup V|_U} = V/V \setminus Z, \] and that latter space is equivalent to $\Th(N_{Z/V})$ by purity.
The short exact sequence (see e.g. \cite[Tag 0690]{stacks-project}) \[ 0 \to N_{Z/X} \to N_{Z/V} \to N_{X/V} \wequi V \to 0 \] implies the desired result by \cite[(3.26)]{hoyois-equivariant}.\NB{stably this is clear since the Thom spectrum functor splits exact sequences (e.g. because it factors through $K$-theory), but unstably this is not so obvious. (we only need the stable statement anyway...)}
\end{proof}

Denote by $\Spc(S)_*^{\pure\tate\ge d}$ the smallest full subcategory of $\Spc(S)_*$ which (1) contains $*$ and $S^{2n,n}$ for $n \ge d$, (2) is closed under filtered colimits, and (3) if $A \to B \to C$ is a \emph{cofiber} sequence with $A, C \in \Spc(S)_*^{\pure\tate\ge d}$ then $B \in \Spc(S)_*^{\pure\tate\ge d}$.
\begin{theorem}[Wendt] \label{thm:wendt-cells}
Let $S$ be the spectrum of a field, $G$ a split reductive group, $X = G/P$ a homogeneous space and $V$ a vector bundle on $X$ of rank $d$.
Then $\Th(V) \in \Spc(S)_*^{\pure\tate\ge d}$.
\end{theorem}
\begin{proof}
We use the ideas of \cite[Proposition 2.2]{wendt2010more}.
The Bruhat decomposition provides a filtration \[ Z_0 = \emptyset \subset Z_1 \subset \dots \subset Z_n = X \] by closed subschemes, such that $Z_i \setminus Z_{i-1}$ is a finite disjoint union of affine spaces (see \cite[Proof of Proposition 3.7]{wendt2010more}).
We obtain via Lemma \ref{lemm:thom-cof} cofibration sequences \[ \Th(V|_{X \setminus Z_i}) \to \Th(V|_{X \setminus Z_{i-1}}) \to \Th(N_{(Z_{i} \setminus Z_{i-1})/(X\setminus Z_{i-1})} \oplus V|_{Z_i \setminus Z_{i-1}}). \]
For $i = n$ we have $X \setminus Z_i = \emptyset$ and hence the left hand Thom space
is contractible; by induction on $i$ it thus suffices to show that if $E$ is a vector bundle on $Z_i \setminus Z_{i-1}$ then $\Th(E) \in \Spc(S)_*^{\pure\tate\ge d}$.
But $Z_i \setminus Z_{i-1}$ is a disjoint union of affine spaces, so $N \oplus V|_{Z_i \setminus Z_{i-1}}$ is trivial by Quillen--Suslin \cite{quillen1976projective}, and the result follows.
\end{proof}
\begin{remark} \label{rmk:wendt-cells}
The assumption that $S$ is the spectrum of a field was used twice in the above proof: (1) to know that there is a Bruhat decomposition of $G/P$ into affine space cells, and (2) to know that vector bundles on affine spaces are trivial.
For usual, symplectic and special linear Grassmannians, (1) holds over $\Z$ (and hence any base).\NB{reference?}
(2) will hold over any base $S$ such that the 
 Bass--Quillen conjecture holds (i.e. any vector bundle on $\A^n_S$ is extended from $S$) and such that all vector bundles on $S$ are trivial, e.g. $S=Spec(\Z)$ (by Lindel's theorem, see e.g. \cite[Theorem 5.2.1]{asok2015affine-I}).
In fact, in order to establish Theorem \ref{thm:mgl-cells} we only need the stable analog.
Since Thom spectra of vector bundles only depend on their classes in homotopy $K$-theory \cite[Remark 16.11]{bachmann-norms}, this holds over any base.
\end{remark}

\begin{proof}[Proof of Theorem \ref{thm:mgl-cells}]
Since the spectra $\MGL, \mathrm{MSL}, \mathrm{MSp}$ are stable under base change, we may assume that $S=Spec(\Z)$.
Each of these spectra is obtained as \[ \colim_i \Th(V_i \ominus \scr O^{rk(V_i)}), \] for suitable Grassmannians (usual, special linear or symplectic) $X_i$ and vector bundles $V_i$ on them.
We have \begin{align*} \Th(V_i \ominus \scr O^{rk(V_i)}) &\wequi \Sigma^\infty \Th(V_i) \wedge S^{-2rk(V_i), -rk(V_i)} \\ & \in \Sigma^\infty \Spc(S)_*^{\pure\tate\ge rk(V_i)} \wedge S^{-2rk(V_i), -rk(V_i)} \\ & \subset \SH(S)^{\pure\tate\ge 0}; \end{align*} here we have used Remark \ref{rmk:wendt-cells} for the first containment.
This concludes the proof.
\end{proof}

\section{A vanishing result for algebraic cobordism} \label{app:vanishing}
We are confident that the following result is well-known, but we could not locate a reference.

\begin{theorem} \label{thm:MGL-vanishing}
Let $X$ be essentially smooth over a semi-local PID and let $S \subset \N$ denote the set of positive residue characteristics of $X$.
Then for $i>0$ we have\NB{Can we get this without inverting $S$? Over more general bases? Remove smoothness by expressing in terms of Borel-Moore homology?} \[ \MGL^{2*+i,*}(X)[1/S] = 0. \]
\end{theorem}
\begin{proof}
The slice tower for $\MGL[1/S]$ of the form \[ \MGL[1/S] = f_0(\MGL)[1/S] \leftarrow f_1(\MGL)[1/S] \leftarrow \dots \] has increasing connectivity and subquotients given by $\Sigma^{2*,*}\textup{H}\Z[1/S] \otimes L_{2*}$ (where $L_{2t}$ is the degree $2t$ part of the Lazard ring).
This follows by combining \cite[proof of Theorem 4.7]{spitzweck2010relations}, \cite[Theorem 11.3]{spitzweck2012commutative} and \cite[Proposition 3.7]{schmidt2018stable}.
Induction up to tower thus reduces to proving the same result for motivic cohomology.
By construction \[ H^{p,q}(X, \Z) = \mathrm{CH}^q(X, 2q-p) \] and hence it is enough to show that $\mathrm{CH}^*(X, -i) = 0$ for $i>0$.
This follows from the fact that $\mathrm{CH}^q(X, -i) = H_{-i}(z^q(X))$, where $z^q(X)$ is a complex concentrated in non-negative degrees \cite[middle of p.3]{levine2001techniques} (this is where we use that the base is semi-local).
\end{proof}

\section{Motivic cohomology}
The motivic cohomology of a point with finite coefficients can be computed as a consequence of the Beilinson--Lichtenbaum and Bloch--Kato conjectures.
These were established by Voevodsky, Rost and others \cite{voevodsky2003motivic,voevodsky2011motivic}.
For a textbook treatment, see \cite{haesemeyer2019norm}.
We collect here some of their results.

\begin{theorem} \label{thm:BL-BK}
Let $k$ be a field of exponential characteristic $e$ and $n > 0$ coprime to $e$.
\begin{enumerate}
\item The cycle class map induces an equivalence $L_\et \Z/n(d) \wequi L_\et \mu_n^{\otimes d}$, as complexes of sheaves on $\Sm_k$.
\item The induced map $\Z/n(d) \to L_\et \mu_n^{\otimes d}$ induces an equivalence $\Z/n(d) \wequi \tau_{\ge -d} L_\et \mu_n^{\otimes d}$, the truncation as complexes of Nisnevich sheaves on $\Sm_k$.
\item There is a canonical isomorphism $H^m_\et(k, \mu_n^{\otimes m}) \wequi K_m^M(k)/n$.
\end{enumerate}
\end{theorem}
\begin{proof}
(1) is well-known, see e.g. \cite[Theorem 10.3]{lecture-notes-mot-cohom}.
(2) in the case $n=\ell$ is \cite[Theorem C]{haesemeyer2019norm}.
The general case follows from this since $\Z/n$ is a sum of iterated extensions of $\Z/\ell$ (for various $\ell$).
(3) By (2) we have $H^m_\et(k, \mu_n^{\otimes m}) \wequi H^{m,m}(k, \Z/m)$.
This is the same as $H^{m,m}(k, \Z)/n$ by \cite[Theorem 3.6]{lecture-notes-mot-cohom}, which is $K_m^M(k)/n$ by \cite[Theorem 5.1]{lecture-notes-mot-cohom}.
\end{proof}

\begin{corollary}
Assumptions as above.
\begin{enumerate}
\item We have \[ \pi_{p,q}(\textup{H}\Z/n) \wequi H^{-p}_\et(k, \mu_n^{\otimes -q}) \text{ for $0 \ge p \ge q$, and $\wequi 0$ else}. \]
\item Suppose that $k$ contains a primitive $n$-th root of unity.
  Then \[ \pi_{*,*}(\textup{H}\Z/n) \wequi H^*_\et(k, \Z/n)[\tau] \wequi K_*^M(k)/n[\tau]. \]
  Here $H^*_\et(k, \Z/n) \wequi K_*^M(k)/n$ is placed in bidegree $(-*,-*)$, and $|\tau| = (0, -1)$.
\end{enumerate}
\end{corollary}
\begin{proof}
In (2) we have $\mu_n^{\otimes r} \wequi \Z/n$ for all $r$.
The element $\tau$ corresponds to \[ 1 \in H^0_\et(k, \mu_n^{\otimes 1}) \wequi H^0_\et(k, \Z/n) \wequi \Z/n. \]
Both (1) and (2) are now a direct consequences of Theorem \ref{thm:BL-BK}.
\end{proof}
\end{appendices}

\bibliographystyle{alpha}
\bibliography{bibliography}

\end{document}